\documentclass[12pt]{article} 
\usepackage{amsthm,amsfonts,amsmath,amssymb,latexsym,epsfig,graphics,color} 
\usepackage[matrix,arrow,curve]{xy} 
\title{Fredholm theory and transversality for the \\parametrized and
for the $S^1$-invariant\\symplectic action} 
 
\author{Fr\'ed\'eric {\sc Bourgeois}, \ Alexandru \sc{Oancea} 
           \\ \quad \\ 
         {\it \small Universit\'e Libre de Bruxelles, B-1050 Bruxelles, 
Belgium} \\ 
         {\it \small IRMA, Universit\'e de Strasbourg, F-67084
           Strasbourg, France}  
} 
 
\date{22 September 2009}

\newtheorem{PARA}{}[section] 
\newtheorem{theorem}[PARA]{Theorem} 
 
\newtheorem{lemma}[PARA]{Lemma} 
\newtheorem{proposition}[PARA]{Proposition} 
\newtheorem{definition}[PARA]{Definition} 
 
\theoremstyle{definition} 
\newtheorem{remark}[PARA]{Remark} 
\newtheorem{example}[PARA]{Example} 
\numberwithin{equation}{section} 
\newcommand{\para}{\begin{PARA}\rm} 
\newcommand{\arap}{\end{PARA}\rm} 
\newcommand{\dfn}{\begin{definition}\rm} 
\newcommand{\nfd}{\end{definition}\rm} 
\newcommand{\rmk}{\begin{remark}\rm} 
\newcommand{\kmr}{\end{remark}\rm} 
\newcommand{\xmpl}{\begin{example}\rm} 
\newcommand{\lpmx}{\end{example}\rm} 
\newcommand{\cA}{\mathcal{A}} 
\newcommand{\cB}{\mathcal{B}}

\newcommand{\cE}{\mathcal{E}} 
 
\newcommand{\cH}{\mathcal{H}} 
\newcommand{\cI}{\mathcal{I}} 
\newcommand{\cJ}{\mathcal{J}} 
 
\newcommand{\cL}{\mathcal{L}} 
\newcommand{\cM}{\mathcal{M}}

\newcommand{\cP}{\mathcal{P}}

\newcommand{\cU}{\mathcal{U}} 
 
\newcommand{\cV}{\mathcal{V}} 
\newcommand{\cW}{\mathcal{W}}

\newcommand{\oD}{{\overline{D}}} 
\newcommand{\uD}{{\underline{D}}} 
 
\newcommand{\of}{{\overline{f}}} 
\newcommand{\uf}{{\underline{f}}} 
\newcommand{\og}{{\overline{\gamma}}} 
\newcommand{\ug}{{\underline{\gamma}}}

\newcommand{\oK}{{\overline{K}}} 
\newcommand{\uK}{{\underline{K}}} 
\newcommand{\olambda}{{\overline{\lambda}}} 
\newcommand{\ulambda}{{\underline{\lambda}}} 
\newcommand{\op}{{\overline{p}}} 
\newcommand{\up}{{\underline{p}}}
\newcommand{\os}{{\overline{s}}}

\newcommand{\ou}{\overline{u}} 
\newcommand{\uu}{\underline{u}}

\newcommand{\ox}{\overline{x}} 
\newcommand{\ux}{\underline{x}}

\newcommand{\otheta}{\overline{\theta}} 

\newcommand{\one} 
{{{\mathchoice \mathrm{ 1\mskip-4mu l} \mathrm{ 1\mskip-4mu l} 
\mathrm{ 1\mskip-4.5mu l} \mathrm{ 1\mskip-5mu l}}}} 

\newcommand{\C}{{\mathbb{C}}}

\newcommand{\R}{{\mathbb{R}}}

\newcommand{\Z}{{\mathbb{Z}}} 
\newcommand{\im}{\mathrm{ im }}        
 
\renewcommand{\Re}{\mathrm{ Re\,}}       

\newcommand{\reg}{{\mathrm{reg}}} 
\newcommand{\eps}{{\varepsilon}} 
\newcommand{\om}{{\omega}}

\renewcommand{\th}{{\widetilde{h}}} 
\newcommand{\tH}{{\widetilde{H}}} 
\newcommand{\tJ}{{\widetilde{J}}} 
\newcommand{\tlambda}{{\widetilde{\lambda}}}

\newcommand{\tR}{{\widetilde{R}}}

\newcommand{\tU}{{\widetilde{U}}} 
\newcommand{\tV}{{\widetilde{V}}}

\newcommand{\tY}{{\widetilde{Y}}} 
\def\NABLA#1{{\mathop{\nabla\kern-.5ex\lower1ex\hbox{$#1$}}}} 
\def\Nabla#1{\nabla\kern-.5ex{}_{#1}} 
\def\Tabla#1{\Tilde\nabla\kern-.5ex{}_{#1}} 
\renewcommand{\Tilde}{\widetilde}

\newcommand{\p}{{\partial}}

\begin{document} 
 
\maketitle 
 

\begin{abstract} 
We study the parametrized Hamiltonian action functional for finite-dimensional families of Hamiltonians. We show that the linearized operator for the $L^2$-gradient lines is Fredholm and surjective, for a generic choice of Hamiltonian and almost complex structure. We also establish the Fredholm property and transversality for generic $S^1$-invariant families of Hamiltonians and almost complex structures, parametrized by odd-dimensional spheres. This is a foundational result used to define $S^1$-equivariant Floer homology. As an intermediate result of independent interest, we generalize Aronszajn's unique continuation theorem to a class of elliptic integro-differential inequalities of order two.
\end{abstract}

\tableofcontents 
 
 
\section{Introduction} 
 
\noindent \emph{\sc Motivation~I.} 
Hamiltonian Floer homology is commonly referred to as Morse homology
for the symplectic action functional on the free loop space of a
symplectic manifold. One of the most
important features of the free loop space is that it carries an
$S^1$-action by reparametrization at the source 
$$
(\tau\cdot \gamma)(\theta):=\gamma(\theta-\tau), \qquad \tau\in S^1,
\ \gamma:S^1\to W,
$$
where $(W,\omega)$ is the target symplectic manifold. It was
realized at an early stage of the theory that Floer 
homology should admit an $S^1$-equivariant version. In the last
paragraph of the foundational article~\cite{FHS}, Floer, Hofer, and
Salamon explicitly set the goal of constructing it.  

Such an $S^1$-equivariant
theory was first defined by Viterbo~\cite{V}, in the context of symplectic
homology. Viterbo's paper contains a wealth of structural properties
with rich applications, but it does not give any kind of technical details
for the definition. The present paper grew out of
our efforts to understand $S^1$-equivariant Floer
homology and put it on firm grounds.  

The topological motivation of the definition is the following. Let $X$ be a topological
space endowed with an $S^1$-action, and  
$ES^1$ be a contractible space on which $S^1$ acts freely. 
 The \emph{Borel construction of $X$}, denoted  $X_{S^1}$, is defined
to be the quotient of $X\times ES^1$ by the free diagonal action. 
The \emph{$S^1$-equivariant homology of $X$} is defined to be 
$$
H_*^{S^1}(X):=H_*(X_{S^1}). 
$$
Taking as a model for $ES^1$ the inductive limit
${\displaystyle\lim_\to} \ S^{2N+1}$ of
the unit spheres $S^{2N+1}\subset \C^{N+1}$, one sees that
$X_{S^1}={\displaystyle\lim_\to} \ 
X\times_{S^1} S^{2N+1}$. Moreover, we have 
$$
H_*^{S^1}(X)=\lim_\to H_*(X\times_{S^1} S^{2N+1}). 
$$

Assume now that $X$ is a finite dimensional manifold. Morse
theory on the finite dimensional approximation $X\times_{S^1}
S^{2N+1}$ of the Borel construction is the same as $S^1$-invariant 
Morse theory on $X\times S^{2N+1}$. 
Viterbo's idea is to define $S^1$-equivariant Floer homology as the
direct limit of $S^1$-invariant Floer homology groups for
$S^1$-invariant action functionals defined on $C^\infty(S^1,W)\times
S^{2N+1}$. The latter space carries the diagonal $S^1$-action 
$$
\tau\cdot(\gamma,\lambda)\mapsto
(\gamma(\cdot-\tau),\tau\cdot\lambda).
$$ 

\medskip 

\noindent \emph{\sc The equation.} Let $H:S^1\times W\times S^{2N+1}\to
\R$, $H=H(\theta,x,\lambda)$ be a smooth function, which we view as an
$S^{2N+1}$-family of Hamiltonians $H_\lambda:S^1\times W\to \R$. Let
$J^\theta_\lambda$, $\theta\in S^1$, $\lambda\in S^{2N+1}$ be an
$S^{2N+1}$-family of time-dependent almost complex structures which
are compatible with $\omega$. Let $g$ be a Riemannian metric on
$S^{2N+1}$. The \emph{parametrized Floer equation} for a pair of maps
$u:\R\times S^1 \to W$ and $\lambda:\R\to S^{2N+1}$ is the integro-differential system 
\begin{eqnarray}  
\label{eq:MAIN1} 
 \p_s u + J_{\lambda(s)}^\theta (\p_\theta u - 
X_{H_{\lambda(s)}}^\theta (u)) & = & 0, \\  
\label{eq:MAIN2} 
 \dot \lambda (s) - \int_{S^1} \vec \nabla_\lambda 
H(\theta,u(s,\theta),\lambda(s)) d\theta & = & 0,   
\end{eqnarray} 
subject to the asymptotic conditions
\begin{equation} \label{eq:MAINasymptotic} 
 \lim_{s\to -\infty} (u(s,\cdot),\lambda(s)) = (\og,\olambda), \quad  
 \lim_{s\to +\infty} (u(s,\cdot),\lambda(s)) = (\ug,\ulambda), 
\end{equation} 
where $(\og,\olambda)$, $(\ug,\ulambda)$ are elements of 
\begin{equation} \label{eq:PH}
 \cP(H):=\big\{(\gamma,\lambda) \, : \, 
   \dot\gamma-X_{H_\lambda}(\gamma)  =  0,
 \quad \int_{S^1} \frac {\p H} {\p \lambda} (\theta,\gamma(\theta),\lambda)\, 
 d\theta  =  0 \big\}. 
\end{equation} 
Here and in the sequel we use the notation $\vec \nabla$ for a 
gradient vector field, whereas $\nabla$ will denote a 
covariant derivative. Our convention for the Hamiltonian vector field
is that $\omega(X_H,\cdot)=dH$. 

The Fredholm and transversality analysis contained in this
paper apply to any symplectic manifold and any component of the free
loop space of $W$. However, in order to
interpret~(\ref{eq:MAIN1}--\ref{eq:MAIN2}) as a (negative) gradient equation, it is convenient to
restrict to the component $C^\infty_{\mathrm{contr}}(S^1,W)$ of
contractible loops and to assume 
that $(W,\omega)$ is symplectically aspherical, i.e. $\langle
[\omega],\pi_2(W)\rangle =0$. The 
equations~(\ref{eq:MAIN1}--\ref{eq:MAIN2}) are in this case the
negative gradient equations 
of the \emph{parametrized action functional} 
$$
\cA : C^\infty_{\mathrm{contr}}(S^1,\widehat W)\times S^{2N+1}\to \R,
$$
defined by  
\begin{equation} \label{eq:INTROA} 
\cA(\gamma,\lambda)  := -\int_{D^2}
\overline\gamma^*  
\om - \int_{S^1} H_\lambda(\theta,\gamma(\theta)) d\theta.
\end{equation} 
Here $\overline\gamma:D^2\to W$ is a smooth extension of $\gamma$ to
the disc. The metric on $C^\infty_{\mathrm{contr}}(S^1,W)\times
S^{2N+1}$ is the product of the ($\lambda$\,-\,dependent) $L^2$-metric determined by
$(J^\theta_\lambda)_{\theta\in S^1}$ with the metric $g$. The elements
of $\cP(H)$ are the critical points of $\cA$. 

\medskip

\noindent \emph{\sc $S^1$-invariance.} Let us now assume that $H$ and $J$
are $S^1$-invariant with respect to the diagonal $S^1$-action on
$S^1\times S^{2N+1}$, meaning that 
\begin{equation} \label{eq:INV} 
H_{\tau\lambda}(\theta+\tau,\cdot)=H_\lambda(\theta,\cdot), \qquad 
J_{\tau\lambda}^{\theta+\tau}=J_\lambda^\theta
\end{equation} 
for all $\theta\in S^1$, $\tau\in S^1$, $\lambda\in S^{2N+1}$. Let us
also assume that the metric $g$ on $S^{2N+1}$ is $S^1$-invariant. Then
equations~(\ref{eq:MAIN1}--\ref{eq:MAIN2}) are invariant under the
diagonal $S^1$-action on $C^\infty_{\mathrm{contr}}(S^1,W)\times
S^{2N+1}$. 

Equation~\eqref{eq:MAINasymptotic} is not, since $S^1$ acts
freely on the asymptotes $\op=(\og,\olambda)$,
$\up=(\ug,\ulambda)$. To fix this, we introduce the $S^1$-orbits
$$
S_p:=S^1\cdot p, \qquad p\in\cP(H), 
$$
and the condition 
\begin{equation} \label{eq:MAINasymptotic-bis} 
 \lim_{s\to -\infty} (u(s,\cdot),\lambda(s)) \in S_\op, \quad  
 \lim_{s\to +\infty} (u(s,\cdot),\lambda(s)) \in S_\up. 
\end{equation} 

The results of Sections~\ref{sec:S1inv} and~\ref{sec:transvS1} are
summarized in the following statement. 

\medskip 

\noindent {\bf Theorem~A.} {\it 
\renewcommand{\theenumi}{(\alph{enumi})}
\begin{enumerate} 
\item \label{item:INTROa} For a generic choice of the $S^1$-invariant Hamiltonian $H$, and 
  for any choice of $S^1$-invariant $(J,g)$, the operator which
  linearizes~(\ref{eq:MAIN1}--\ref{eq:MAIN2}) is Fredholm between
  Sobolev spaces with suitable exponential weights.
\item \label{item:INTROb} There exists an explicit class consisting of
  $S^1$-invariant triples $(H,J,g)$ with $H$ as above such that, for a generic choice of
  $(H,J,g)$ inside this class, the Fredholm operator which
  linearizes~(\ref{eq:MAIN1}--\ref{eq:MAIN2}) is surjective for all
  solutions of~(\ref{eq:MAIN1}--\ref{eq:MAINasymptotic}) and all
  $\op,\up\in\cP(H)$.
\end{enumerate} 
}

Part~\ref{item:INTROb} is proved as Theorem~\ref{thm:transvS1}. As a
 matter of fact, that theorem is more precise. It states that we can
 achieve transversality within a special class of almost 
 complex structures (called \emph{adapted},
 Definition~\ref{defi:adaptedJ}), after possibly perturbing a
 Hamiltonian which is either generic (in the sense that it belongs to
 the class $\cH_{\mathrm{gen}}$ defined in Section~\ref{sec:transvS1}), or
 split (in the sense that it belongs to the class
 $\cH_{\mathrm{split}}$, loc. cit.). In the case of split Hamiltonians, our proof
 works under the assumption that $W$ is symplectically aspherical. For
 generic Hamiltonians, this assumption is not used.   

\medskip 

The Hamiltonians satisfying~\ref{item:INTROa} are those for which the
Hessian of $\cA$ at a critical point is degenerate only along the
infinitesimal generator of the $S^1$-action. As a consequence
of~\ref{item:INTROb}, for a generic choice of $(H,J,g)$ inside the
given class, the spaces of trajectories 
$$
\widehat \cM(\op,\up;H,J,g):=\{(u,\lambda) \mbox{
  solving~(\ref{eq:MAIN1}, \ref{eq:MAIN2}, \ref{eq:MAINasymptotic})} \, \}  
$$
and 
$$
\widehat \cM(S_\op,S_\up;H,J,g):=\{(u,\lambda) \mbox{
  solving~(\ref{eq:MAIN1}, \ref{eq:MAIN2}, \ref{eq:MAINasymptotic-bis})} \, \}  
$$
are smooth manifolds, for all $\op,\up\in\cP(H)$. Viterbo's definition
of $S^1$-equivariant Floer homology relies on counting modulo the
$S^1$-action the elements of the moduli spaces 
$$
\cM(S_\op,S_\up;H,J,g):= \widehat \cM(S_\op,S_\up;H,J,g)/\R.
$$

\medskip 

\noindent \emph{\sc The parameter space.}
In part~\ref{item:INTROa} of the above theorem we
need to consider the linearized operator acting between weighted
Sobolev spaces. This is necessary since, for a generic choice of
the $S^1$-invariant Hamiltonian $H$, the 
elements of $\cP(H)$ come in Morse-Bott nondegenerate families of
dimension $1$ given by the free $S^1$-action.
In order to prove the Fredholm property, one first has to 
establish it for operators of the same form and having nondegenerate
asymptotics. This corresponds to considering the linearization of
equations~(\ref{eq:MAIN1}--\ref{eq:MAIN2}) for a generic and \emph{non}-invariant $H$.

The point is that equations~(\ref{eq:MAIN1}--\ref{eq:MAINasymptotic})
and the action functional~\eqref{eq:INTROA} still make sense if one
replaces the parameter space $S^{2N+1}$ by some arbitrary manifold
$\Lambda$, and so do the spaces of trajectories $\widehat
\cM(\op,\up;H,J,g)$. 

We summarize the results of Sections~\ref{sec:param}
and~\ref{sec:param-transv} in the following statement. 

\medskip 

\noindent {\bf Theorem~B.} {\it Let $\Lambda$ be an arbitrary finite
  dimensional parameter space. 
\renewcommand{\theenumi}{(\alph{enumi})}
\begin{enumerate} 
\item \label{item:INTROa-notS1} For a generic choice of $H$ and 
  for any choice of $(J,g)$, the operator which
  linearizes~(\ref{eq:MAIN1}--\ref{eq:MAIN2}) is Fredholm between 
  suitable Sobolev spaces.
\item \label{item:INTROb-notS1} For a generic choice of the triple
  $(H,J,g)$, the Fredholm operator which
  linearizes~(\ref{eq:MAIN1}--\ref{eq:MAIN2}) is surjective for all
  solutions of~(\ref{eq:MAIN1}--\ref{eq:MAINasymptotic}) and all
  $\op,\up\in\cP(H)$.
\end{enumerate} 
}

\medskip

The Hamiltonians satisfying~\ref{item:INTROa-notS1} are those for which the
Hessian of $\cA$ at a critical point is nondegenerate. As a consequence
of~\ref{item:INTROb-notS1}, for a generic choice of $(H,J,g)$
the moduli spaces of parametrized Floer trajectories 
$$
\cM(\op,\up;H,J,g):= \widehat \cM(\op,\up;H,J,g)/\R
$$
are smooth manifolds, for all $\op,\up\in\cP(H)$. We use these moduli spaces
in~\cite{BO3}  
to define parametrized symplectic homology groups and establish
a Gysin long exact sequence for symplectic homology. 

\medskip 

\noindent \emph{\sc Motivation~II.} Our initial motivation was the desire
to interpret the long exact sequence in~\cite{BOcont} as a Gysin exact
sequence. To this effect, we prove in~\cite{BO4} that, given 
an aspherical
symplectic manifold $W$ with contact type boundary $M=\p W$, 
the (positive
part of) the $S^1$-equivariant symplectic homology of $W$ is isomorphic to
the linearized contact homology of $M$, provided the latter is
well-defined. Via this isomorphism, the long exact sequence
of~\cite{BOcont} is isomorphic to the Gysin exact sequence
of~\cite{BO4}. 

However, we believe that the present paper has ramifications
going well-beyond $S^1$-equivariant symplectic homology.  
\begin{itemize} 
\item \emph{Transversality in linearized contact homology.} The second
  author is currently developing with Cieliebak a version of
  ``non-equivariant'' contact homology~\cite{CO}. The Borel
  construction can be applied to it in order to define an invariant
  which is isomorphic to linearized
  contact homology. The results of the present paper will be
  instrumental to prove that
  transversality can be achieved for this theory, modulo having it for
  finite energy holomorphic planes or, alternatively, modulo the data
  of a linearization for the contact complex. Transversality can
  currently be achieved for linearized 
  contact homology only for homotopy classes of loops which contain
  only simple Reeb orbits.    
\item \emph{Lagrange multiplier problems.}
  Equations~(\ref{eq:MAIN1}--\ref{eq:MAIN2}) can be viewed as a Floer
  type Lagrange multiplier problem. To prove unique
  continuation for this integro-differential system, we were led to prove a
  generalization of Aronzsajn's theorem for 
  integro-differential inequalities (see below). This is relevant for any
  Floer-type problem involving an additional parameter space. Examples
  are Rabinowitz-Floer homology~\cite{CF}, or $G$-equivariant
  Floer homology~\cite{OW}.  
\item \emph{Floer homology for families.} Our methods can be extended
  in order to define parametrized Floer homology groups for a
  symplectic fibration. We expect these to coincide with the target of
  the Hutchings spectral sequence~\cite{Hu}. 
\item \emph{Relation to Givental's point of view.} Given a closed symplectic manifold $X$, Givental defined in~\cite{Gi} a $D$-module structure on 
$H^*(X;\C)\otimes \Lambda_{Nov}\otimes \C[\hbar]$, where $\Lambda_{Nov}$ is a suitable Novikov ring and $\hbar$ is the generator of $H^*(BS^1)$. He interprets this 
as being the $S^1$-equivariant Floer cohomology of $X$. Our construction of $S^1$-equivariant Floer homology in~\cite{BO3} provides an interpretation of the underlying homology group as the homology of a Floer-type complex. We expect that the $D$-module structure can also be defined within our setup.

\end{itemize} 

\medskip 

\noindent \emph{\sc Aronszajn's theorem.} We prove in
Section~\ref{sec:unique} the following 
unique continuation result for solutions of integro-differential
inequalities, as Theorem~\ref{thm:inequality}. This generalizes a
celebrated theorem of Aronszajn~\cite{Ar}. It allows one to prove
unique continuation for solutions of the
system~(\ref{eq:MAIN1}--\ref{eq:MAIN2}). 

\medskip 

\noindent {\bf Theorem~C.} 
{\it Let $h>0$ and denote $Z_h:=]-h,h[\times S^1$. Assume $u\in C^\infty(Z_h,\C^n)$ satisfies 
\begin{equation*} 
|\Delta u(s,\theta)|^2\le M\Big[
|u(s,\theta)|^2+ |\nabla u(s,\theta)|^2 + \int_{S^1}|u(s,\tau)|^2 \,
d\tau \Big]
\end{equation*}
for all $(s,\theta)\in Z_h$, where $M>0$ is a positive constant. If
$u$ vanishes together with all its derivatives on $\{0\}\times S^1$,
then $u\equiv 0$ on $Z_h$. 
}

\medskip 

\noindent \emph{\sc Non-compact setup.} We use the setup of
symplectic homology, since this was our initial motivation. 
The consequences are merely cosmetic, and the adaptation
to the setup of closed manifolds is straightforward.

\medskip 

\noindent \emph{\sc Structure of the paper.} In~\S\ref{sec:param} we
prove part~\ref{item:INTROa-notS1} of Theorem~B as
Proposition~\ref{prop:Hreg} and
Theorem~\ref{thm:Fredholm}. In~\S\ref{sec:unique} we prove several
results on unique continuation, and in particular Theorem~C as
Theorem~\ref{thm:inequality}. In~\S\ref{sec:param-transv} we prove
part~\ref{item:INTROb-notS1} of Theorem~B as
Theorem~\ref{thm:Jreg}. In~\S\ref{sec:S1inv} we prove
part~\ref{item:INTROa} of Theorem~A as
Propositions~\ref{prop:genericHS1} and~\ref{prop:indexMB}. In~\S\ref{sec:uniqueS1} we prove a unique continuation result needed for the $S^1$-invariant theory. 
Finally, in~\S\ref{sec:transvS1} we prove part~\ref{item:INTROb} of Theorem~A
as Theorem~\ref{thm:transvS1}.

\medskip 

\noindent \emph{\sc Acknowledgements.} We thank Luc Robbiano
for having read our proof of
Proposition~\ref{prop:semilocalCarleman}, and for having suggested an
alternative one. 

\smallskip

\noindent The authors were partially supported by the Minist\`ere Belge
des Affaires \'etrang\`eres and the Minist\`ere
Fran\c{c}ais des Affaires \'etrang\`eres et europ\'eennes, through the
programme PHC--Tournesol Fran\c{c}ais. F.B. was partially supported by the Fonds National de la Recherche Scientifique (Belgium). Both authors were partially
supported by ANR project ``Floer Power'' ANR-08-BLAN-0291-03 (France).


\section{Fredholm theory for the parametrized Floer equation} \label{sec:param}  

In this section, we prove part~\ref{item:INTROa-notS1} of Theorem~B. The setup is that of 
symplectic homology. Our ambient symplectic manifold, denoted $(\widehat
W,\widehat \omega)$, is the symplectic completion of a compact
symplectic manifold $(W,\omega)$ with contact type boundary.
This means that there exists a vector 
field $X$ defined in a neighbourhood of $\p W$, transverse and pointing 
outwards along $\p W$, such that $\cL _X \om = \om$. The $1$-form
$\alpha:=(\iota_X\om)|_{\p W}$ is a contact form, and the flow of
$X$ determines a symplectic trivialization of a neighbourhood of $\p
W$ as $([-\delta,0] \times \p W, d(e^t\alpha))$. The symplectic
completion is 
$$
\widehat W=W\cup_{\p W} [0,\infty[\,\times \, \p W. 
$$
Moreover, we assume that $\widehat W$ (or, equivalently, $W$) is
symplectically aspherical, i.e. $\langle \widehat
\omega,\pi_2(\widehat W)\rangle =0$. The Reeb
vector field $R_\alpha$ on $M:=\p W$ is defined by the conditions $\ker \, \om|_M =
\langle R_\alpha \rangle$ and $\alpha(R_\alpha)=1$. The contact
distribution on $M$ is defined by $\xi=\ker\,\alpha$. 
Finally, we define the {\bf action spectrum} of $(M,\alpha)$ by
$$ 
\textrm{Spec}(M,\alpha) := \{ T \in \R^+\, | \, \textrm{ there is a 
   closed } R_\alpha\textrm{-orbit of period } T\}. 
$$ 

Let $\Lambda$ denote a finite dimensional closed 
manifold of dimension $m$, which we call ``parameter space''. The 
elements of $\Lambda$ are denoted by $\lambda$. 

\medskip 
 
We define the set $\cH_\Lambda$ of 
{\bf admissible Hamiltonian families} to consist of elements  
$H\in C^\infty(S^1\times \widehat W\times \Lambda,\R)$ which satisfy 
the following conditions: 
\begin{itemize}  
\item $H<0$ on $S^1\times W\times \Lambda$;  
\item there exists $t_0\ge 0$ such that $H(\theta,p,t,\lambda)=\beta 
  e^t +\beta'(\lambda)$ for $t\ge t_0$, with $0<\beta\notin 
  \mathrm{Spec}(M,\alpha)$ and $\beta'\in C^\infty(\Lambda,\R)$. 
\end{itemize}  
 
Let $H:S^1 \times \widehat W \times \Lambda \to 
\R$ be an admissible Hamiltonian family denoted by 
$H(\theta,x,\lambda)=H_\lambda(\theta,x)$. The differential of the
corresponding action functional $\cA$ defined by~\eqref{eq:INTROA} is given by  
\begin{equation} \label{eq:dA} 
d\cA(\gamma,\lambda) \cdot (\zeta,\ell)= 
\int_{S^1}\om(\dot\gamma(\theta)-X_{H_\lambda}(\gamma(\theta)),\zeta(\theta)) 
d\theta 
- 
\int_{S^1} \frac {\p H} {\p \lambda} (\theta,\gamma(\theta),\lambda) d\theta 
\cdot \ell 
\end{equation}  
and therefore $(\gamma,\lambda)$ is a critical point of $\cA$ if and 
only if  
\begin{equation} \label{eq:periodicpar} 
 \gamma\in\cP(H_\lambda) \quad \mbox{and} \quad  
 \int_{S^1} \frac {\p H} {\p \lambda} (\theta,\gamma(\theta),\lambda)\, 
d\theta =0.  
\end{equation}  
In~\eqref{eq:PH} we denoted the set of critical points of $\cA$ by
$\cP(H)$. 
 
\begin{remark} {\rm  
 Equation~\eqref{eq:periodicpar} can be interpreted as follows. Every 
loop $\gamma:S^1\to\widehat W$ determines a function  
\begin{equation} \label{eq:Fgamma} 
F_\gamma:\Lambda \to \R, \qquad \lambda \mapsto \int_{S^1} 
H(\theta,\gamma(\theta),\lambda) \, d\theta.  
\end{equation}  
A pair $(\gamma,\lambda)$ belongs therefore to $\cP(H)$ if and only if 
$$ 
\gamma\in \cP(H_\lambda) \quad \mbox{ and } \quad \lambda\in 
\textrm{Crit}(F_\gamma). 
$$  
} 
\end{remark}  
 
Let $J=(J_\lambda^\theta)$, $\lambda\in\Lambda$, $\theta\in S^1$ be a 
family of $\theta$-dependent compatible  
almost complex structures on $\widehat W$ 
which, at infinity, are invariant under translations in the 
$t$-variable and satisfy the relations  
\begin{equation} \label{eq:standardJ} 
J_\lambda^\theta \xi=\xi, \qquad J_\lambda^\theta (\frac \partial 
{\partial t}) =R_\alpha. 
\end{equation} 
Such an {\bf admissible family of almost complex 
  structures} $J$ induces a  
  family of $L^2$-metrics on the space $C^\infty(S^1,\widehat W)$, parametrized 
  by $\Lambda$ and defined by   
$$ 
\langle \zeta,\eta\rangle_\lambda := \int_{S^1} 
\om(\zeta(\theta),J_\lambda^\theta\eta(\theta)) d\theta, \quad \zeta,\eta\in 
T_\gamma C^\infty(S^1,\widehat 
W)=\Gamma(\gamma^*T\widehat W). 
$$ 
Such a metric can be coupled with any metric $g$ on $\Lambda$ and 
gives rise to a metric on  
$C^\infty(S^1,\widehat W)\times \Lambda$ acting  
at a point $(\gamma,\lambda)$ by  
$$ 
\langle(\zeta,\ell), (\eta,k)\rangle_{J,g}:= \langle \zeta,\eta\rangle_\lambda + 
g(\ell,k), \qquad (\zeta,\ell),(\eta,k)\in \Gamma(\gamma^*T\widehat 
W)\oplus T_\lambda\Lambda.  
$$ 
We denote by $\cJ_\Lambda$ the set of pairs $(J,g)$ consisting of an 
admissible almost complex structure $J$ on $\widehat W$ and of a 
Riemannian metric $g$ on $\Lambda$. The {\bf parametrized Floer
  equations}~(\ref{eq:MAIN1}--\ref{eq:MAIN2}) are the gradient
equation for $\cA$ with respect to such a metric 
$\langle\cdot,\cdot\rangle_{J,g}$. For the reader's convenience, we
rewrite them:
\begin{eqnarray}  
\label{eq:Floer1par} 
 \p_s u + J_{\lambda(s)}^\theta (\p_\theta u - 
X_{H_{\lambda(s)}}^\theta (u)) & = & 0, \\  
\label{eq:Floer2par} 
 \dot \lambda (s) - \int_{S^1} \vec \nabla_\lambda 
H(\theta,u(s,\theta),\lambda(s)) d\theta & = & 0,   
\end{eqnarray} 
and, for $(\og,\olambda),(\ug,\ulambda)\in \cP(H)$, 
\begin{equation} \label{eq:asymptoticpar} 
 \lim_{s\to -\infty} (u(s,\cdot),\lambda(s)) = (\og,\olambda), \quad  
 \lim_{s\to +\infty} (u(s,\cdot),\lambda(s)) = (\ug,\ulambda). 
\end{equation}

\begin{remark}{\rm Equation~\eqref{eq:Floer2par} is equivalent to  
\begin{equation} \label{eq:Floer2parbis} 
 \dot \lambda(s) - \vec \nabla F_{u(s,\cdot)}(\lambda(s))=0, 
\end{equation}  
where $F_{u(s,\cdot)}$ is defined by~\eqref{eq:Fgamma}. Thus, the 
parametrized Floer equation is a system involving a Floer equation and 
a finite-dimensional gradient equation.  
} 
\end{remark}  
 
 
Let us fix $p\ge 2$. The linearization of the 
equations~(\ref{eq:Floer1par}-\ref{eq:Floer2par}) gives rise to the  
operator  
$$ 
D_{(u,\lambda)} : W^{1,p}(u^*T\widehat W) \oplus 
W^{1,p}(\lambda^*T\Lambda) \to  
L^p(u^*T\widehat W) \oplus L^p(\lambda^*T\Lambda), 
$$ 
$$ 
D_{(u,\lambda)} (\zeta,\ell) :=  
\left(\begin{array}{c}  
D_u\zeta + (D_\lambda J\cdot \ell)(\p_\theta u - X_{H_\lambda}(u)) - 
J_\lambda (D_\lambda X_{H_\lambda}\cdot \ell) \\ 
\nabla_s \ell - \nabla_\ell \int_{S^1} \vec \nabla_\lambda H 
(\theta,u,\lambda) d\theta  
- \int_{S^1} \nabla_\zeta \vec \nabla_\lambda H(\theta,u,\lambda) d\theta 
\end{array}\right), 
$$ 
where 
$$ 
D_u : W^{1,p}(u^*T\widehat W) \to L^p(u^*T\widehat W) 
$$ 
is the usual Floer operator given by  
$$ 
D_u\zeta := \nabla_s \zeta + J_\lambda \nabla_\theta \zeta - 
J_\lambda \nabla_\zeta X_{H_\lambda} + \nabla_\zeta J_\lambda (\p_\theta u - 
X_{H_\lambda}). 
$$ 
 
The Hessian of $\cA$ at a critical point $p=(\gamma,\lambda)$ is given 
by the formula  
\begin{eqnarray} \label{eq:d2A} 
\lefteqn{d^2\cA(\gamma,\lambda)\big((\zeta,\ell),(\eta,k)\big)} \\ 
&=& \int_{S^1} \omega(\nabla_\theta\eta-\nabla_\eta 
X_{H_\lambda},\zeta) d\theta - \int_{S^1}\eta(\frac {\partial H} 
{\partial \lambda}\cdot \ell) d\theta \nonumber \\ 
&& - \ \int_{S^1} k(dH_\lambda\cdot \zeta) d\theta - \int_{S^1} \frac 
{\partial^2 H} {\partial \lambda^2}(\ell,k) d\theta \nonumber \\ 
&=& d^2\cA_{H_\lambda}(\gamma)(\zeta,\eta) - \int_{S^1}\eta(\frac {\partial H} 
{\partial \lambda}\cdot \ell) d\theta  
- \int_{S^1} k(dH_\lambda\cdot \zeta) d\theta  
- d^2 F_\gamma(\lambda)(\ell,k). \nonumber 
\end{eqnarray}  
 
We define the asymptotic operator at a critical point 
$(\gamma,\lambda)$ by 
$$ 
D_{(\gamma,\lambda)} : H^1(S^1,\gamma^*T\widehat W) \times T_\lambda 
\Lambda \to L^2(S^1,\gamma^*T\widehat W) \times T_\lambda 
\Lambda, 
$$ 
\begin{equation}  \label{eq:Dasy} 
D_{(\gamma,\lambda)}(\zeta,\ell) = \left(\begin{array}{c}  
J_\lambda(\nabla_\theta \zeta - \nabla_\zeta X_{H_\lambda} - 
(D_\lambda X_{H_\lambda})\cdot \ell) \\  
-\int_{S^1} \nabla_\zeta \frac {\partial H} {\partial \lambda} d\theta 
- \int_{S^1} \nabla_\ell \frac {\partial H} {\partial \lambda} d\theta 
\end{array}\right). 
\end{equation} 
Note that $D_{(\gamma,\lambda)}$ is obtained from $D_{(u,\lambda)}$ 
for $(u(s,\theta),\lambda(s))\equiv (\gamma(\theta),\lambda)$ and 
$(\zeta(s,\theta),\ell(s)) \equiv (\zeta(\theta),\ell)$.  
 
\begin{lemma} \label{lem:d2-asy} 
 The Hessian $d^2\cA(\gamma,\lambda)$ has trivial kernel if and only if 
 the asymptotic operator $D_{(\gamma,\lambda)}$ is injective.  
\end{lemma}  
 
\begin{proof} 
The conclusion of the Lemma follows readily from the identity  
$$ 
d^2\cA(\gamma,\lambda)\big((\zeta,\ell),(\eta,k)\big)=\langle 
D_{(\gamma,\lambda)}(\zeta,\ell),(\eta,k)\rangle.  
$$ 
\end{proof}  
 
We say that a critical point $(\gamma,\lambda)$ is {\bf nondegenerate} if 
the Hessian $d^2\cA(\gamma,\lambda)$ has trivial kernel. Since the 
operator $D_{(\gamma,\lambda)}$ is self-adjoint, this is equivalent to 
its surjectivity by Lemma~\ref{lem:d2-asy}.  
 

An admissible Hamiltonian family $H$ is called {\bf nondegenerate} if $\cP(H)$ 
consists of nondegenerate elements.   
We denote the set of nondegenerate and admissible Hamiltonian families by 
$\cH_{\Lambda,\textrm{reg}}\subset \cH_\Lambda$.  
 
\begin{proposition} \label{prop:Hreg} 
 The set $\cH_{\Lambda,\textrm{reg}}$ is of the second Baire category in 
$\cH_\Lambda$. Moreover, if $H\in\cH_{\Lambda,\textrm{reg}}$ the set 
$\cP(H)$ is discrete.   
\end{proposition}  
 
\begin{proof} 
Given an integer $r\ge 2$, we denote by $\cH^r_\Lambda$ the set of functions 
$H:S^1\times \widehat W\times\Lambda\to\R$ of class $C^r$ which satisfy the 
defining conditions for an admissible Hamiltonian family.
This is a Banach manifold with respect to the $C^r$-norm. As a matter 
of fact, it is an open subset of the Banach space 
of $C^r$-functions 
$h:S^1\times \widehat W\times \Lambda\to\R$ which, outside a compact 
set, have the form $\beta e^t +\beta'(\lambda)$ with 
$\beta\in\R$ and $\beta':\Lambda\to\R$ of class $C^r$. As such, the tangent space  
$T_H\cH^r_\Lambda$ is identified with this Banach space.  
We denote by $\cH^r_{\Lambda,\textrm{reg}}\subset \cH^r_\Lambda$ the 
set of Hamiltonians $H$ such that $\cP(H)$ consists of nondegenerate 
elements as defined above. For $t_0\ge 0$, we denote $\{t\le 
t_0\}:=W\cup M\times [0,t_0]$ and  
$\cH^r_{\Lambda,\textrm{reg},t_0}\subset\cH^r_\Lambda$ the set of 
Hamiltonians $H$ such that the elements $(\gamma,\lambda)\in\cP(H)$ with 
$\mathrm{im}(\gamma)\subset \{t\le t_0\}$ are nondegenerate. Then  
$$ 
\cH^r_{\Lambda,\textrm{reg}}=\bigcap_{t_0\ge 0} \cH^r_{\Lambda,\textrm{reg},t_0}. 
$$ 
 
Our first claim is that each $\cH^r_{\Lambda,\textrm{reg},t_0}$ is open and 
dense in $\cH^r_\Lambda$, so that $\cH^r_{\Lambda,\textrm{reg}}$ is of 
the second Baire category. To prove that 
$\cH^r_{\Lambda,\textrm{reg},t_0}$ is dense, we consider the  
Banach bundle $\cE\to \cH^r_\Lambda \times C^r(S^1,\{t\le t_0\})\times 
\Lambda$ whose fiber at $(H,\gamma,\lambda)$ is 
$\cE_{(H,\gamma,\lambda)}:=C^{r-1}(S^1,\gamma^*T\widehat W)\times 
T_\lambda\Lambda$, and the section $f$ given by  
$$ 
f(H,\gamma,\lambda):= (\dot \gamma - X_H\circ \gamma,-\int_{S^1} \vec 
\nabla _\lambda H).  
$$ 
The main step is to prove that $\cP:=f^{-1}(0)$ is a Banach 
submanifold of $\cH^r_\Lambda \times C^r(S^1,\{t\le t_0\})\times 
\Lambda$. Indeed, the vertical differential of $f$ at a point 
$(H,\gamma,\lambda)\in\cP$ is given by  
$$ 
df(H,\gamma,\lambda)\cdot(h,\zeta,\ell)= 
\left(\begin{array}{c} 
\nabla_\theta\zeta-\nabla_\zeta X_H - (D_\lambda X_H)\cdot\ell-X_h \\ 
-\int_{S^1}\nabla_\zeta \vec\nabla_\lambda H - \int_{S^1} \nabla_\ell 
\vec \nabla_\lambda H - \int_{S^1}\vec \nabla _\lambda h 
\end{array}\right), 
$$ 
where $h\in T_H\cH^r_\Lambda$ and $X_h$ is its Hamiltonian vector 
field. That $df(H,\gamma,\lambda)$ is surjective is seen as 
follows. Given $k\in T_\lambda\Lambda$, we have 
$(0,k)=df(H,\gamma,\lambda)\cdot(h,0,0)$, with 
$h(\cdot,\cdot,\lambda)=\mathrm{ct.}$ in some neighbourhood of 
$\mathrm{im}(\gamma)$ and $\vec\nabla _\lambda h=k$. Given $\eta\in 
C^{r-1}(S^1,\gamma^*T\widehat W)$, we have 
$(\eta,0)=df(H,\gamma,\lambda)\cdot(h,0,0)$, with $h$ independent of 
$\lambda$ and such that $X_h=-\eta$ along $\gamma$. This proves that 
$\cP$ is a Banach submanifold as desired. Since 
$\cH^r_{\Lambda,\textrm{reg},t_0}$ coincides with the set of regular 
values of the natural projection $\cP\to \cH^r_\Lambda$, we conclude 
by the Sard-Smale theorem that it is dense.  
 
To prove that $\cH^r_{\Lambda,\textrm{reg},t_0}$ is open in 
$\cH^r_\Lambda$, we prove that its complement is closed. Let therefore 
$H^\nu\in \cH^r_\Lambda \setminus \cH^r_{\Lambda,\textrm{reg},t_0}$ be a 
sequence such that $H^\nu\to H\in\cH^r_\Lambda$ as $\nu\to\infty$. Let 
$(\gamma^\nu,\lambda^\nu)\in\cP(H^\nu)$ be such that 
$D_{(\gamma^\nu,\lambda^\nu)}$ is not surjective and 
$\mathrm{im}(\gamma^\nu)\subset \{t\le t_0\}$. Since $\Lambda$ is 
compact, it follows from the Arzel\`a-Ascoli theorem that, up to  
a subsequence, $(\gamma^\nu,\lambda^\nu)$ converges to some 
$(\gamma,\lambda)\in\cP(H)$, with $\mathrm{im}(\gamma)\subset \{t\le 
t_0\}$. Since the sequence $D_{(\gamma^\nu,\lambda^\nu)}$  
converges to $D_{(\gamma,\lambda)}$, the latter cannot be surjective, 
so that $H\in \cH^r_\Lambda \setminus 
\cH^r_{\Lambda,\textrm{reg},t_0}$ as desired.  
 
Let $\cH_{\Lambda,\textrm{reg},t_0}:=\cap_{r\ge 2} 
\cH^r_{\Lambda,\textrm{reg},t_0}\subset \cH_\Lambda$. The same 
argument as above shows that $\cH_{\Lambda,\textrm{reg},t_0}$ is 
open. We claim that it is also dense, so that 
$\cH_{\Lambda,\textrm{reg}}=\cap_{t_0\ge 0} 
\cH_{\Lambda,\textrm{reg},t_0}$ is of the second Baire category in 
$\cH_\Lambda$. To see this, let $H\in \cH_\Lambda$ be fixed and 
consider a sequence $H^r\in \cH^r_{\Lambda,\textrm{reg},t_0}$ such 
that $H^r\to H$ in any fixed norm $C^{r_0}$, i.e. in the 
$C^\infty$-topology as $r\to\infty$. Since 
$\cH^r_{\Lambda,\textrm{reg},t_0}$ is open in $\cH^r_\Lambda$ and 
$\cH_\Lambda$ is dense in $\cH^r_\Lambda$, there exists $\widetilde 
H^r\in \cH^r_{\Lambda,\textrm{reg},t_0}\cap 
\cH_\Lambda=\cH_{\Lambda,\textrm{reg},t_0}$ such that 
$\|H^r-\widetilde H^r\|_{C^r}\le \eps_r$, with $\eps_r\to 0$ as $r\to 
\infty$. Then $\widetilde H^r\to H$ in the $C^\infty$-topology, which 
shows that $\cH_{\Lambda,\textrm{reg},t_0}\subset \cH_\Lambda$ is 
dense.  
 
We are left to prove that, given $H\in\cH_{\Lambda,\textrm{reg}}$, the 
elements of $\cP(H)$ are isolated. This follows from the nondegeneracy 
of the Hessian $d^2\cA$, as can be easily seen using a Taylor 
expansion at first order for $d\cA$.   
\end{proof}  
 
Let $I\subset \R$ be any interval. We denote 
\begin{eqnarray*} 
\cW^{1,p}(I) & := & W^{1,p}(I\times S^1,u^*T\widehat W) \oplus 
W^{1,p}(I,\lambda^*T\Lambda), \\  
\cL^p(I) & := & L^p(I\times S^1,u^*T\widehat W) \oplus 
L^p(I,\lambda^*T\Lambda), 
\end{eqnarray*} 
and we abbreviate $\cW^{1,p}:=\cW^{1,p}(\R)$, $\cL^p:=\cL^p(\R)$.  
 
Given $(\og,\olambda),(\ug,\ulambda)\in\cP(H)$ and $(u,\lambda)\in 
\widehat \cM((\og,\olambda),(\ug,\ulambda); H,J,g)$, we denote 
$D:=D_{(u,\lambda)}$. We can choose a unitary trivialization of 
$u^*T\widehat W$  
and a trivialization of $\lambda^*T\Lambda$ in which $D$ has 
the form   
\begin{equation} \label{eq:Dtriv} 
D\left(\begin{array}{c} 
\zeta \\ \ell  
\end{array}\right) 
:= 
\Bigg[\left(\begin{array}{cc} 
\partial_s +J_0\partial_\theta & 0 \\ 0 & \frac d {ds}  
\end{array}\right)  
+ N 
\Bigg] 
\left(\begin{array}{c} 
\zeta \\ \ell  
\end{array}\right),  
\end{equation} 
with $N:\R\times S^1\to \mathrm{Mat}_{2n+m}(\R)$ pointwise 
bounded and $\lim_{s\to\pm\infty}N(s,\theta)$  
symmetric.

\begin{theorem} \label{thm:Fredholm}  
Assume $(\og,\olambda),(\ug,\ulambda)\in\cP(H)$ are 
  nondegenerate. For any $(u,\lambda)\in 
\widehat \cM((\og,\olambda),(\ug,\ulambda); H,J,g)$ the operator 
$$ 
D:= D_{(u,\lambda)} : \cW^{1,p}\to \cL^p 
$$  
is Fredholm for $1<p<\infty$.    
\end{theorem} 
 
\begin{remark} 
The nonlinear theory only requires the case $p>2$, so that our 
$W^{1,p}$-maps to $\widehat W\times \Lambda$ are continuous.  
\end{remark} 
 
\begin{remark}[Structure of the proof]  
There are two main ingredients in the proof of 
Theorem~\ref{thm:Fredholm}. The first is that $D$ is an elliptic 
operator, so that it satisfies the estimates in 
Lemma~\ref{lem:elliptic} below. The second ingredient is that the 
constant operators at the asymptotes are bijective, due to our 
standing nondegeneracy assumption. This is proved in 
Lemma~\ref{lem:constant} below, and allows to refine the elliptic 
estimate by introducing a compact operator 
(Lemma~\ref{lem:Fredholm}).  
\end{remark}  
 
\begin{proof} 
By Lemma~\ref{lem:Fredholm} below, the operator  
$D$ satisfies an estimate of the form   
\begin{equation} \label{eq:Fredholm} 
\|x\|_{\cW^{1,p}}\le C\big(\|Dx\|_{\cL^p} + \|Kx\|_{\cL^p([-T,T])}\big),  
\end{equation} 
where $K:\cW^{1,p}\to \cL^p([-T,T])$ is the restriction operator and $T>0$ is 
large enough. The embedding   
$W^{1,p}\hookrightarrow C^0$ with $p>2$ is compact if the domain is 
bounded and has dimension at most $2$, so that $K$ is a compact 
operator. By~\cite[Lemma~A.1.1]{McDS} it follows that $D$ has a finite 
dimensional kernel and a closed image.  
 
To show that $D$ has a finite dimensional cokernel, we introduce 
its formal adjoint $D^*:\cW^{1,q}\to \cL^q$, $1/p+1/q=1$ defined by  
\begin{equation} \label{eq:adjoint} 
D^*\left(\begin{array}{c} 
\zeta \\ \ell  
\end{array}\right) 
:= 
\Bigg[\left(\begin{array}{cc} 
-\partial_s +J_0\partial_\theta & 0 \\ 0 & -\frac d {ds}  
\end{array}\right)  
+ N^T 
\Bigg] 
\left(\begin{array}{c} 
\zeta \\ \ell  
\end{array}\right),  
\end{equation}  
where $N^T$ denotes the transpose of $N$. Lemma~\ref{lem:Fredholm} 
applies also to the operator $D^*$, which therefore satisfies an 
estimate of the form   
\begin{equation} \label{eq:Fredholm*} 
\|x\|_{\cW^{1,q}}\le C\big(\|D^*x\|_{\cL^q} + \|Kx\|_{\cL^q([-T,T])}\big),  
\end{equation} 
with $K:\cW^{1,q}\to \cL^q([-T,T])$ the restriction operator and $T>0$ is large 
enough. The embedding $W^{1,q}\hookrightarrow L^q$ is compact for a 
bounded domain of dimension at most $2$, so that $K$ is compact and we 
infer that $D^*$ has a finite dimensional kernel.  
 
Given an element $y\in 
\cL^q$ which annihilates the image of $D$, we have $D^*y=0$. On the 
other hand, by 
elliptic regularity for $D^*$, we have $y\in \cW^{1,q}$. The cokernel of $D$ 
therefore coincides with the kernel of $D^*$ and is finite 
dimensional. This proves the Fredholm property for $D$.   
\end{proof} 
 
\begin{lemma} \label{lem:elliptic} 
 Under the hypotheses of Theorem~\ref{thm:Fredholm}, and for $p>1$, 
 there exists a constant $C>0$ such that, for any $k\in \Z$ and $x\in 
 \cW^{1,p}([k-1,k+2])$, we have  
\begin{equation} \label{eq:elliptic-loc} 
\|x\|_{\cW^{1,p}([k,k+1])}\le C\big(\|Dx\|_{\cL^p([k-1,k+2])} + 
\|x\|_{\cL^p([k-1,k+2])}\big).   
\end{equation} 
Similarly, there exists a constant $C_1>0$ such that, for $x\in 
\cW^{1,p}$, we have  
\begin{equation} \label{eq:elliptic} 
\|x\|_{\cW^{1,p}}\le C_1\big(\|Dx\|_{\cL^p} + 
\|x\|_{\cL^p}\big).   
\end{equation} 
\end{lemma} 
 
\begin{proof} 
Let us denote  
$$ 
D_1:=\left(\begin{array}{cc} 
-\partial_s +J_0\partial_\theta & 0 \\ 0 & -\frac d {ds}  
\end{array}\right) 
$$  
and let $D_0$ be the operator given by multiplication 
with $N$, so that $D=D_1+D_0$. The crucial point is that $D_1$ is 
diagonal and each of its components satisfies an estimate of 
the form~\eqref{eq:elliptic-loc}. For the component $\partial _s + 
J_0\partial _\theta$, this follows immediately 
from~\cite[Lemma~B.4.6.(ii)]{McDS} (with the notations therein, one 
has to take $q=r$, $p=\infty$, $\Omega'=]k,k+1[\times S^1$, 
$\Omega=]k-1,k+2[\times S^1$). For the 
component $\frac d {ds}$, the estimate follows from the fact that the 
right hand side of~\eqref{eq:elliptic-loc} defines a norm which is 
equivalent to the Sobolev norm on $W^{1,p}([k-1,k+2],T_\lambda\Lambda)$.  
 
We have $\|D_0x\|_{\cL^p([k-1,k+2])}\le 
C_1\|x\|_{\cL^p([k-1,k+2])}$ since $N$ is bounded pointwise, so that  
\begin{eqnarray*} 
\|x\|_{\cW^{1,p}([k,k+1])} \hspace{-2mm} & \le & 
\hspace{-2mm}C\big(\|D_1x\|_{\cL^p([k-1,k+2])} + 
\|x\|_{\cL^p([k-1,k+2])}\big) \\  
& \le & \hspace{-2mm}C\big(\|Dx\|_{\cL^p([k-1,k+2])} + \|D_0x\|_{\cL^p([k-1,k+2])} + 
\|x\|_{\cL^p([k-1,k+2])}\big) \\  
& \le & \hspace{-2mm}C_2\big(\|Dx\|_{\cL^p([k-1,k+2])} + \|x\|_{\cL^p([k-1,k+2])}\big). 
\end{eqnarray*} 
The estimate~\eqref{eq:elliptic} follows  
from~\eqref{eq:elliptic-loc} by summing over $k\in\Z$.  
\end{proof}

\begin{lemma} \label{lem:constant} 
Let $(\gamma,\lambda_0)\in\cP(H)$ and $(u,\lambda)$ be the constant 
trajectory at $(\gamma,\lambda_0)$, defined by 
$u(s,\theta):=\gamma(\theta)$ and $\lambda(s)=\lambda_0$. If 
$(\gamma,\lambda_0)$ is nondegenerate, then the operator 
$D:=D_{(u,\lambda)}:\cW^{1,p}\to \cL^p$ is bijective for $p>1$.  
\end{lemma} 
 
\begin{proof}  We follow~\cite[Lemma~2.4]{S} 
  and~\cite[Exercise~2.5]{S}.  
 
\smallskip  
\noindent {\it Step~1.} {\it The claim holds for $p=2$.} 
\smallskip  
 
Let  
$$ 
A=D_{(\gamma,\lambda)}:H^1(S^1,\gamma^*T\widehat W)\oplus 
T_{\lambda}\Lambda \to L^2(S^1,\gamma^*T\widehat W)\oplus 
T_{\lambda}\Lambda 
$$ 
be the asymptotic operator at $(\gamma,\lambda)$, defined 
by~\eqref{eq:Dasy}. Our nondegeneracy assumption on 
$(\gamma,\lambda)$ ensures that $A$ is bijective. We view $A$ as an 
unbounded self-adjoint operator on $H:=L^2(S^1,\gamma^*T\widehat W)\oplus 
T_{\lambda}\Lambda$ with domain $W:=H^1(S^1,\gamma^*T\widehat W)\oplus 
T_{\lambda}\Lambda$. The Hilbert space $H$ admits an orthogonal 
decomposition into negative and positive eigenspaces as $H=E^+\oplus 
E^-$. Let $P^\pm:H\to E^\pm$ be the corresponding orthogonal 
projections, and denote $A^\pm:=A|_{E^\pm}$. These operators generate 
strongly continuous semigroups $s\mapsto e^{-A^+s}$ and $s\mapsto 
e^{A^-s}$ defined for $s\ge 0$ and acting on $E^\pm$ respectively. We 
define $K:\R\to \cL(H)$ by  
$$ 
K(s):= \left\{ \begin{array}{ll} e^{-A^+s}P^+, & s\ge 0, \\ 
-e^{-A^-s}P^-, & s<0. 
\end{array}\right. 
$$ 
This function is discontinuous at $s=0$, and strongly continuous for 
$s\neq 0$. Moreover, it satisfies  
$$ 
\|K(s)\|_{\cL(H)}\le e^{-\delta s} 
$$ 
for a suitable constant $\delta>0$, because $A$ is bijective and 
therefore its eigenvalues are bounded away from $0$. We define the 
operator $Q: L^2(\R,H)\to W^{1,2}(\R,H)\cap L^2(\R,W)$ by  
$$ 
Qy(s)=\int_{-\infty}^\infty K(s-\tau)y(\tau) d\tau.  
$$ 
We note that $W^{1,2}(\R,H)\cap L^2(\R,W)=\cW^{1,2}$ and $L^2(\R,H)=\cL^2$, 
and we claim that $Q$ is the inverse of $D$. Indeed, given 
$y\in \cL^2$, the orthogonal decomposition of $x=Qy=x^+ + x^-$ is given 
by   
$$ 
x^+(s)=\int_{-\infty}^s e^{-A^+(s-\tau)}y^+(\tau)\, d\tau, \qquad  
x^-(s)=-\int^{\infty}_s e^{-A^-(s-\tau)}y^-(\tau)\, d\tau. 
$$ 
One computes directly that $\dot x^\pm + A^\pm x^\pm = y^\pm$, so that 
$\dot x+Ax=Dx=y$. This proves Step~1.  
 
\smallskip  
\noindent {\it Step~2.} {\it  
Let $p\ge 2$. There exists a constant $C>0$ such that, for 
all $k\in \Z$ and $x\in \cW^{1,p}([k-1,k+2])$, we have 
$$ 
\|x\|_{\cW^{1,p}([k,k+1])}\le C\Big(\|Dx\|_{\cL^p([k-1,k+2])} + 
\|x\|_{\cL^2([k-1,k+2])}\Big).  
$$ 
} 
\smallskip  
 
We have  
\begin{eqnarray*} 
\|x\|_{\cW^{1,p}([k,k+1])} \hspace{-2mm} & \le & \hspace{-2mm} 
  C_1\big(\|Dx\|_{\cL^p([k-\frac 1 2,k+\frac  
  3 2])} +\|x\|_{\cL^p([k-\frac 12,k+\frac 32])}\big) \\ 
& \le & \hspace{-2mm}C_2\big(\|Dx\|_{\cL^p([k-\frac 1 2,k+\frac 
  3 2])} +\|x\|_{\cW^{1,2}([k-\frac 12,k+\frac 32])}\big) \\ 
& \le & \hspace{-2mm}C_3\big(\|Dx\|_{\cL^p([k-\frac 1 2,k+\frac 
  3 2])} + \|Dx\|_{\cL^2([k-1,k+2])} + \|x\|_{\cL^2([k-1,k+2])}\big)\\ 
& \le & \hspace{-2mm}C_4\big(\|Dx\|_{\cL^p([k-1,k+2])} + 
\|x\|_{\cL^2([k-1,k+2])}\big). 
\end{eqnarray*}  
The first and third inequalities follow from Lemma~\ref{lem:elliptic}. 
The second inequality follows from the Sobolev embedding 
$\cW^{1,2}([k-\frac 12,k+\frac 32])\hookrightarrow \cL^p([k-\frac 12,k+\frac 
32])$ (see~\cite[Theorem~B.1.12]{McDS} and the subsequent discussion 
for the summands defined on $[k-\frac 12,k+\frac 32]\times S^1$, 
and~\cite[Theorem~B.1.11]{McDS} for the summands defined on $[k-\frac 
12,k+\frac 32]$). The last inequality holds because $p\ge 2$, so that  
$\cL^p(I)\hookrightarrow \cL^2(I)$ for any bounded interval $I$.   
 
\smallskip  
\noindent {\it Step~3.} {\it Let $p\ge 2$. There exists a constant 
  $C>0$ such that, if $x\in \cW^{1,2}$ and $Dx\in \cL^p$, then $x\in \cW^{1,p}$ and  
\begin{equation} \label{eq:Step3}   
\|x\|_{\cW^{1,p}} \le C \|Dx\|_{\cL^p}. 
\end{equation}  
} 
\smallskip  
 
We first remark that, if $x\in \cW^{1,2}$ and $Dx\in \cL^p_{\mathrm{loc}}$, 
then $x\in \cW^{1,p}_{\mathrm{loc}}$. Indeed, as seen in Step~2, we have an 
embedding $\cW^{1,2}(I)\hookrightarrow \cL^p(I)$ for any bounded interval 
$I$. The remark then follows from elliptic regularity for $D$ 
(see~\cite[Proposition~1.2.1]{Sik} and the references therein).  
 
Let $H:=L^2(S^1,\gamma^*T\widehat W)\oplus T_\lambda\Lambda$ and, for 
  an interval $I\subset \R$, denote the natural norm on $L^p(I,H)$ by 
  $\|\cdot\|_{L^p(I,H)}$. It follows from Step~2 and the inequality 
  $(a+b)^p\le 2^p(a^p+b^p)$ that  
\begin{eqnarray*} 
\|x\|^p_{\cW^{1,p}([k,k+1])} & \le & 
C\big(\|Dx\|_{\cL^p([k-1,k+2])}+\|x\|_{\cL^2([k-1,k+2])}\big)^p \\ 
& \le & 2^pC\big( \|Dx\|^p_{\cL^p([k-1,k+2])} + 
\|x\|^p_{\cL^2([k-1,k+2])}\big) \\ 
& = & 2^pC\big( \|Dx\|^p_{\cL^p([k-1,k+2])} + 
\|x\|^p_{L^2([k-1,k+2],H)}\big) \\ 
& \le & 2^pC\big( \|Dx\|^p_{\cL^p([k-1,k+2])} + 
3^{p/2-1}\|x\|^p_{L^p([k-1,k+2],H)}\big) \\ 
& \le & 3^{p/2-1} 2^pC \big( \|Dx\|^p_{\cL^p([k-1,k+2])} + 
\|x\|^p_{L^p([k-1,k+2],H)}\big). 
\end{eqnarray*}  
The third inequality is H\"older's. By summing over $k\in \Z$ we 
obtain  
\begin{equation} \label{eq:ineq-2-p}  
\|x\|^p_{\cW^{1,p}} \le C_1\big( \|Dx\|^p_{\cL^p} + 
\|x\|^p_{L^p(\R,H)}\big). 
\end{equation}  
 
Let $Q:\cL^2\to \cW^{1,2}$ be the inverse of $D:\cW^{1,2}\to \cL^2$ as in Step~1.  
Then  
\begin{eqnarray*}  
\lefteqn{\hspace{-3cm}\|x\|_{L^p(\R,H)} = \|QDx\|_{L^p(\R,H)} = 
  \|K*(Dx)\|_{L^p(\R,H)} } \\ 
& \le & \|K\|_{L^1(\R,\cL(H))}\|Dx\|_{L^p(\R,H)} \\ 
& \le & C_2\|Dx\|_{L^p(\R,H)} \\ 
& \le & C_3\|Dx\|_{\cL^p}. 
\end{eqnarray*}  
The first inequality is Young's inequality for a 
convolution~\cite[Th\'eor\`eme~4.30]{B}, and the last inequality follows from the 
fact that $\|\cdot\|_{L^2(S^1)}\le\|\cdot\|_{L^p(S^1)}$, while any two 
norms are equivalent on the finite dimensional space 
$T_\lambda\Lambda$. Combining the above inequality 
with~\eqref{eq:ineq-2-p}, we obtain~\eqref{eq:Step3}. This proves 
Step~3.  
 
\smallskip  
\noindent {\it Step~4.} {\it We prove the lemma for $p\ge 2$.} 
\smallskip  
 
The estimate~\eqref{eq:Step3} holds in particular for $x\in 
C^\infty_0(\R\times S^1,u^*T\widehat W)\oplus 
C^\infty_0(\R,T_\lambda\Lambda)$ and, by density, for all $x\in 
\cW^{1,p}$. We infer that $D:\cW^{1,p}\to \cL^p$ is injective and has a closed 
image. To prove that it is surjective, it is therefore enough to show 
that its image is dense in $\cL^p$. Indeed, it follows from Step~1 and 
Step~3 that its image contains the dense subspace $\cL^p\cap \cL^2$.  
 
\smallskip  
\noindent {\it Step~5.} {\it We prove the lemma for $1<p<2$.} 
\smallskip  
 
Let $q>2$ be such that $1/p+1/q=1$. Define 
$\cW^{-1,p}:=W^{-1,p}(\R\times S^1,u^*T\widehat W)\oplus 
W^{-1,p}(\R,T_\lambda\Lambda)$, where $W^{-1,p}$ is the dual space of 
$W^{1,q}$, so that $\cW^{-1,p}$ is the dual space of $\cW^{1,q}$. Note also 
that $\cL^q$ is the dual of $\cL^p$. The formal 
adjoint $D^*$ defined in~\eqref{eq:adjoint} is canonically identified 
with the functional analytic adjoint $D^*:\cW^{1,q}\to \cL^q$ of 
$D:\cL^p\to \cW^{-1,p}$. By Step~4, there exists a constant $C>0$ such 
that, for any $x\in \cW^{1,q}$, we have  
\begin{equation} \label{eq:D*}  
\|x\|_{\cW^{1,q}} \le C\|D^*x\|_{\cL^q}.  
\end{equation}  
Using that $D^*$ is bijective and duality, we obtain for $y\in \cL^p$ 
\begin{eqnarray} 
\lefteqn{\hspace{-3cm}\|y\|_{\cL^p}  =  \sup _{\|z\|_{\cL^q}=1} |\langle z,y \rangle|  
 =  \sup _{\|D^*x\|_{\cL^q}=1} |\langle D^*x,y \rangle|  
 =  \sup _{\|D^*x\|_{\cL^q}=1} |\langle x,Dy \rangle| } \nonumber \\ 
& \le & \sup _{\|D^*x\|_{\cL^q}=1} \|x\|_{\cW^{1,q}} \|Dy\|_{\cW^{-1,p}} 
 \nonumber \\ 
& \le & \sup _{\|D^*x\|_{\cL^q}=1} C\|D^*x\|_{\cL^q} \|Dy\|_{\cW^{-1,p}} 
 \nonumber \\ 
& = & C\|Dy\|_{\cW^{-1,p}}. \label{eq:X-1p} 
\end{eqnarray}  
The last inequality uses~\eqref{eq:D*}.  
 
We now prove that there exists a constant $C>0$ such that, for any 
$x\in \cW^{1,p}$, we have  
\begin{equation} \label{eq:x-Dx}   
\|x\|_{\cW^{1,p}} \le C\|Dx\|_{\cL^p}.  
\end{equation}  
For $x=(\zeta,\ell)\in \cW^{1,p}$, we have  
\begin{equation} \label{eq:3terms} 
\|x\|_{\cW^{1,p}}\le 
C_1(\|x\|_{\cL^p} + \|\partial_s x\|_{\cL^p} + \|\partial_\theta 
\zeta\|_{L^p}). 
\end{equation}  
Using~\eqref{eq:X-1p} and the inclusion 
$\cL^p\hookrightarrow \cW^{-1,p}$ we obtain  
$$ 
\|x\|_{\cL^p} \le C\|Dx\|_{\cW^{-1,p}} \le C\|Dx\|_{\cL^p}.  
$$ 
On the other hand we have  
\begin{eqnarray*}  
\|\partial_s x\|_{\cL^p} & \le & C\|D(\partial _s x)\|_{\cW^{-1,p}} \\ 
& = & C\| \partial_s(Dx) - (\partial_s N)x\|_{\cW^{-1,p}} \\ 
& \le & C_1\big( \|\partial_s(Dx)\|_{\cW^{-1,p}} + 
\|x\|_{\cW^{-1,p}}\big)\\ 
& \le & C_1\big( \|Dx\|_{\cL^p} + 
\|x\|_{\cL^p}\big) \\ 
&\le & C_2 \|Dx\|_{\cL^p}. 
\end{eqnarray*}  
The first and last inequalities use~\eqref{eq:X-1p} for $\partial_s x$ 
and $x$, the second inequality uses that $\partial_s N$ and its 
derivatives are pointwise bounded, and the third inequality uses that 
$\partial _s:\cL^p\to \cW^{-1,p}$ is bounded (and of norm $1$).  
Similarly, we have $\|\partial_\theta \zeta\|_{L^p}\le 
C_3\|Dx\|_{\cL^p}$. Using~\eqref{eq:3terms} we obtain~\eqref{eq:x-Dx}. 
 
It follows from~\eqref{eq:x-Dx} that $D:\cW^{1,p}\to \cL^p$ is injective and 
has a closed image. To prove that it is surjective, it is enough to 
show that its image is dense in $\cL^p$. Consider therefore $y\in \cL^q$ 
such that $\langle Dx,y\rangle=0$ for all $x\in \cW^{1,p}$. We obtain 
$D^*y=0$ in $\cW^{-1,q}$. By elliptic regularity for $D^*$, we infer 
$y\in \cW^{1,q}$. Since $D^*:\cW^{1,q}\to \cL^q$ is injective by Step~4, we obtain 
$y=0$.    
\end{proof} 
 
\begin{lemma} \label{lem:Fredholm} 
 Let $p>1$. Under the hypotheses of Theorem~\ref{thm:Fredholm}, there exists 
$T>0$ and a constant $C>0$ such that  
\begin{equation} \label{eq:Fredholm-bis} 
\|x\|_{\cW^{1,p}}\le C\big(\|Dx\|_{\cL^p} + \|Kx\|_{\cL^p([-T,T])}\big), 
\end{equation} 
where $K:\cW^{1,p}\to \cL^p([-T,T])$ is the restriction operator. 
\end{lemma} 
 
\begin{proof} 
Let $(\ou,\widetilde\olambda)$ and $(\uu,\widetilde\ulambda)$ be the 
constant trajectories at $(\og,\olambda)$ and $(\ug,\ulambda)$ 
respectively. Denote by $\oD:=D_{(\ou,\widetilde\olambda)}$ and 
$\uD:=D_{(\uu,\widetilde\ulambda)}$ the corresponding operators which, 
by Lemma~\ref{lem:constant}, are isomorphisms.  
Since invertibility is an open 
property in the space of operators, and because the order $0$ part of 
$D$ converges as $s\to\pm\infty$ to the order $0$ part of $\uD$ and 
$\oD$ respectively, we infer the existence of constants $T>0$ and 
$C>0$ such that, for every $x\in \cW^{1,p}$ such that $x(s)=0$ for $|s|\le 
T-1$, we have  
\begin{equation} \label{eq:truncate}  
\|x\|_{\cW^{1,p}}\le C\|Dx\|_{\cL^p}. 
\end{equation}  
Let $\beta:\R\to [0,1]$ be a smooth cutoff function such that 
$\beta(s)=0$ for $|s|\ge T$ and $\beta(s)=1$ for $|s|\le T-1$. We 
obtain  
\begin{eqnarray*} 
\|x\|_{\cW^{1,p}} & \le & \|\beta x\|_{\cW^{1,p}} + \|(1-\beta) x\|_{\cW^{1,p}} \\ 
& \le & C_1 \big( \|D(\beta x)\|_{\cL^p} + \|\beta x\|_{\cL^p} +  
\|D((1-\beta) x)\|_{\cL^p} \big) \\ 
& \le & C_2 \big( \|Dx\|_{\cL^p} + \|Kx\|_{\cL^p([-T,T])}\big). 
\end{eqnarray*}  
The first and the third inequality are straightforward, whereas the 
second uses~\eqref{eq:elliptic} and~\eqref{eq:truncate}. This proves 
the lemma. 
\end{proof}


\section{Unique continuation for the parametrized Floer equation} \label{sec:unique}

The fundamental property on which rest transversality results in Floer
theory~\cite{FHS} is the unique continuation principle for Floer
trajectories. We know of two ways to prove it. The first one is the Carleman similarity
principle~\cite[Theorem~2.2]{FHS}, which cannot hold in our setup due
to the integral term, which makes the system of
equations~(\ref{eq:Floer1par}-\ref{eq:Floer2par}) non-local. The
second one is Aronszajn's theorem, stating that a solution of a
\emph{pointwise} differential inequality involving an elliptic
operator of order $2$ satisfies the unique continuation
property~\cite{Ar}. Again, one cannot apply it to our setup
because of the integral term. 

Aronszajn's theorem relies on a local estimate~\cite[(2.4)]{Ar} which is
nowadays called a {\it Carleman-type inequality}. We will extend
Aronszajn's theorem to a class of integro-differential elliptic
inequalities by proving a semi-local Carleman-type inequality. This
generalization of Aronszajn's theorem will apply to the solutions of
our system of equations. Our arguments closely follow the ones of
Aronszajn~\cite{Ar}. 

\medskip 

For $r>0$, we denote $Z_r:=]-r,r[\times S^1$, with coordinates $(s,\theta)$. 

\begin{proposition}[Semi-local Carleman inequality] 
\label{prop:semilocalCarleman}
Let $h>0$. There exist $c>0$ and $\alpha_0 > 0$ such that, for any $0<r\le h$, 
$\alpha \ge \alpha_0$ and $u\in C^\infty_0(Z_r,\C^n)$ which vanishes together
with all its derivatives along $\{0\}\times S^1$, we have 
\begin{equation} \label{eq:Carleman}
c\, r^2 \, \int_{Z_r} |s|^{-2\alpha} |\Delta u|^2 \, dsd\theta \ge 
\int_{Z_r} |s|^{-2\alpha} \big(|\nabla u|^2 + |u|^2\big) \, dsd\theta.
\end{equation}
\end{proposition}

\begin{proof}[First proof.] In our first proof, we use a change of variables inspired by the original paper of Aronszajn~\cite{Ar}. 
 It is enough to prove the inequality on $Z_r^+:=]0,r[\times S^1$. We
make the change of variables $s=e^{-\rho}$, $\chi < \rho <\infty$,
$\chi=-\log r$ and define 
$$w:]\chi,\infty[\times S^1\to \C^n, \quad 
w(\rho,\theta)=e^{\beta\rho}u(e^{-\rho},\theta),
$$ 
with $\beta=\alpha+ \frac 3 2$. Our assumption on $u$ guarantees that
$w$ vanishes together with all its derivatives as $\rho\to\infty$.
We denote $w'=\p w /\p \rho$ and
$w''=\p^2 w/ \p \rho^2$. A straightforward computation shows
that 
\begin{eqnarray*}
\frac {\p u} {\p
s}(e^{-\rho},\theta) & = & -e^{(1-\beta)\rho}\big(w'(\rho,\theta)-\beta
w(\rho,\theta)\big), \\
\frac {\p^2 u} {\p
s^2}(e^{-\rho},\theta) & = & e^{(2-\beta)\rho}\big(w''(\rho,\theta)-(2\beta-1)
w'(\rho,\theta) + (\beta^2-\beta)w(\rho,\theta)\big).
\end{eqnarray*}
We therefore obtain 
$$
\int_{Z_r^+} |s|^{-2\alpha} |\Delta u|^2 \, dsd\theta =
\int_\chi^\infty \!\!\!\!\int_{S^1} |w''-(2\beta-1)w'+(\beta^2-\beta)w +
\Delta_\theta w|^2 \, d\rho d\theta, 
$$
with $\Delta_\theta w:=e^{-2\rho} \frac {\p^2 w} {\p \theta^2}$. 
We denote the last integral by $I$. Expanding the integrand in $I$ we
obtain 
\begin{eqnarray*}
I & = & \int_\chi^\infty \!\!\!\!\int_{S^1} |w''|^2 + |(\beta^2-\beta)w +
\Delta_\theta w|^2  \\
& & \hspace{-.5cm}+ \ \int_\chi^\infty  \!\!\!\!\int_{S^1} (2\beta-1)^2|w'|^2 +
w''[(\beta^2-\beta)\bar w + \Delta_\theta \bar w] + 
\bar w''[(\beta^2-\beta) w + \Delta_\theta w] \\
& & \hspace{-.5cm}+ \ \int_\chi^\infty  \!\!\!\!\int_{S^1} (1-2\beta)
(w'\Delta_\theta \bar w + \bar w'\Delta_\theta w).
\end{eqnarray*}
We have used that $\iint w'' \bar w' + \bar w''w' =0$ and that $\iint
w'\bar w+\bar w'w=0$, which follow from the fact that $w$ and $w'$
vanish for $\rho\to\infty$ and $\rho$ near $\chi$. We denote the above
three integrals by $J^1$, $J^2$, and $J^3$ respectively. Since $J^1\ge
0$ we have $I\ge J^2+J^3$. 

We treat $J^2$. Using integration by parts with respect to $\rho$ we
obtain $\iint w''\bar w+\bar w''w=-2\iint |w'|^2$. Using integration
by parts with respect to $\theta$ and $\rho$ we obtain 
\begin{eqnarray*}
\int_\chi^\infty \!\!\!\!\int_{S^1} w''\Delta_\theta\bar w & = & 
\int_\chi^\infty \!\!\!\!\int_{S^1} \Big(\big| \frac {\p w'}
{\p\theta} \big|^2 -2 \frac {\p w'}
{\p\theta} \frac {\p \bar w} {\p\theta} \Big) e^{-2\rho} \\
& = & \int_\chi^\infty \!\!\!\!\int_{S^1} \Big(\big| \frac {\p w'}
{\p\theta} \big|^2 +2 \frac {\p w}
{\p\theta} \frac {\p \bar w'} {\p\theta} - 4 \big|\frac {\p w}
{\p\theta}\big|^2 \Big) e^{-2\rho}.
\end{eqnarray*}
For the second equality we have used another integration by parts with
respect to $\rho$. Thus 
$$
\int_\chi^\infty \!\!\!\!\int_{S^1} w''\Delta_\theta\bar w +
\bar w''\Delta_\theta w = \int_\chi^\infty \!\!\!\!\int_{S^1} \Big(2\big| \frac {\p w'}
{\p\theta} \big|^2  - 4 \big|\frac {\p w}
{\p\theta}\big|^2 \Big) e^{-2\rho}.
$$
Finally 
\begin{equation} \label{eq:J2}
J^2= \int_\chi^\infty \!\!\!\!\int_{S^1} \big(2(\beta-\frac 1 2)^2+
\frac 1 2\big) |w'|^2 + 2 \big| \frac {\p w'}
{\p\theta} \big|^2 e^{-2\rho} - \int_\chi^\infty \!\!\!\!\int_{S^1} 4 \big|\frac {\p w}
{\p\theta}\big|^2 e^{-2\rho}.
\end{equation}

We treat $J^3$. Integrating by parts with respect to $\theta$ and
$\rho$ we obtain 
\begin{eqnarray*}
\int_\chi^\infty \!\!\!\!\int_{S^1} w'\Delta_\theta \bar w & = & -
\int_\chi^\infty \!\!\!\!\int_{S^1} \frac {\p w'} {\p \theta} \frac {\p
\bar w} {\p \theta} e^{-2\rho} \\
& = & \int_\chi^\infty \!\!\!\!\int_{S^1} \Big( \frac {\p w} {\p
\theta} \frac {\p \bar w'} {\p \theta} -2 \big| \frac {\p w} {\p
\theta} \big|^2\Big) e^{-2\rho},
\end{eqnarray*}
so that 
$$
\int_\chi^\infty \!\!\!\!\int_{S^1} w'\Delta_\theta \bar w  + \bar
w'\Delta_\theta w = -\int_\chi^\infty \!\!\!\!\int_{S^1} 2\big| \frac {\p w} {\p
\theta} \big|^2 e^{-2\rho}.
$$

We denote $I^2$ the first integral in the expression~\eqref{eq:J2}
for $J^2$, and 
\begin{eqnarray*}
I^3 & := & J^3 -\int_\chi^\infty \!\!\!\!\int_{S^1} 4 \big|\frac {\p w}
{\p\theta}\big|^2 e^{-2\rho} \\
& = & \int_\chi^\infty \!\!\!\!\int_{S^1} (4\beta-6) \big|\frac {\p w}
{\p\theta}\big|^2 e^{-2\rho}.
\end{eqnarray*}
 so
that $J^2+J^3=I^2+I^3$.  

We now treat the right hand side in~\eqref{eq:Carleman}. Using the
same change of variables as above we obtain 
$$
\int_{Z_r^+} \!\! |s|^{-2\alpha} \big(|\nabla u|^2 + |u|^2\big) \,
dsd\theta \! = \!\!
\int_\chi^\infty \!\!\!\!\int_{S^1} \!\! |w'-\beta w|^2 e^{-2\rho} + \big|\frac {\p w}
{\p\theta}\big|^2 e^{-4\rho} + |w|^2 e^{-4\rho} \, d\rho d\theta.
$$
The first term in the integrand is $|w'-\beta
w|^2e^{-2\rho}=|w'|^2e^{-2\rho} + \beta^2|w|^2e^{-2\rho} -
\beta(w'\bar w+\bar w'w)e^{-2\rho}$, and we have 
$$
\int_\chi^\infty \!\!\!\!\int_{S^1} (w'\bar w+\bar w'w)e^{-2\rho} = 
\int_\chi^\infty \!\!\!\!\int_{S^1} (w'\bar w - \bar w(w'-2w))
e^{-2\rho}
=\int_\chi^\infty \!\!\!\!\int_{S^1} 2|w|^2 e^{-2\rho}.
$$
The right hand side in~\eqref{eq:Carleman} is therefore equal to 
$$
\int_\chi^\infty \!\!\!\!\int_{S^1} |w'|^2e^{-2\rho} +
(\beta^2-2\beta)|w|^2 e^{-2\rho} + \big|\frac {\p w}
{\p\theta}\big|^2 e^{-4\rho} + |w|^2 e^{-4\rho}.
$$

We now recall the inequality~(4.10) in~\cite{Ar} which writes in our case,
for $\theta\in S^1$ fixed,
\begin{equation} \label{eq:Poincare}
\int_\chi^\infty |w|^2 e^{-\tau \rho} \, d\rho \le \frac
{e^{-\tau\chi}} {\tau ^2} \int_\chi^\infty |w'|^2\, d\rho, \qquad \tau>0.
\end{equation}
To prove~\eqref{eq:Poincare} we write $w(\rho)=\int_\chi^\rho w'$, and use
the Cauchy-Schwarz inequality to obtain
$|w(\rho)|^2\le(\rho-\chi)\int_\chi^\infty |w'|^2$. On the other hand
we have $\int_\chi^\infty e^{-\tau\rho}(\rho-\chi)\, d\rho =
e^{-\tau\chi}/\tau^2$. 

Using~\eqref{eq:Poincare} with $\tau=2$ and the relations $4\beta -
6=4\alpha \ge 4$ and $I\ge I^2+I^3$, we obtain the desired conclusion
with the constants $c=1$ and $\alpha_0 = 1$. 
\end{proof}

\begin{proof}[Second proof, by Luc Robbiano]
We again work on $Z_r^+$. We define $v:=s^{-\alpha}u$, so that $v$ vanishes with all its derivatives along $\{0\}\times S^1$. 
We have
\begin{eqnarray*}
\p_s u & = & s^\alpha \p_s v + \alpha s^{\alpha-1}v, \\
\p_s^2 u & = & s^\alpha\p_s^2v +2\alpha s^{\alpha-1}\p_s v + \alpha(\alpha-1) s^{\alpha-2}v.
\end{eqnarray*}
We obtain 
$$
s^{-\alpha}\Delta u = \p_s^2v + \p_\theta^2v+ 2\alpha s^{-1}\p_s v + \alpha(\alpha-1) s^{-2}v.
$$ 
In order to estimate 
$A:=\int_{Z_r^+}|s^{-\alpha}\Delta u|^2$, we separate self-adjoint and anti-adjoint terms in the previous expression. Denoting $\langle\cdot,\cdot\rangle$ the $L^2$-scalar product for functions defined on $Z_r^+$, and $\|\cdot\|$ the corresponding $L^2$-norm, we obtain 
\begin{eqnarray*}
A& = & \|\p_s^2 v + \p_\theta^2v+ \alpha(\alpha-1) s^{-2}v\|^2 + 4\alpha^2\|s^{-1}\p_s v\|^2 \\
& & + \ 2\Re\langle \p_s^2 v + \p_\theta^2v+ \alpha(\alpha-1) s^{-2}v,2\alpha s^{-1}\p_s v\rangle.
\end{eqnarray*}
Let us further compute the last term. We have
\begin{eqnarray*}
\langle \p_s^2v,2\alpha s^{-1}\p_sv\rangle & = & -\langle \p_s v,2\alpha\p_s(s^{-1}\p_sv)\rangle \\
& = & -\langle \p_sv,-2\alpha s^{-2} \p_sv\rangle -\langle \p_sv,2\alpha s^{-1}\p_s^2v\rangle,
\end{eqnarray*}
so that 
$$
2\Re\langle \p_s^2 v,2\alpha s^{-1}\p_s v\rangle = 2\alpha\| s^{-1}\p_s v\|^2.
$$
Similarly, we have 
\begin{eqnarray*}
\langle \p_\theta^2v,2\alpha s^{-1}\p_sv\rangle & =  & -\langle 2\alpha \p_s(s^{-1}\p_\theta^2 v), v\rangle \\
& = & -2\alpha \langle s^{-1}\p_s\p_\theta^2 v, v\rangle + 2\alpha\langle s^{-2}\p_\theta^2 v,v\rangle \\
& = & -2\alpha \langle s^{-1}\p_s v,\p_\theta^2 v\rangle + 2\alpha\langle s^{-2}\p_\theta^2 v,v\rangle,
\end{eqnarray*}
so that 
$$
2\Re\langle \p_\theta^2v,2\alpha s^{-1}\p_sv\rangle =  2\alpha\langle s^{-2}\p_\theta^2 v,v\rangle=-2\alpha\|s^{-1}\p_\theta v\|^2.
$$
Finally, we have 
\begin{eqnarray*}
\langle \alpha(\alpha-1) s^{-2}v,2\alpha s^{-1}\p_s v\rangle & \!\!=\!\! & -\langle 2\alpha^2(\alpha-1)\p_s(s^{-3}v),v\rangle\\
& \!\!=\!\!& -\langle 2\alpha^2(\alpha-1)s^{-3}\p_sv,v\rangle + \langle 6\alpha^2(\alpha-1)s^{-4}v,v\rangle,
\end{eqnarray*}
so that
$$
2\Re \langle \alpha(\alpha-1) s^{-2}v,2\alpha s^{-1}\p_s v\rangle = 6\alpha^2(\alpha-1)\|s^{-2}v\|^2.
$$
Let us now denote $Bv:=\p_s^2v+\p_\theta^2v+\alpha(\alpha-1)s^{-2}v$, so that 
\begin{eqnarray*}
A& =& \|Bv\|^2+(4\alpha^2+2\alpha)\|s^{-1}\p_sv\|^2 + 6\alpha^2(\alpha-1)\|s^{-2}v\|^2 + 2\alpha\|s^{-1}\p_\theta v\|^2 \\
& & + \  4\alpha\langle s^{-2}(Bv-\p_s^2v-\alpha(\alpha-1)s^{-2}v),v\rangle.
\end{eqnarray*}
We again further compute the last term. We have 
$$
|\langle 4\alpha s^{-2}Bv,v\rangle |\le \|Bv\|^2 + 4\alpha^2\|s^{-2}v\|^2.
$$
We also have 
\begin{eqnarray*}
-4\alpha\langle s^{-2}\p_s^2v,v\rangle & =& 4\alpha\langle \p_s v,\p_s(s^{-2}v)\rangle \\
& =& 4\alpha \langle \p_sv,s^{-2}\p_s v\rangle -8\alpha \langle \p_sv,s^{-3}v\rangle,
\end{eqnarray*}
and 
$$
|8\alpha \langle \p_sv,s^{-3}v\rangle|\le 4\|s^{-1}\p_s v\|^2 + 4\alpha^2 \|s^{-2}v\|^2. 
$$
We obtain 
\begin{eqnarray*}
A & \ge & (4\alpha^2+6\alpha-4)\|s^{-1}\p_s v\|^2 + 2\alpha^2(\alpha-5)\|s^{-2}v\|^2 +2\alpha \|s^{-1}\p_\theta v\|^2\\
& \ge & 2 \|s^{-1}\p_s v\|^2+4\alpha^2\|s^{-2}v\|^2+2\|s^{-1}\p_\theta v\|^2.
\end{eqnarray*}
The last inequality holds if $\alpha\ge 7$. Now since $v := s^{-\alpha} u$, we have
\begin{eqnarray*}
\| s^{-\alpha-1} \p_s u\| &\le&  \| s^{-1} \p_s v \| + \alpha \| s^{-2} v \|,  \\
\| s^{-\alpha-2} u \| &=& \| s^{-2} v \|, \\
\| s^{-\alpha-1} \p_\theta u \| &=& \| s^{-1} \p_\theta v \|. 
\end{eqnarray*}
Substituting in the above estimate, we obtain
$$
A \ge \| s^{-\alpha-1} \p_s u\|^2 + \| s^{-\alpha-1} \p_\theta u \|^2 + \| s^{-\alpha-2} u \|^2. 
$$
Since $s < r \le h$, we deduce the desired inequality with the constants $c = \max(1,h^2)$ and $\alpha_0 = 7$.
\end{proof}

\begin{theorem}[Unique continuation for integro-differential
inequalities] \label{thm:inequality}
Let $h>0$. Assume $u\in C^\infty(Z_h,\C^n)$ satisfies 
\begin{equation} \label{eq:inequality}
|\Delta u(s,\theta)|^2\le M\Big[
|u(s,\theta)|^2+ |\nabla u(s,\theta)|^2 + \int_{S^1}|u(s,\tau)|^2 \,
d\tau \Big]
\end{equation}
for all $(s,\theta)\in Z_h$, where $M>0$ is a positive constant. If
$u$ vanishes together with all its derivatives on $\{0\}\times S^1$,
then $u\equiv 0$ on $Z_h$. 
\end{theorem}

\begin{proof} It is enough to prove that $u$ vanishes in a
neighbourhood of $\{0\}\times S^1$. The conclusion then follows by a
connectedness argument on $]-h,h[$. Let $0<r<1/\sqrt{(2\pi+1)cM}$ be
fixed, where $c>0$ is the constant in~\eqref{eq:Carleman}. 
Let $\varphi : ]-r,r[\to [0,1]$ be
a smooth function equal to $1$ for $|s|\le r/3$ and equal to $0$ for
$|s|\ge 2r/3$. Let $u_1(s,\theta):=\varphi(s) u(s,\theta)$. We have 
\begin{eqnarray*}
\lefteqn{\int_{Z_{r/3}} |s|^{-2\alpha} \Big[
|u(s,\theta)|^2+ |\nabla u(s,\theta)|^2 + \int_{S^1}|u(s,\tau)|^2 \,
d\tau \Big] dsd\theta } \\
&\le & (2\pi+1) \int_{Z_{r/3}} |s|^{-2\alpha} \Big[
|u(s,\theta)|^2+ |\nabla u(s,\theta)|^2 \Big] dsd\theta \\
&\le & (2\pi+1) \int_{Z_r} |s|^{-2\alpha} \big(
|u_1|^2+ |\nabla u_1|^2 \big)  \\
&\le & (2\pi+1) cr^2 \int_{Z_r} |s|^{-2\alpha} |\Delta u_1|^2 \\
&= & (2\pi+1) cr^2 \int_{Z_{r/3}} |s|^{-2\alpha} |\Delta u|^2 +
(2\pi+1) cr^2 \int_{Z_r\setminus Z_{r/3}} |s|^{-2\alpha} |\Delta
u_1|^2 \\
&\le & (2\pi+1) cr^2M \int_{Z_{r/3}} |s|^{-2\alpha} \Big[
|u|^2+ |\nabla u|^2 + \int_{S^1}|u|^2\Big] \\ 
& & + \ (2\pi+1) cr^2 \int_{Z_r\setminus Z_{r/3}} |s|^{-2\alpha} |\Delta
u_1|^2.
\end{eqnarray*}
The third inequality follows from Proposition~\ref{prop:semilocalCarleman}, for $\alpha\ge \alpha_0$. It follows that 
\begin{eqnarray*}
\int_{Z_{r/3}} |s|^{-2\alpha} \Big[
|u|^2+ |\nabla u|^2 + \int_{S^1}|u|^2\Big]
& \le & C \int_{Z_r\setminus Z_{r/3}} |s|^{-2\alpha} |\Delta
u_1|^2 \\
& \le & \frac C {(r/3)^{2\alpha}} \int_{Z_r\setminus Z_{r/3}} |\Delta
u_1|^2,
\end{eqnarray*}
with $C=(2\pi+1) cr^2/\big(1-(2\pi+1) cr^2M\big)$. We claim that
$u\equiv 0$ on $Z_{r/3}$. Following Carleman~\cite{Ca}, we assume this to
be false: there exists
$(s_0,\theta_0)\in Z_{r/3}$ such that $u(s_0,\theta_0)\neq 0$. Hence
there exists a constant $k>0$ (depending on $u$, but not on $\alpha$)
such that 
$$
\frac k {|s_0|^{2\alpha}}\le \int_{Z_{r/3}} |s|^{-2\alpha} |u|^2
$$
for all $\alpha\ge \alpha_0$. In view of the above, we obtain 
$$
0< k \le C \frac {|s_0|^{2\alpha} }  {(r/3)^{2\alpha}} 
\int_{Z_r\setminus Z_{r/3}} |\Delta
u_1|^2. 
$$
Since $|s_0|<r/3$, we obtain a contradiction as $\alpha\to \infty$. 
\end{proof}

In the next statement we denote $I_h:=]-h,h[$ for $h>0$. 

\begin{proposition} \label{prop:vanish}
Let $h>0$ and $u:Z_h\to \C^n$,
$\lambda:I_h\to \R^m$ be $C^\infty$-functions satisfying
\begin{equation} \label{eq:eqCn}
\left\{\begin{array}{rcl} 
\p_s u + J(s,\theta) \p_\theta u + C(s,\theta) u + D(s,\theta) \lambda
& = & 0, \\
\p_s \lambda + \int_{S^1} E(s,\theta) u(s,\theta) \, d\theta +
F(s)\lambda & = & 0, 
\end{array}\right.
\end{equation}
with $C,D,E,F$ of class $C^1$, $J$ of class $C^\infty$ and
$J^2=-\one$. Assume there exists a nonempty open set $\cU\subset Z_h$
such that $(u(s,\theta),\lambda(s))=(0,0)$ for all $(s,\theta)\in \cU$. 
Then $u\equiv 0$ on $Z_h$ and $\lambda\equiv 0$ on $I_h$. 
\end{proposition}

\begin{proof} We first notice that, for any $(s,\theta)\in \cU$, there
exists $\eps >0$ such that $u\equiv 0$ on $]s-\eps,s+\eps[\times S^1$
and $\lambda \equiv 0$ on $]s-\eps,s+\eps[$. Indeed, choose $\eps>0$
small enough such that 
$]s-\eps,s+\eps[\times]\theta-\eps,\theta+\eps[\subset \cU$. The
condition on $\lambda$ follows then from the hypothesis. On the other
hand, $u$ satisfies $\p_s u +J\p_\theta u +C u=0$ on
$]s-\eps,s+\eps[\times S^1$ and vanishes on this domain by the
standard unique continuation property~\cite[Theorem~2.2,
Proposition~3.1]{FHS}. 
Let us assume without loss of generality that $(0,\theta)\in \cU$ for
some $\theta\in S^1$. The previous discussion shows that the pair
$(u,\lambda)$ vanishes together with all its derivatives along
$\{0\}\times S^1$.

Let $i$ denote the standard complex structure on
$\C^n$. We choose a $C^\infty$ function $\Psi:Z_h\to
\mathrm{GL}_\R(\C^n)$ such that $J\Psi=\Psi i$, and we define
$v:Z_h\to \C^n$ by $u=\Psi v$. Then $v$ is $C^\infty$ and satisfies  
$$
\p_s v + i \p_\theta v + \widetilde C(s,\theta) v + \widetilde
D(s,\theta) \lambda=0, 
$$
with $\widetilde C=\Psi^{-1}(\p_s \Psi + J\p_\theta\Psi + C\Psi)$ and
$\widetilde D=\Psi^{-1}D$. Moreover, $\lambda$ satisfies 
$$
\p_s \lambda + \int_{S^1} \widetilde E(s,\theta) v(s,\theta) \, d\theta +
F(s)\lambda = 0,
$$
with $\widetilde E=E\Psi$. Thus $\widetilde C$, $\widetilde D$, and
$\widetilde E$ are $C^1$. We assume in the sequel 
without loss of generality that $J=i$. 

Denote $U(s,\theta):=(u(s,\theta),\lambda(s),0)\in \C^n\times \C^m$, so
that $U:Z_h\to \C^{n+m}$ satisfies an equation of the form
$$
\p_sU+i\p_\theta U+A(s,\theta)U+\int_{S^1}B(s,\tau)U(s,\tau)\, d\tau =0,
$$
for some $A,B$ of class $C^1$. Applying $\p_s$ and $-i\p_\theta$ to
this equation, summing, substituting $\p_s U$ from the equation, and
integrating once by parts with respect to $\theta$, we obtain 
$$
\Delta U + A_1 U + A_2 \p_s U + A_3 \p_\theta U + \int_{S^1}
A_4(s,\tau) U(s,\tau)\, d\tau =0. 
$$
Here $A_j$, $j=1,\dots,4$ are $C^0$ and given by $A_1=\p_s A
-i\p_\theta A$, $A_2=A$, $A_3=-iA$, $A_4=\p_s B-BA+\p_\theta B i
-B\int_{S^1}B$. By restricting to a smaller cylinder $Z_{h'}$, $h'<h$
so that the $A_j$ are bounded pointwise by some constant $K>0$, we
obtain 
\begin{eqnarray*}
|\Delta U|^2 & \le & 4K \Big[ |U|^2 + |\nabla U|^2 +
\big(\int_{S^1}|U|\big)^2\Big] \\
& \le & 8\pi K \Big[ |U|^2 + |\nabla U|^2 +
\int_{S^1}|U|^2\Big].
\end{eqnarray*}
The conclusion follows from Theorem~\ref{thm:inequality}.
\end{proof}

\begin{remark} \label{rmk:vanish} 
Assuming that the coefficients $C,D,E,F$ in~\eqref{eq:eqCn} are
$C^\infty$, the conclusion of Proposition~\ref{prop:vanish} holds
under the assumption that $\lambda(0)=0$ and $u(0,\cdot)\equiv
0$. Indeed, by successive differentiation in~\eqref{eq:eqCn}, we
obtain that the pair $(u,\lambda)$ vanishes together with all its 
derivatives along $\{0\}\times S^1$. 
\end{remark}

\begin{proposition}[Unique continuation] \label{prop:unique} 
  Let $h>0$ and $u_i:Z_h \to \widehat W$, $\lambda_i:I_h\to \Lambda$, $i=0,1$ be smooth functions satisfying
  equations~(\ref{eq:Floer1par}-\ref{eq:Floer2par}), i.e. 
\begin{eqnarray*}  
 \p_s u + J_{\lambda(s)}^\theta (\p_\theta u - 
X_{H_{\lambda(s)}}^\theta (u)) & = & 0, \\  
 \dot \lambda (s) - \int_{S^1} \vec \nabla_\lambda 
H(\theta,u(s,\theta),\lambda(s)) d\theta & = & 0. 
\end{eqnarray*} 
If $(u_0,\lambda_0)$ and $(u_1,\lambda_1)$ coincide on some
nonempty open set $\cU\subset Z_h$, then they coincide on
$Z_h$. 
\end{proposition} 

\begin{proof} We can assume without loss of generality that
  $\cU=I_\delta\times I_\eps$, for some $\delta,\eps>0$. Since $\lambda_0 =
  \lambda_1$ on $I_\delta$, it follows that $u_0$ and $u_1$ satisfy the 
  same Floer equation $\p_s u + J^\theta_s (\p_\theta u - X^\theta_s)=0$
  on $I_\delta\times S^1$. Since $u_0$ and $u_1$ coincide on $\cU$, it
  follows by the unique continuation property for the Floer
  equation~\cite[Proposition~3.1]{FHS} that $u_0=u_1$ on
  $I_\delta\times S^1$. 

  Let $I\subset I_h$ be the set of points $s$ such that $u_0=u_1$ on
  $\{s\}\times S^1$, and $\lambda_0(s)=\lambda_1(s)$. Then $I\supset
  I_\delta$ and hence is nonempty. Moreover, it is closed. To 
  prove the Proposition, it is enough to show that $I$ is open. Let
  $s_0\in I$ be a point on the boudary of a connected component of $I$
  with nonempty interior, and denote
  $\gamma:=u_0(s_0,\cdot)=u_1(s_0,\cdot)$. We 
  consider a trivialization of $\gamma^*T\widehat W$ of the form
  $S^1\times \C^n$, and a local chart in $\Lambda$ around
  $\lambda_0(s_0)=\lambda_1(s_0)$, which we 
  identify with $\R^m$. Then, for $s$ close to $s_0$, we can view
  $u_0(s,\cdot)$ and $u_1(s,\cdot)$ as taking values in $\C^n$, and
  similarly $\lambda_0(s)$ and $\lambda_1(s)$ as taking values in
  $\R^m$. The difference $(u,\lambda):=(u_0-u_1,\lambda_0-\lambda_1)$
  then satisfies an equation of the form~\eqref{eq:eqCn} with smooth
  coefficients (the computation is similar to the one in the proof
  of~\cite[Proposition~3.1]{FHS}). Moreover, $(u,\lambda)$ vanishes to
  infinite order along $\{s_0\}\times S^1$. By
  Proposition~\ref{prop:vanish}, we obtain that $(u,\lambda)\equiv 0$
  on a small strip around $\{s_0\}\times S^1$, so that $s_0$ belongs
  to the interior of $I$. 
\end{proof} 

\begin{remark} The conclusion of Proposition~\ref{prop:unique} holds
  under the assumption that $u_0(s_0,\cdot)=u_1(s_0,\cdot)$ and
  $\lambda_0(s_0)=\lambda_1(s_0)$ for some $s_0\in\R$ (use
  Remark~\ref{rmk:vanish}). By successive
  differentiation, this hypothesis implies that $(u_0,\lambda_0)$ and
  $(u_1,\lambda_1)$ coincide together with all their derivatives along
  $\{s_0\}\times S^1$. 
\end{remark}


\section{Transversality for the parametrized Floer equation} \label{sec:param-transv}

Let $H\in\cH_{\Lambda,\mathrm{reg}}$. A pair $(J,g)\in\cJ_\Lambda$ is 
  called {\bf regular for \boldmath$H$} if the operator $D_{(u,\lambda)}$ is 
  surjective for any solution $(u,\lambda)$ 
  of~(\ref{eq:Floer1par}-\ref{eq:asymptoticpar}). We denote the space 
  of such pairs by $\cJ_{\Lambda,\mathrm{reg}}(H)$. We prove in this
  section part~\ref{item:INTROb-notS1} of Theorem~B as the following
  statement. 

\begin{theorem} \label{thm:Jreg}
There exists a subset of second Baire category
$\cH\cJ_{\Lambda,\mathrm{reg}}\subset 
\cH_{\Lambda,\mathrm{reg}}\times \cJ_\Lambda$ such that $(J,g)\in
\cJ_{\Lambda,\mathrm{reg}}(H)$ whenever $(H,J,g)\in
\cH\cJ_{\Lambda,\mathrm{reg}}$. 
\end{theorem}

\begin{remark} In general, it is not possible to first fix $H\in
  \cH_{\Lambda,\mathrm{reg}}$ and then prove that
  $\cJ_{\Lambda,\mathrm{reg}}(H)$ is of the second Baire category in
  $\cJ_\Lambda$, as the following example shows. Consider a
  Hamiltonian of the form
  $H(\theta,x,\lambda)=K(\theta,x)+g(x)f(\lambda)$. Assume $K$ has
  nondegenerate $1$-periodic orbits with disjoint geometric
  images, fix a regular almost complex structure $J$ for $K$,
  consider a Floer trajectory $u$ for $(K,J)$ with asymptotes $\og$,
  $\ug$, let $g\equiv 1$ near $\og$, $g\equiv -1$ near $\ug$, and let
  $\lambda_0$ be a minimum of $f$. If $\dim\,\Lambda >
  \mathrm{ind}(u)$, then $(u,\lambda_0)$ is a parametrized
  Floer trajectory of negative index, independently of the choice of
  Riemannian metric $g$ on $\Lambda$. Moreover, since $u$ survives
  under small perturbations of $J$, the parametrized trajectory
  $(u,\lambda_0)$ will survive under small perturbations of the pair
  $(J,g)$. This shows that the latter cannot be chosen generically to
  be regular. 

  This phenomenon is similar to the one arising in the construction of
the continuation morphism in Morse homology from a regular pair
$(f_-,g_-)$ to a regular pair $(f_+,g_+)$. In that situation, we again
\emph{cannot} first fix
the homotopy $(f_t)$ and then choose the homotopy $(g_t)$
generically. One has to choose the \emph{pair} $(f_t,g_t)$
generically. An explicit example is provided by the homotopy $f_t:\R\to \R$ 
given by $f_t(x)=-\frac 1 2 tx^2$ for $t\in 
[-1,1]$, $x\in \R$. In this case the constant trajectory at $x=0$ has
index $-1$ and exists for any choice of metric.  
\end{remark}

Let $(u,\lambda)\in \cM(\op,\up)$, $\op=(\og,\olambda)$,
$\up=(\ug,\ulambda)$. We define the set of {\bf regular points for
\boldmath$(u,\lambda)$} by
$$
R(u,\lambda):=\left\{(s,\theta)\in \R\times S^1 \, : \,
\left\{\begin{array}{l}
(\p_s u(s,\theta),\p_s \lambda(s))\neq (0,0),\\
(u(s,\theta),\lambda(s))\neq (\og(\theta),\olambda), \
(\ug(\theta),\ulambda), \\
(u(s,\theta),\lambda(s))\notin
(u(\cdot,\theta),\lambda(\cdot))(\R\setminus \{s\})
\end{array}\right. \right\}.
$$

\noindent {\bf Notation.} In the following we denote
$U(s,\theta)=(u(s,\theta),\lambda(s))$ and assume that $U$ satisfies
equations (\ref{eq:Floer1par}-\ref{eq:Floer2par}). We also denote
$R(U):=R(u,\lambda)$.  

\begin{proposition}[Regular points] \label{prop:regularpts} 
Assume $\p_s U \not\equiv (0,0)$. Then :

(i) The set $\{ (s,\theta)  : \p_s U(s,\theta) \neq (0,0) \}$ is open and dense in $\R \times S^1$. 
Moreover, given $s\in \R$, there exists $\theta\in S^1$ such that $\p_s U(s,\theta)\neq (0,0)$.

(ii) The set $R(u,\lambda)$ is open. 

(iii) If $\p_s u\equiv 0$, then $R(u,\lambda)$ is equal to $\R\times S^1$.
If $\p_s u\not\equiv 0$, then $R(u,\lambda)$ is dense in the open 
set $\{(s,\theta)\, : \, \p_s u(s,\theta)\neq 0\}$.  
\end{proposition}

\begin{remark}
If $\p_su \not\equiv 0$ and $\p_s\lambda \equiv 0$, the Proposition implies that 
$R(u,\lambda)$ is dense in $\R\times S^1$. Indeed, $u$ satisfies a 
Floer equation which is independent of $s$ and the open set 
$\{(s,\theta) \, : \, \p_su(s,\theta)\neq 0\}$ is dense in 
$\R\times S^1$~\cite[Lemma~4.1]{FHS}. 
\end{remark}

To prove Proposition~\ref{prop:regularpts}, we need the following
enhancement of Proposition~\ref{prop:unique} (this is the analogue of
Lemma~4.2 in~\cite{FHS}). In the next statement we denote 
$V_h(s,\theta):=\ ]s-h,s+h[ \ \times \ ]\theta-h,\theta+h[ \ \subset \R\times
S^1$ and $I_h(s):= \ ]s-h,s+h[ \ \subset \R$ for $h>0$. 

\begin{lemma} \label{lem:unique} 
Let $U_i=(u_i,\lambda_i)$, $i=0,1$ be smooth functions defined on 
a strip $I_{h_0}\times S^1$, $h_0>0$ and 
  satisfying equations~(\ref{eq:Floer1par}-\ref{eq:Floer2par}) in
  Proposition~\ref{prop:unique}. Assume that 
$$
U_0(s_0,\theta_0)=U_1(s_0,\theta_0), \qquad \p_s u_0(s_0,\theta_0)
\neq 0, \qquad \p_s U_1(s_0,\theta_0)\neq (0,0)
$$
for some $(s_0,\theta_0)\in\R\times S^1$. Assume also 
that, for any $0<h'\le h_0$, there exists $0<h\le h_0$ with the
following property: for any $(s,\theta)\in  
  V_h(s_0,\theta_0)$, there exists $(s',\theta)\in
  V_{h'}(s_0,\theta_0)$ such that  
$$
U_0(s,\theta)=U_1(s',\theta).
$$
Then $U_0=U_1$. 
\end{lemma} 

\begin{remark} 
We could not prove Lemma~\ref{lem:unique} under the more general
assumption $\p_s U_0(s_0,\theta_0)\neq (0,0)$ (instead of $\p_s
u_0(s_0,\theta_0)\neq 0$). This in turn influences the
conclusion of~(iii) in Proposition~\ref{prop:regularpts}: we only show
that $R(u,\lambda)$ is dense in the set $\{(s,\theta)\, : \, \p_s
u(s,\theta)\neq 0\}$.
\end{remark} 

\begin{proof}[Proof of Lemma~\ref{lem:unique}] By
  Proposition~\ref{prop:unique}, it is enough to prove 
  that $U_0=U_1$ on some open neighbourhood of $(s_0,\theta_0)$. 
Let us choose $h'>0$ small enough so that $I_{h'}(s_0)\to \widehat
W\times \Lambda, s\mapsto U_1(s,\theta)$ is an embedding for all
$\theta\in I_{h'}(\theta_0)$. By further
diminishing the corresponding $h>0$ we 
can also assume that $I_h(s_0)\to \widehat
W, s\mapsto u_0(s,\theta)$ is an embedding for all
$\theta\in I_h(\theta_0)$. 

For each $\theta\in I_h(\theta_0)$, we have by assumption 
$U_0(I_h(s_0),\theta)\subset U_1(I_{h'}(s_0),\theta)$. We can
therefore define smooth embeddings $G_\theta:=(U_1(\cdot,\theta))^{-1}\circ 
U_0(\cdot,\theta):I_h(s_0)\to I_{h'}(s_0)$. Moreover, for
$h$ small enough, we have $s_0\in \mathrm{im}(G_\theta)$. Let us
choose $0<h''<h'$ small enough such that
$I_{h''}(s_0)\subset 
\mathrm{im}(G_\theta)$ for all $\theta\in I_h(\theta_0)$. By the
implicit function theorem, we obtain a 
smooth embedding $F_\theta:=(G_\theta)^{-1}:I_{h''}(s_0)\to
I_h(s_0)$. The collection of maps $\{F_\theta\}$ gives rise to the
smooth map $F:V_{h''}(s_0,\theta_0)\to V_h(s_0,\theta_0)$ defined by
$F(s,\theta):=(F_\theta(s),\theta):=(\phi(s,\theta),\theta)$. We have 
$$
U_1(s,\theta)=U_0(\phi(s,\theta),\theta)
$$
for all $(s,\theta)\in V_{h''}(s_0,\theta_0)$. Substituting in the
Floer equation for $u_1$, we obtain 
\begin{eqnarray*}
0 & = & \p_s u_1 + J^\theta_{\lambda_1}(u_1)(\p_\theta u_1 -
X^\theta_{H_{\lambda_1}}(u_1)) \\
& = & \p_s u_0(F) \cdot \p_s \phi \\
& & + \ 
J^\theta_{\lambda_0(F)}(u_0(F))\big(\p_su_0(F)\cdot\p_\theta\phi +
\p_\theta u_0(F) - X^\theta_{H_{\lambda_0(F)}}(u_0(F))\big) \\
& = & \p_s u_0(F) \cdot (\p_s\phi-1) + J^\theta_{\lambda_0(F)}
(u_0(F)) \p_s u_0(F)\cdot \p_\theta\phi.
\end{eqnarray*}
The last equality follows from the Floer equation for $u_0$. 
Since $\p_s u_0 \neq 0$ on
$V_{h}(s_0,\theta_0)$ we see that the vectors $\p_s u_0(F)$ and
$J^\theta_{\lambda_0(F)}\p_s u_0(F)$ are linearly independent, so that
$\p_s\phi\equiv 1$ and $\p_\theta\phi\equiv 0$. Since
$\phi(s_0,\theta_0)=s_0$, we obtain $\phi(s,\theta)=s$ for all
$(s,\theta)\in V_{h''}(s_0,\theta_0)$ and the conclusion follows.
\end{proof}

\begin{proof}[Proof of Proposition~\ref{prop:regularpts}]
(i) A straightforward computation shows that, for any $s\in\R$, the pair
$\p_s U=(\p_s u,\p_s \lambda)$ satisfies an equation 
of the form~\eqref{eq:eqCn} with smooth coefficients in a local
trivialization along the loop $u(s,\cdot)$ and in a local chart around
$\lambda(s)$. Assume by contradiction that $\p_sU \equiv (0,0)$ on
some nonempty open set $\cU$. By Proposition~\ref{prop:vanish}, we obtain
that $\p_sU \equiv (0,0)$ on some open strip around $\cU$. By the now 
standard open-closed argument we get $\p_sU\equiv 0$ on $\R\times S^1$, 
which contradicts the hypothesis. 

We now prove the second statement. Assuming by
contradiction the existence of a point $s_0\in\R$ such that
$\p_sU\equiv (0,0)$ along $\{s_0\}\times S^1$, we obtain by
Remark~\ref{rmk:vanish} that $\p_s U\equiv (0,0)$ on a strip around
$\{s_0\}\times S^1$. We then conclude as above.

(ii) The first two conditions defining the elements of $R(u,\lambda)$
are clearly open. We need to show that the third one is 
open as well. Arguing by contradiction, we find a point $(s_0,\theta_0)\in
R(u,\lambda)$, a sequence $(s^\nu,\theta^\nu)\to (s_0,\theta_0)$, and
a sequence $s^{\prime\nu}\neq s^\nu$ such that
$U(s^{\prime\nu},\theta^\nu)= U(s^\nu,\theta^\nu)$. Since $\p_s
U(s_0,\theta_0)\neq (0,0)$, we can find $h>0$ such that
$U(\cdot,\theta_0)$ is an embedding on $I_h(s_0)$ and
$U(\cdot,\theta^\nu)$ is an embedding on $I_h(s^\nu)$, for $\nu$ large
enough. Thus, we can assume without loss of generality that
$s^{\prime\nu}$ is bounded away from $s_0$ (otherwise $s^{\prime\nu}
\in I_h(s^\nu)$ for $\nu$ large enough, a contradiction). Since $U$ 
converges at $\pm\infty$ to its asymptotes, and $U(s_0,\theta_0)$ does
not lie on those asymptotes by assumption, we infer the existence of
some $T>0$ such that $s^{\prime\nu}\in [-T,T]$ for all $\nu$. We can
therefore extract a convergent subsequence, still denoted
$s^{\prime\nu}$, such that $s^{\prime\nu}\to s'_0\neq s_0$. Then
$U(s'_0,\theta_0)=U(s_0,\theta_0)$, which contradicts the assumption
that $(s_0,\theta_0)\in R(u,\lambda)$. 

(iii) If $\p_s u\equiv 0$, then $\lambda$ satisfies an ordinary
differential equation independent of $s$ and, since $\p_s \lambda
\not\equiv 0$, we obtain that $\p_s \lambda\neq 0$ on $\R$ and
every point $(s,\theta)\in \R\times S^1$ is regular. 

Let us now assume $\p_su \not \equiv 0$. It is enough to show that
for any $(s,\theta)$ such that $\p_s
u(s,\theta)\neq 0$, there exists a neighbourhood 
$\cU$ such that $R(u,\lambda)\cap \cU$ is dense
in $\cU$. 
Let $(s_0,\theta_0)\in \R\times S^1$ be such that
$\p_s u(s_0,\theta_0)\neq 0$. We choose $h>0$ small enough such
that $\p_s u\neq 0$ on $V_h(s_0,\theta_0)$ and $I_h(s_0)\to
\widehat W$, $s\mapsto u(s,\theta)$ is an embedding for
all $\theta\in I_h(\theta_0)$. Then $I_h(s_0)\to
\widehat W\times \Lambda$, $s\mapsto U(s,\theta)$ is a fortiori also
an embedding for all $\theta\in I_h(\theta_0)$.
Since every point $(s,\theta)\in
V_h(s_0,\theta_0)$ can be approximated by a sequence
$(s^\nu,\theta^\nu)$ satisfying 
$U(s^\nu,\theta^\nu)\neq \op(\theta^\nu), \up(\theta^\nu)$, 
we can assume without loss of generality that 
\begin{equation} \label{eq:Uopup}
 \forall \, (s,\theta)\in V_h(s_0,\theta_0), \ U(s,\theta)\neq
 \op(\theta),\up(\theta).  
\end{equation} 
The conclusion of the Proposition now reduces to
showing that $(s_0,\theta_0)$ can be approximated by a sequence
$(s^\nu,\theta^\nu)\in R(U)$. Assuming this to be false, there exists 
$0<\eps<h$ such that $V_\eps(s_0,\theta_0) \cap R(U) = \emptyset$,
i.e. 
\begin{equation} \label{eq:notinj}
\forall \, (s,\theta)\in V_\eps(s_0,\theta_0), \ \exists \, s'\neq s, \
U(s',\theta)=U(s,\theta). 
\end{equation} 
Since $\lim_{s\to -\infty} U(s,\theta)=\op(s,\theta)$ and $\lim_{s\to
  \infty} U(s,\theta)=\up(s,\theta)$ uniformly in $\theta$, we infer
from~\eqref{eq:Uopup} the existence of a constant $T>0$ such that
$|s'|\le T$ in~\eqref{eq:notinj}.

Let us denote $C(U):=\{(s,\theta)\in \R\times S^1 \, : \,
\p_sU(s,\theta)=(0,0)\}$. Note that, by the proof of (i), the set
$C(U)$ has empty interior. 
We now claim that $(s_0,\theta_0)$ can be approximated by a sequence
$(s^\nu,\theta_0)$ such that, for all $\nu$ and all $s'\in\R$ with
$U(s',\theta_0)=U(s^\nu,\theta_0)$, we have $(s',\theta_0)\notin
C(U)$. Assuming the claim, we can suppose 
without loss of generality that, for each $s'\in\R$ such that
$U(s',\theta_0)=U(s_0,\theta_0)$, we have $(s',\theta)\notin
C(U)$. Moreover, after further diminishing $\eps>0$, we can 
assume without loss of generality that 
\begin{equation} \label{eq:C(U)}
\forall \, (s,\theta)\in V_\eps(s_0,\theta_0), \ \forall \, s'\in\R,
\ U(s,\theta)=U(s',\theta) \ \Rightarrow \ (s',\theta)\notin C(U).
\end{equation} 
Indeed, if this would fail for all $\eps>0$, we could find a
sequence $(s^\nu,\theta^\nu)\to (s_0,\theta_0)$ and a sequence
$s^{\prime\nu}$ such that $(s^{\prime\nu},\theta^\nu)\in C(U)$,
$U(s^{\prime\nu},\theta^\nu)=U(s^\nu,\theta^\nu)$, and $|s^{\prime\nu}
- s_0|\ge \eps_0>0$. Up to a subsequence, we have $s^{\prime\nu}\to s'
\in [-T,T]$, $\theta^\nu\to \theta_0$, and
$U(s',\theta_0)=U(s_0,\theta_0)$ with $(s',\theta_0)\in C(U)$. This
contradicts our last assumption on $(s_0,\theta_0)$, obtained via
the claim.  

To prove the claim, let us choose a neighbourhood $\cV$ of
$U(I_\eps(s_0),\theta_0)$ in $\widehat W\times \Lambda$, of the form
$I_\eps(s_0)\times \R^{2n+m-1}$, and denote $\mathrm{pr}_1$ the
projection to the first coordinate interval $I_\eps(s_0)$. Let
$f:=\mathrm{pr}_1\circ U(\cdot,\theta_0)$, with
$f:\mathrm{dom}(f):=U(\cdot,\theta_0)^{-1}(\cV)\to I_\eps(s_0)$. Let
$C(U)_{\theta_0}:=\{s\in \R \, : \, (s,\theta_0)\in C(U)\}$. Then
$f(C(U)_{\theta_0} \cap \mathrm{dom}(f))$ is contained in
the set of critical values of $f$. By Sard's theorem, this is a
nowhere dense set in $I_\eps(s_0)$, and the claim follows. 

We now closely follow the proof of Theorem~4.3 in~\cite{FHS}. We first
remark that, for any $(s,\theta)\in V_\eps(s_0,\theta_0)$, there is
only a finite number of values $s'\in\R$ such that
$U(s',\theta)=U(s,\theta)$. If not, we could find an accumulation
point $s'\in [-T,T]$ such that $\p_s U(s',\theta)=(0,0)$ and
$U(s',\theta)=U(s,\theta)$, a contradiction with~\eqref{eq:C(U)}. 
Let $s_1,\dots,s_N\in [-T,T]$ be the points such that
$U(s_j,\theta_0)=U(s_0,\theta_0)$, $j=1,\dots,N$. 

We now claim that, for any $r>0$ there exists $\delta>0$ such that 
$$
\forall \, (s,\theta)\in V_{2\delta}(s_0,\theta_0), \ \exists \,
(s',\theta)\in \bigcup _{j=1} ^N V_r(s_j,\theta_0), \
U(s,\theta)=U(s',\theta). 
$$
If this would fail, we could find $r>0$ and a sequence
$(s^\nu,\theta^\nu)\to (s_0,\theta_0)$ such that, for all $\nu$ and
for all $(s',\theta^\nu)\in \bigcup _{j=1} ^N V_r(s_j,\theta_0)$, we
have $U(s^\nu,\theta^\nu)\neq U(s',\theta^\nu)$. On the other hand,
by~\eqref{eq:notinj} there exists $s^{\prime\nu}\in [-T,T]$ such that
$U(s^{\prime\nu},\theta^\nu)=U(s^\nu,\theta^\nu)$, and in particular
$|s^{\prime\nu}-s_j|\ge r$ for all $j$. Up to a subsequence
we have $s^{\prime\nu}\to s'$ and
$\theta^\nu\to \theta_0$, so that $U(s',\theta_0)=U(s_0,\theta_0)$ and
$s'\neq s_j$, $j=1,\dots,N$, a contradiction. 

Following~\cite{FHS}, we define
$$
\Sigma_j:=\{(s,\theta)\in \overline {V_\delta}(s_0,\theta_0) \, : \,
\exists \, (s',\theta)\in \overline {V_r}(s_j,\theta_0), \, U(s',\theta)=U(s,\theta)\}.
$$
Then $\Sigma_j$ is closed and $\overline {V_\delta}(s_0,\theta_0) =
\Sigma_1\cup \dots\cup \Sigma_N$. It follows from Baire's theorem that
one of the $\Sigma_j$ -- say $\Sigma_1$ -- has nonempty
interior. 

Let $(\os,\otheta)\in \mathrm{int}(\Sigma_1)$ and denote
$(\os_1,\otheta)$ the unique preimage of
$U(\os,\otheta)$ in $V_r(s_1,\theta_0)$. Let $0<r_1<r$ be such that
$V_{r_1}(\os_1,\otheta)\subset V_r(s_1,\theta_0)$, and $0<\delta_1<\delta$
be such that $V_{\delta_1}(\os,\otheta)\subset \Sigma_1$, and such that
for all $(s,\theta)\in V_{\delta_1}(\os,\otheta)$, there exists $(s',\theta)\in
V_{r_1}(\os_1,\otheta)$ such that $U(s,\theta)=U(s',\theta)$. It
follows from our construction that, 
for all $0<h'\le r_1$, there exists $0<h\le \delta_1$, such that 
for all $(s,\theta)\in V_h(\os,\otheta)$, there exists
$(s',\theta)\in V_{h'}(\os_1,\otheta)$, such that $U(s,\theta)=U(s',\theta)$. We can
therefore apply Lemma~\ref{lem:unique} with
$(s_0,\theta_0):=(\os,\otheta)$, $U_0:=U$, $U_1:=U(\cdot
+\os_1-\os,\cdot)$, and $h_0=r_1$, to obtain $U_0=U_1$. 
This implies  
$$
U(s,\theta)=\lim_{k\to\pm\infty} U(s+k(\os_1-\os),\theta) = \op(\theta)
= \up(\theta). 
$$
This contradicts our standing assumption $\p_sU \not\equiv
(0,0)$. Proposition~\ref{prop:regularpts} is proved. 
\end{proof} 


%

\begin{proof}[Proof of Theorem~\ref{thm:Jreg}] Let
$\cJ_\Lambda^r$, $r\ge 1$ denote the space of pairs 
$(J,g)$ of class $C^r$ such that $J$ is an admissible almost complex
structure on $\widehat W$. Let $\cH_{\Lambda,\mathrm{reg}}^r$, $r\ge
1$ denote the space of regular admissible Hamiltonians of class
$C^r$. Let $\cH\cJ_{\Lambda,\mathrm{reg}}^r \subset
\cH_{\Lambda,\mathrm{reg}}^r \times \cJ_\Lambda^r$ denote the space of
triples $(H,J,g)$ such that $(J,g)$ is regular for $H$. 
By a standard argument due to Taubes~\cite[p.52]{McDS}, it is enough to
prove that $\cH\cJ_{\Lambda,\mathrm{reg}}^r$ is of the second Baire
category in $\cH_{\Lambda,\mathrm{reg}}^r \times \cJ_\Lambda^r$. 

Given $p>2$ and $\op,\up\in\cP(H)$, we denote by $\cB$ the space of pairs
$(u,\lambda)$ consisting of maps $u:\R\times S^1\to \widehat W$ and
$\lambda:\R\to \Lambda$ which are locally of class $W^{1,p}$, which
satisfy~\eqref{eq:asymptoticpar}, and which are of class $W^{1,p}$
in local charts near the asymptotes. Then $\cB\times
\cH_{\Lambda,\mathrm{reg}}^r \times \cJ_\Lambda^r$ is 
a Banach manifold. There is a Banach bundle $\cE\to \cB\times
\cH_{\Lambda,\mathrm{reg}}^r \times \cJ_\Lambda^r$ whose fiber at
$(u,\lambda,H,J,g)$ is $\cL^p:=L^p(\R\times S^1,u^*T\widehat W)\oplus 
L^p(\R,\lambda^*T\Lambda)$. The solutions of the parametrized Floer
equations~(\ref{eq:Floer1par}-\ref{eq:Floer2par}) for $(H,J,g)$ are the
zeroes of the section $f:\cB\times \cH_{\Lambda,\mathrm{reg}}^r \times
\cJ_\Lambda^r \to \cE$ given by  
$$
f(u,\lambda,H,J,g):=\left(\begin{array}{c}
\p_s u + J_{\lambda(s)}^\theta (\p_\theta u - 
X_{H_{\lambda(s)}}^\theta (u)) \\
\dot \lambda (s) - \int_{S^1} \vec \nabla_\lambda 
H(\theta,u(s,\theta),\lambda(s)) d\theta 
\end{array}\right). 
$$
The crucial step is to prove that the universal moduli space
$\cM:=f^{-1}(0)$ is a Banach submanifold of $\cB\times
\cH_{\Lambda,\mathrm{reg}}^r \times \cJ_\Lambda^r$. Then the claim
follows easily from the Sard-Smale 
theorem as in~\cite[Proof of Theorem~3.1.5.(ii)]{McDS}. 

The vertical differential of $f$ at a point $(u,\lambda,H,J,g)\in\cM$ is
given by 
\begin{eqnarray*} 
\lefteqn{df(u,\lambda,H,J,g)\cdot (\zeta,\ell,h,Y,A) :=   D_{(u,\lambda)}(\zeta,\ell)} \\
& &
+ \ 
\left(\begin{array}{c}
-J_{\lambda(s)}^\theta X_{h_{\lambda(s)}}^\theta (u) +
Y_{\lambda(s)}^\theta (\p_\theta u -  
X_{H_{\lambda(s)}}^\theta (u)) \\
- \int_{S^1} \vec \nabla_\lambda 
h(\theta,u(s,\theta),\lambda(s)) d\theta +  A \cdot \int_{S^1} \vec
\nabla_\lambda  
H(\theta,u(s,\theta),\lambda(s)) d\theta 
\end{array}\right),
\end{eqnarray*}
where $h\in T_H\cH_{\Lambda,\mathrm{reg}}^r$ and $(Y,A)\in
T_{(J,g)}\cJ_\Lambda^r$. For a description of $h$ we refer to the
proof of Proposition~\ref{prop:Hreg}. We view $Y$ as a family  
$Y=(Y^\theta_\lambda)$, $\lambda\in \Lambda$, $\theta\in S^1$ with
$Y^\theta_\lambda\in \mathrm{End}(T\widehat W)$, such that
$Y^\theta_\lambda J^\theta_\lambda + J^\theta_\lambda
Y^\theta_\lambda=0$ and $\widehat \omega(Y^\theta_\lambda \cdot,\cdot)+
\widehat \omega(\cdot,Y^\theta_\lambda \cdot)=0$. Moreover, for $t\ge
0$ large enough, we have that $Y^\theta_\lambda$ is independent of
$t$, it preserves $\xi$ and vanishes on $\langle \partial/\partial
t,R_\alpha\rangle$. The element $A$ is a tangent vector at $g$ to the
space $\mathrm{Met}^r(\Lambda)$ of Riemannian metrics of class
$C^r$ on $\Lambda$. Considering a $1$-parameter family 
$g^\eps\in \mathrm{Met}^r(\Lambda)$ such that $g^0=g$, we define $A$
by $g(A\cdot,\cdot)=\frac d
{d\eps}|_{\eps=0} g^\eps$, so that $A$ is an element of
$\mathrm{End}(T\Lambda)$ which is symmetric with respect to
$g$. Denoting by $\vec\nabla^\eps_\lambda H$ the $\lambda$-gradient of
$H$ with respect to $g^\eps$, we then have $\frac d {d\eps}|_{\eps=0}
\vec\nabla^\eps_\lambda H = -A \cdot \vec\nabla_\lambda H$, hence the
formula for the vertical differential of $f$. 

We need to show that $df$ is surjective. Since the image of $D_{(u,\lambda)}$ is
closed and has finite codimension, it follows that the image of $df$
has the same property. Thus, it suffices to show that it is dense. Let
$(\eta,k)\in \cL^q=(\cL^p)^*$, $1/p + 1/q=1$ be an element annihilating
$\mathrm{Im}(df)$. Using that $(u,\lambda,H,J,g)\in \cM$, this means
that 
\begin{eqnarray}
\int_{\R\times S^1} \Big\langle \eta,D_u\zeta + (D_\lambda J\cdot
\ell)J \partial _s u - J(D_\lambda X_{H_\lambda}\cdot \ell) -
JX_{h_\lambda} +
Y^\theta_\lambda J\partial _s u \Big\rangle \, dsd\theta 
\nonumber
\\
+\int_\R \Big\langle k,\nabla_s\ell -\nabla_\ell \int_{S^1} \vec
\nabla_\lambda H -  \int_{S^1} \nabla_\zeta \vec \nabla _\lambda H 
-\int_{S^1} \vec \nabla_\lambda h +
A\cdot \dot \lambda \Big\rangle \, ds =  0 \nonumber \\
\label{eq:ann1}
\end{eqnarray}
for any $(\zeta,\ell,h,Y,A)\in T_{(u,\lambda)}\cB \oplus
T_H\cH_{\Lambda,\mathrm{reg}}^r \oplus
T_{(J,g)}\cJ_\Lambda^r$. We claim that $(\eta,k)=(0,0)$. 
Taking $h=0$, $Y=0$, $A=0$ we obtain that $(\eta,k)$ lies in the kernel of
the formal adjoint $D_{(u,\lambda)}^*$. The latter has the same form
as $D_{(u,\lambda)}$ and is therefore elliptic with smooth
coefficients. By elliptic regularity, it follows that $\eta$ and $k$
are smooth. We distinguish now three cases. 

\smallskip 

\noindent {\it Case~1.} $\p_s u \equiv 0$ and $\p_s\lambda \equiv
0$. By Lemma~\ref{lem:constant} the operator $D_{(u,\lambda)}$ is
bijective, so that $df$ is surjective. 

\smallskip 

\noindent {\it Case~2.} $\p_su\equiv 0$ and $\dot\lambda \not\equiv
0$. In this case $\lambda$ satisfies an ordinary differential equation
independent of $s$ and therefore $\dot\lambda\neq 0$ on $\R$ and
every point $(s,\theta)\in \R\times S^1$ is regular. We claim $k\equiv
0$. Indeed, if there existed $s_0\in\R$ with $k(s_0)\neq 0$, we could
take $\zeta=0$, $\ell=0$, $h=0$, $Y=0$ and $A$ supported in a small
neighbourhood of $\lambda(s_0)$ such that $A(\lambda(s_0))\dot\lambda(s_0)=k(s_0)$, so
that the sum of the integrals in~\eqref{eq:ann1} would be $>0$. We claim
$\eta\equiv 0$. Indeed, if there existed $(s_0,\theta_0)$ such that
$\eta(s_0,\theta_0)\neq 0$, we could take $\zeta=0$, $\ell=0$, $Y=0$,
$A=0$ and $h$ supported
near $\lambda(s_0)$, and satisfying
$J^\theta_{\lambda(s_0)}
X_h(\theta,\og(\theta),\lambda(s_0))=-\eta(s_0,\theta)$ for all
$\theta\in S^1$. Then
the sum of the integrals in~\eqref{eq:ann1} would be $>0$. 

\smallskip 

\noindent {\it Case~3.} $\p_su\not\equiv 0$. By Proposition~\ref{prop:regularpts}, there exists a nonempty open set $\Omega\subset \R\times S^1$ consisting of regular points $(s,\theta)$ such that $\p_su(s,\theta)\neq 0$. 

We first claim that $\eta\equiv 0$ on $\Omega$. Arguing by contradiction, we find $(s_0,\theta_0)\in\Omega$ such that $\eta(s_0,\theta_0)\neq 0$. We then take $\zeta=0$, $\ell=0$, $h=0$, $A=0$ and $Y$ supported near $(\theta_0, u(s_0,\theta_0), \lambda(s_0))$ such that
$Y^{\theta_0}_{\lambda(s_0)} J^{\theta_0}_{\lambda(s_0)}
\p_s u(s_0,\theta_0)=\eta(s_0,\theta_0)$. The second integral in~\eqref{eq:ann1} is zero, whereas the first one localizes near $(s_0,\theta_0)$ and is positive. This contradicts~\eqref{eq:ann1}. 

We now claim that $k(s)=0$ for all $(s,\theta)\in\Omega$. 
Arguing again by contradiction, we find $(s_0,\theta_0)\in\Omega$ such that $k(s_0)\neq 0$. We consider a function $h$ of the form $h(\theta,x,\lambda)=\phi(\theta)\psi(x)h_1(\lambda)$ such that $\phi$ is a cutoff function supported near $\theta_0$, $\psi$ is a cutoff function supported near $u(s_0,\theta_0)$, $h_1$ is supported in a neighbourhood of $\lambda(s_0)$ and satisfies $\vec\nabla_\lambda h_1(\lambda(s_0))=-k(s_0)$. The crucial observation is that, if the support of $\psi$ is small enough (depending on the choice of $h_1$), then 
$$
\langle k(s),\vec\nabla_\lambda h(\theta,u(s,\theta),\lambda(s)\rangle\ge 0
$$ 
on $\R\times S^1$, and vanishes outside a small neighbourhood of $(s_0,\theta_0)$. Here we use that $(s_0,\theta_0)$ is a regular point and $\p_su(s_0,\theta_0)\neq 0$. 
We now take $\zeta=0$, $\ell=0$, $Y=0$, $A=0$, and $h$ as above, so that the first integral in~\eqref{eq:ann1} vanishes, and the second integral is positive, a contradiction. 

\end{proof}  

\begin{remark}
We needed the possibility to deform the metric $g$ in the proof of
Theorem~\ref{thm:Jreg} only to treat Case~2. 
\end{remark}


\section{Fredholm theory in the \boldmath$S^1$-invariant case} \label{sec:S1inv}

In this section we take the parameter space to be $\Lambda=S^{2N+1}$, $N\ge 1$. 

\medskip 

We denote by $\cH^{S^1}_N\subset
\cH_{S^{2N+1}}$ the set of admissible    
Hamiltonian families $H:S^1\times \widehat W\times S^{2N+1}\to \R$ 
which are invariant with respect to the diagonal $S^1$-action on
$S^1\times S^{2N+1}$, meaning that $H(\theta+\tau,x,\tau\lambda)=H(\theta,x,\lambda)$ for all $\tau\in S^1$. It follows from the 
definitions that there exists $t_0\ge 0$ such that, for $t\ge t_0$, we 
have $H(\theta,p,t,\lambda)=\beta e^t +\beta'(\lambda)$, with 
$0<\beta\notin\mathrm{Spec}(M,\alpha)$, and $\beta'\in 
C^\infty(S^{2N+1},\R)$ being $S^1$-invariant.  
 
Given $H\in \cH^{S^1}_N$, the parametrized action functional $\cA$ is 
$S^1$-invariant, and so is the set of its critical points $\cP(H)$. 
For $p=(\gamma,\lambda)\in\cP(H)$, we denote
$$ 
S_p=S_{(\gamma,\lambda)}:= \{(\tau\gamma,\tau\lambda) \ : \ 
\tau\in S^1\} \subset \cP(H), 
$$ 
so that $S_p=S_{\tau \cdot p}$, $\tau\in S^1$. We refer to $S_p$ 
as an {\bf \boldmath$S^1$-orbit of critical points}.  
 
We denote by $\cJ_N^{S^1}$ 
the set of pairs $(J,g)$ consisting of an $S^1$-invariant admissible 
$S^{2N+1}$-family of almost complex structures $J$ on $\widehat W$, and of an 
$S^1$-invariant Riemannian metric $g$ on $S^{2N+1}$. The $S^1$-invariance condition on $J$ means that 
$J_{\tau\lambda}^{\theta+\tau}=J_\lambda^\theta$ for all $\tau\in S^1$.
 
\medskip 

An $S^1$-orbit of critical points $S_p\subset \cP(H)$ is called {\bf 
  nondegenerate} if the Hessian $d^2\cA(\gamma,\lambda)$ has a 
  $1$-dimensional kernel $V_p$ for some (and hence any) $(\gamma,\lambda)\in 
  S_p$. It follows from 
  Lemma~\ref{lem:d2-asy} that nondegeneracy is equivalent to the fact 
  that the kernel of the asymptotic operator $D_p$ is also 
  $1$-dimensional and equal to $V_p$. In both cases, a generator of 
  $V_p$ is given by the infinitesimal generator of the $S^1$-action.  
 
We define the set $\cH^{S^1}_{N,\reg}\subset \cH^{S^1}_N$ to consist 
of elements $H$ such that, for any $p\in\cP(H)$, the $S^1$-orbit $S_p$ 
is nondegenerate.  
 
\begin{proposition} \label{prop:genericHS1}
 The set $\cH^{S^1}_{N,\reg}$ is of the second Baire category in 
 $\cH^{S^1}_N$. Moreover, if $H\in \cH^{S^1}_{N,\reg}$, each 
 $S^1$-orbit $S_p\subset C^\infty(S^1,\widehat W)\times S^{2N+1}$ is 
 isolated.  
\end{proposition}  
 
\begin{proof} 
The proof is similar to that of Proposition~\ref{prop:Hreg}. Given an
integer $r \ge 2$, we denote by $\cH^{r,S^1}_N$ the space of 
$S^1$-invariant admissible Hamiltonian families of class $C^r$. 
We denote by $\cH^{r,S^1}_{N,\textrm{reg}} \subset \cH^{r,S^1}_{S^{2N+1}}$ 
the set of Hamiltonians  $H$ such that, for any $p\in\cP(H)$, the $S^1$-orbit $S_p$ 
is nondegenerate.  
For $t_0\ge 0$, we denote $\{t\le t_0\}:=W\cup M\times [0,t_0]$ and  
$\cH^{r,S^1}_{N,\textrm{reg},t_0}\subset\cH^{r,S^1}_N$ the set of 
Hamiltonians $H$ such that for any $p = (\gamma,\lambda)\in\cP(H)$ with 
$\mathrm{im}(\gamma)\subset \{t\le t_0\}$, the $S^1$-orbit $S_p$ is nondegenerate. 
Then  
$$ 
\cH^{r,S^1}_{N,\textrm{reg}}=\bigcap_{t_0\ge 0} \cH^{r,S^1}_{N,\textrm{reg},t_0}. 
$$ 
As in Proposition~\ref{prop:Hreg}, it is enough to prove that  $\cH^{r,S^1}_{N,\textrm{reg},t_0}$ 
is open and dense, so that $\cH^{r,S^1}_{N,\textrm{reg}}$ is of the second Baire category. 
The proof that $\cH^{r,S^1}_{N,\textrm{reg},t_0}$ is open is similar to the proof that 
$\cH^r_{\Lambda, \textrm{reg}, t_0}$ is open in Proposition~\ref{prop:Hreg}. 
We now prove that $\cH^{r,S^1}_{N,\textrm{reg},t_0}$ is dense. 
We consider the Banach bundle $\cE\to \cH^{r,S^1}_N \times C^r(S^1,\{t\le t_0\})\times 
S^{2N+1} \times \R$ whose fiber at $(H,\gamma,\lambda,a)$ is 
$\cE_{(H,\gamma,\lambda,a)}:=C^{r-1}(S^1,\gamma^*T\widehat W)\times 
T_\lambda S^{2N+1}$, and the section $\of$ given by  
$$ 
\of(H,\gamma,\lambda,a):= (\dot \gamma - X_H\circ \gamma + a \dot \gamma, 
-\int_{S^1} \vec \nabla _\lambda H + a X),  
$$ 
where the vector field $X$ denotes the infinitesimal generator of the $S^1$-action 
on $S^{2N+1}$. We first prove that $\overline \cP:=\of^{-1}(0)$ is a Banach 
submanifold of $\cH^{r,S^1}_N \times C^r(S^1,\{t\le t_0\})\times S^{2N+1} \times 
\R$. Indeed, the vertical differential of $\of$ at a point 
$(H,\gamma,\lambda,a)\in \overline\cP$ is given by  
\begin{eqnarray*} 
\lefteqn{ d\of(H,\gamma,\lambda,a)\cdot(h,\zeta,\ell,b) } \\ 
&=& 
\left(\begin{array}{c} 
\nabla_\theta\zeta-\nabla_\zeta X_H - (D_\lambda X_H)\cdot\ell-X_h + b \dot \gamma \\ 
-\int_{S^1}\nabla_\zeta \vec\nabla_\lambda H - \int_{S^1} \nabla_\ell 
\vec \nabla_\lambda H - \int_{S^1}\vec \nabla _\lambda h + b X 
\end{array}\right) \\ 
&=& df(H,\gamma,\lambda) \cdot (h, \zeta, \ell) + b \left(\begin{array}{c} 
\dot \gamma \\ 
X  
\end{array} \right) , 
\end{eqnarray*} 
where $f(H,\gamma, \lambda)$ is the restriction of $\of$ to $\{ a = 0 \}$. 
That $d\of(H,\gamma,\lambda,a)$ is surjective is seen as 
follows. First, $d\of(H,\gamma,\lambda,a)\cdot(h,0,0,a+1) =  
(0,X)$ for $h = H$ near $\im(\gamma)$. 
Given $k\in T_\lambda S^{2N+1}$ such that $g(k,X) = 0$, we have 
$(0,k)=d\of(H,\gamma,\lambda,a)\cdot(h,0,0,0)$, with 
$h(\cdot,\cdot,\lambda)=\mathrm{ct.}$ in some neighbourhood of 
$\mathrm{im}(\gamma)$, $h$ is $S^1$-invariant 
and $\vec\nabla _\lambda h=k$.  
Given $\eta\in C^{r-1}(S^1,\gamma^*T\widehat W)$, let us choose 
$h \in T_H \cH^{r,S^1}_N$ such that $X_h=-\eta$ along $\gamma$.  
Then the first component of $d\of(H,\gamma,\lambda,a)\cdot(h,0,0,0)$ 
is equal to $\eta$. This proves that $d\of(H,\gamma, \lambda, a)$ is surjective 
and that $\overline \cP$ is a Banach submanifold as desired.  
 
We now claim that the set of regular values of the natural projection  
$pr : \overline\cP\to \cH^{r,S^1}_N$ is contained in $\cH^{r,S^1}_{N,\textrm{reg},t_0}$. 
It then follows from the Sard-Smale theorem that the latter is dense.  
Given such a regular value $H$, for any $(H, \gamma, \lambda, a)\in pr^{-1}(H)$ we have that
\begin{eqnarray*} 
\forall \, h \in T_H \cH^{r,S^1}_N, \
\exists \, (\zeta, \ell, b),   \  \left( \begin{array}{c} 
-X_h \\ 
 - \int_{S^1} \vec\nabla_\lambda h 
\end{array} \right) 
+ D_{(\gamma, \lambda)} (\zeta, \ell) + b  
\left( \begin{array}{c} 
\dot \gamma \\ 
X 
\end{array} \right) = 0. 
\end{eqnarray*} 
Since the restriction of $d\of$ to $T_H\cH^{r,S^1}_N \oplus 0 \oplus 0 \oplus \R$ 
is surjective, we deduce that the cokernel of $D_{(\gamma, \lambda)}$ has dimension  
at most $1$ for any $(H, \gamma, \lambda, a) \in pr^{-1}(H)$ and in particular for any 
$(\gamma, \lambda) \in \cP(H)$. 
On the other hand, since $D_{(\gamma, \lambda)}$ is selfadjoint, the 
same holds for $\dim \ker D_{(\gamma, \lambda)}$. But the latter is at least $1$ due 
to the $S^1$-symmetry, which proves the claim. 
\end{proof}  
 
Let $d>0$ be small enough (for a fixed $H\in\cH^{S^1}_{N,\reg}$, one 
can take $d>0$ to be smaller than the minimal spectral gap of the asymptotic 
operators $D_p$, $p\in\cP(H)$), and fix $1<p<\infty$. Given 
$\op,\up\in \cP(H)$ and $(u,\lambda)\in \widehat 
\cM(S_\op,S_\up;H,J,g)$, we define   
\begin{eqnarray*} 
  \cW^{1,p,d} & := & W^{1,p}(u^*T\widehat 
  W;e^{d|s|}dsd\theta) \oplus W^{1,p}(\lambda^* 
  TS^{2N+1};e^{d|s|}ds)\oplus V_{\op}\oplus V_{\up}, \\ 
 \cL^{p,d} & := & L^p(u^*T\widehat 
  W;e^{d|s|}dsd\theta) \oplus L^p(\lambda^* 
  TS^{2N+1};e^{d|s|}ds). 
\end{eqnarray*} 
Here we identify $V_{\op}$, $V_{\up}$ with the $1$-dimensional spaces  
generated by the sections $\beta(s)(\dot\og,X_{\olambda})$, respectively $\beta(-s)(\dot\ug,X_{\ulambda})$ 
 of $u^*T\widehat W\oplus \lambda^*TS^{2N+1}$. For this identification, we denote by $X_{\olambda}$, $X_{\ulambda}$ the values of the infinitesimal generator of the $S^1$-action on $S^{2N+1}$  
 at the points $\olambda$, respectively $\ulambda$, and choose a cut-off function $\beta:\R\to [0,1]$ which is equal to $1$ near $-\infty$, and vanishes near $+\infty$.  
 
\begin{proposition} \label{prop:indexMB} 
Assume $S_\op,S_\up\subset \cP(H)$ are nondegenerate. For any 
$(u,\lambda)\in \widehat \cM(S_\op,S_\up;H,J,g)$ the operator  
$$ 
D_{(u,\lambda)}: \cW^{1,p,d} \to \cL^{p,d} 
$$ 
is Fredholm.
\end{proposition}

\begin{proof} 
Let $\cW^{1,p}$ and $\cL^p$ be defined as $\cW^{1,p,d}$ and
$\cL^{p,d}$ above, with $d=0$ and without taking into account the
direct summands $V_{\op}$, $V_{\up}$. Let 
$\widetilde D_{(u,\lambda)}:\cW^{1,p}\to \cL^p$ 
be the operator obtained by
conjugating with $e^{\frac d p |s|}$ the restriction of
$D_{(u,\lambda)}$ to   
$W^{1,p}(u^*T\widehat   W;e^{d|s|}dsd\theta) \oplus W^{1,p}(\lambda^* TS^{2N+1};e^{d|s|}ds)$.  
Then $\widetilde D_{(u,\lambda)}$ has nondegenerate asymptotics, hence
it is Fredholm by Theorem~\ref{thm:Fredholm} (the asymptotic operator
at $-\infty$ is $\widetilde D_{\op}=D_{\op}+\frac d p \one$, and the
asymptotic operator at $+\infty$ is $\widetilde D_{\up}=D_{\up}-\frac
d p \one$).  It follows that the operator 
$D_{(u,\lambda)}$ is also Fredholm. 
\end{proof}  
 

\section{Unique continuation in the \boldmath$S^1$-invariant case} \label{sec:uniqueS1}

The purpose of this section is to prove a unique continuation result which is slightly more general than the one in Section~\ref{sec:unique}.
This is needed in the proof of Theorem~A\,\ref{item:INTROb}.

\medskip 

\noindent {\bf Notation.} We denote by $X$ the
infinitesimal generator of the $S^1$-action on the parameter space $\Lambda=S^{2N+1}$.  
We denote by $\cH^{S^1}_N\subset
\cH_{S^{2N+1}}$ the set of admissible    
Hamiltonian families $H:S^1\times \widehat W\times S^{2N+1}\to \R$ 
which are invariant with respect to the diagonal $S^1$-action on
$S^1\times S^{2N+1}$. We denote by $\cJ_N^{S^1}$ 
the set of pairs $(J,g)$ consisting of an $S^1$-invariant admissible 
$S^{2N+1}$-family of almost complex structures $J$ on $\widehat W$ and of an 
$S^1$-invariant Riemannian metric $g$ on $S^{2N+1}$.  

\medskip 

\begin{definition} \label{defi:tilde}
Given $H\in\cH^{S^1}_N$, we define $\tH:\widehat W\times
S^{2N+1}\to \R$ by 
$$
\tH(x,\lambda):=H(0,x,\lambda).
$$
Given an $S^1$-invariant almost complex structure
$J=(J^\theta_\lambda)$, we define 
$$
\tJ_\lambda(x):=J^0_\lambda(x).
$$
Given maps $u:\R\times S^1\to \widehat W$, $\lambda:\R\to S^{2N+1}$,
we define maps $\tlambda:\R\times S^1\to S^{2N+1}$ and $\tU:\R\times
S^1\to \widehat W\times S^{2N+1}$ by 
$$
\tlambda(s,\theta):=(-\theta)\cdot \lambda(s),\qquad 
\tU(s,\theta):=(u(s,\theta),\tlambda(s,\theta)). 
$$
\end{definition} 

It follows from the definitions that the pair $(u,\lambda)$ satisfies
equations~(\ref{eq:Floer1par}--\ref{eq:asymptoticpar}) 
if and only if $\tU=(u,\tlambda)$ satisfies the equations
\begin{eqnarray}
\label{eq:tFloer1}
\p_s u + \tJ_{\tlambda} (\p_\theta u - 
X_{\tH_{\tlambda}} (u)) & = & 0, \\ 
\label{eq:tFloer2} 
\p_s\tlambda - \int_{S^1} \tau_*\vec \nabla_\lambda 
\tH(\tU(s,\theta+\tau)) \, d\tau & = & 0, \\
\label{eq:tFloer3}
\p_\theta\tlambda + X_{\tlambda} & = & 0,
\end{eqnarray} 
and
\begin{equation} \label{eq:tasymptotic} 
 \lim_{s\to -\infty} \tU(s,\theta)
 =(\og(\theta),(-\theta)\cdot \olambda), \qquad
 \lim_{s\to +\infty} \tU(s,\theta)
 =(\ug(\theta),(-\theta)\cdot\ulambda)
\end{equation} 
for all $\theta\in S^1$. 
We note that equations~(\ref{eq:tFloer1}--\ref{eq:tFloer3}) are
independent of the variables $s$ and $\theta$. 

\begin{proposition}[Unique continuation] \label{prop:tunique} 
  Let $h>0$ and $\tU_i=(u_i,\tlambda_i):Z_h \to \widehat W\times S^{2N+1}$,
  $i=0,1$ be smooth functions
  satisfying~(\ref{eq:tFloer1}-\ref{eq:tFloer3}).  
If $\tU_0=\tU_1$ on some
nonempty open set $\cU\subset Z_h$, then $\tU_0=\tU_1$ on 
$Z_h$. 
\end{proposition} 

To prove Proposition~\ref{prop:tunique}, we need the following
enhancement of Proposition~\ref{prop:vanish}. 

\begin{proposition} \label{prop:tvanish}
Let $h>0$ and $\tU=(u,\tlambda):Z_h\to \C^n\times \R^{2N+1}$ be
$C^\infty$-functions satisfying 
\begin{equation} \label{eq:teqCn}
\left\{\begin{array}{rcl} 
\p_s u + J(s,\theta) \p_\theta u + C(s,\theta) u + D(s,\theta) \lambda
& = & 0, \\
\p_s \tlambda + \int_{S^1} E(s,\theta,\tau) \tU(s,\tau) \, d\tau & = &
0, \\
\p_\theta\tlambda + F(s,\theta)\tlambda & = & 0, 
\end{array}\right.
\end{equation}
with $C,D,E,F$ of class $C^1$, $J$ of class $C^\infty$ and
$J^2=-\one$. Assume there exists a nonempty open set $\cU\subset Z_h$
such that $\tU(s,\theta)=(0,0)$ for all $(s,\theta)\in \cU$. 
Then $\tU\equiv (0,0)$ on $Z_h$. 
\end{proposition}

\begin{proof} 
We first remark that $\tU$ must vanish on some strip
$]s_0-\eps,s_0+\eps[\times S^1\subset Z_h$. More precisely, let us
choose 
$(s_0,\theta_0)\in\cU$ and $\eps>0$ such that
$]s_0-\eps,s_0+\eps[\, \times \, ]\theta_0-\eps,\theta_0+\eps[\subset \cU$.
Then, for $s\in]s_0-\eps,s_0+\eps[$, we have that  $\tlambda(s,\cdot)$
solves a linear ODE on $S^1$ and vanishes at $\theta_0$, hence vanishes
identically. Thus $u$ solves $\p_s u +J\p_\theta u +Cu=0$ on
$]s_0-\eps,s_0+\eps[ \times S^1$, and therefore must also vanish
identically by the standard unique 
continuation property~\cite[Theorem~2.2,\, Proposition~3.1]{FHS}. In
particular, $\tU$ vanishes with all its derivatives along
$\{s_0\}\times S^1$. 

As in the proof
of Proposition~\ref{prop:vanish}, we can assume without loss of
generality that $J=i$. Let us denote
$$
\tV(s,\theta):=(u(s,\theta),\tlambda(s,\theta),0)\in\C^n\times
\C^{2N+1}.
$$
Then $\tV$ satisfies an equation of the form 
$$
\p_s \tV + i\p_\theta \tV + A(s,\theta)\tV + \int_{S^1}
B(s,\theta,\tau)\tV(s,\tau)\, d\tau =0,
$$
with $A,B$ of class $C^1$. As in
Proposition~\ref{prop:vanish}, we infer an inequality 
$$
|\Delta \tV|^2 \le 8\pi K\Big[ |\tV|^2 + |\nabla \tV|^2 + \int_{S^1}
|\tV|^2 \Big].
$$
The only difference with respect to Proposition~\ref{prop:vanish} is that the
function $A_4$ therein depends now on $\theta$. However, it is still
pointwise bounded and the same argument carries through. 

The conclusion follows from Theorem~\ref{thm:inequality}.
\end{proof} 

\begin{proof}[Proof of Proposition~\ref{prop:tunique}] The proof
  follows the same pattern as that of Proposition~\ref{prop:unique},
  and makes use of Proposition~\ref{prop:tvanish}. 

Let us assume without loss of generality that $\cU=I_\delta\times
I_\eps$, for some $\delta,\eps>0$. Since $\tlambda_0(s,\cdot)$ and
$\tlambda_1(s,\cdot)$ solve the same ODE on $S^1$ and coincide on
$I_\eps$, we infer that they coincide on $S^1$, for all $s\in
I_\delta$. By the unique continuation property for the Floer
equation~\cite[Proposition~3.1]{FHS}, we infer that $u_0=u_1$ 
on the strip $I_\delta\times S^1$.

Let $I\subset I_h$ be the set of points $s$ such that $\tU_0=\tU_1$ on
$\{s\}\times S^1$. Then $I$ is nonempty (it contains $I_\delta$),
closed, and we must prove that it is open. Let $s_0\in I$ be a point
on the boundary of a connected component of $I$ with nonempty
interior, and denote $\gamma:=\tU_0(s_0,\cdot)=\tU_1(s_0,\cdot)$. We
consider a trivialization of $\gamma^*(T\widehat W\times TS^{2N+1})$ of
the form $S^1\times \C^n\times \R^{2N+1}$. Then, for $s$ close to
$s_0$, we can view $\tU_0(s,\cdot)$ and $\tU_1(s,\cdot)$ as taking
values in $\C^n\times \R^{2N+1}$. The difference
$\tU := (u,\tlambda):=(u_0-u_1,\tlambda_0-\tlambda_1)$ satisfies an equation
of the form~\eqref{eq:teqCn} with smooth coefficients. The computation
is similar to the one in~\cite[Proposition~3.1]{FHS}, and we just
establish the second equation in~\eqref{eq:teqCn}. We have, for
suitable matrices $\widehat E$ and $E$, 
\begin{eqnarray*}
  \p_s\tlambda(s,\theta) & = & \int_{S^1} (\tau-\theta)_*
  \Big[\vec\nabla_\lambda \tH(\tU_0(s,\tau)) - 
\vec\nabla_\lambda \tH(\tU_1(s,\tau))\Big] \, d\tau \\
& = & \int_{S^1} (\tau-\theta)_* \widehat E(s,\tau)\tU(s,\tau)\, d\tau
  \\
&=& \int_{S^1} E(s,\theta,\tau)\tU(s,\tau)\, d\tau.
\end{eqnarray*}

The conclusion follows from Proposition~\ref{prop:tvanish}. 
\end{proof}


\section{Transversality in the \boldmath$S^1$-invariant case} \label{sec:transvS1}

We prove in this section that transversality for the $S^1$-invariant
Floer equations can be achieved within the following two classes of
Hamiltonians.  
\renewcommand{\theenumi}{{\bf \Alph{enumi}}}
\begin{enumerate}
\item {\bf Generic Hamiltonians}. We require such Hamiltonians $H$ to
  be admissible, regular, and to satisfy the following two conditions: 
\begin{itemize}
\item for all $(\gamma,\lambda)\in\cP(H)$ we have that $\gamma$ is a simple embedded curve;
\item for all distinct elements
  $(\gamma_1,\lambda_1),(\gamma_2,\lambda_2)\in\cP(H)$ we have
  $\gamma_1\neq \gamma_2$.  
\end{itemize}
We denote the class of generic Hamiltonians by $\cH_{\mathrm{gen}}$.
\item {\bf Split Hamiltonians}. We require such Hamiltonians to be
  admissible and of the form $K(x)+f(\lambda)$, with $K$ being
  $C^2$-small on $W$. Here $f$ is
  $S^1$-invariant and $K$ has either constant and nondegenerate
  $1$-periodic orbits, or nonconstant and transversally nondegenerate
  ones.  

We denote the class of split Hamiltonians by $\cH_{\mathrm{split}}$.
\end{enumerate}

We denote 
$$
\cH_*:=\cH_{\mathrm{gen}}\, \cup \, \cH_{\mathrm{split}}.
$$
\renewcommand{\theenumi}{\arabic{enumi}}
\begin{definition} \label{defi:strongH}
An admissible Hamiltonian $H\in\cH^{S^1}_N$ is called {\bf strongly
  admissible} if the following two conditions hold: 
\begin{enumerate} 
\item \label{item:strongH1} for every $(\gamma,\lambda)\in
  \cP(H)$ such that $\gamma$ is not constant, we have
  $$
X^\theta_{H_\lambda}(\gamma(\theta))\neq 0, \qquad \forall \, \theta\in
  S^1.
 $$
\item \label{item:strongH2} for every $(\gamma,\lambda)\in\cP(H)$ such
  that $\gamma$ is constant (equal to $x\in \widehat W$), there exists
  a neighbourhood $\cU$ of $\{x\}\times (S^1\cdot\lambda)$ in
  $\widehat W\times S^{2N+1}$ such that
  $H(\theta,x',\lambda')=K(x')+f(\lambda')$ for all $\theta\in S^1$
  and $(x',\lambda')\in \cU$.  Moreover, $x$ is an isolated critical
  point of $K$.
\end{enumerate}
We denote by $\cH'$ the class of strongly admissible Hamiltonians. 
\end{definition} 

We clearly have 
$$
\cH_*\subset \cH'.
$$

\renewcommand{\theenumi}{\arabic{enumi}}
\begin{definition} \label{defi:adaptedJ} Given $H\in \cH'$, an almost complex structure $J\in \cJ^{S^1}_N$ is called {\bf adapted
    to \boldmath$H$} if the following hold: 
\begin{enumerate} 
\item \label{item:adaptedJ1} for every
    $(\gamma,\lambda_0)\in \cP(H)$, we have 
$$
[J^\theta_\lambda X^\theta_{H_\lambda},
X^\theta_{H_\lambda}](\gamma(\theta)) \notin
\mathrm{Span}(J^\theta_\lambda X^\theta_{H_\lambda},
X^\theta_{H_\lambda}), \qquad \forall \, \theta\in S^1, \, \lambda\in
S^1\cdot \lambda_0.
$$
\item \label{item:adaptedJ2} for every $(\gamma,\lambda_0)\in \cP(H)$
  such that $\gamma$ is constant (equal to $x\in\widehat W$), there
  exists a neighbourhood $\cU$ of $\{x\}\times (S^1\cdot\lambda)$ in
  $\widehat W\times S^{2N+1}$ such that $J^\theta_\lambda$ is
  independent of $\theta$ and $\lambda$ on $\cU$, i.e. 
  $J^\theta_\lambda(x')=J(x')$ for all $(x',\lambda)\in\cU$ and
  $\theta\in S^1$. 
\end{enumerate} 
We denote by $\cJ'(H)\subset \cJ^{S^1}_N$ the set of
almost complex structures adapted to $H$.  
\end{definition} 

\begin{remark}
The set $\cJ'(H)$ is nonempty for every choice of strongly admissible
Hamiltonian $H$. This is proved by a
genericity argument: given a nonzero vector field $X$ along a curve,
one can choose generically a nonzero vector field $Y$ which is
linearly independent from $X$ along the same curve, and such that the
distribution spanned by $X$ and $Y$ is non-involutive and symplectic.  
\end{remark}

We denote 
$$
\cH_*\cJ':=\Bigg\{(H,J,g)\, : \, \begin{array}{l}H\in\cH_*, \ (J,g)\in\cJ'(H), \\
  J \mbox{ admissible, cylindrical for } t\ge 1,\\
  \qquad \mbox{independent of } (\theta,\lambda) \mbox{ if } H\in\cH_{\mathrm{split}} 
  \end{array}\Bigg\},
$$
and
$$
\cH\cJ':=\{(H,J,g)\, : \, H\in\cH', \ (J,g)\in\cJ'(H)\},
$$
so that $\cH_*\cJ'\subset \cH\cJ'$.

\medskip 

Let $H\in\cH^{S^1}_{N,\reg}$. A pair $(J,g)\in \cJ^{S^1}_N$ is called 
{\bf regular for \boldmath$H$} if the operator $D_{(u,\lambda)}$ is surjective 
for any $\op,\up\in\cP(H)$ and any $(u,\lambda)\in \widehat 
\cM(\op,\up;H,J,g)$. We denote the set of such regular pairs by 
$\cJ^{S^1}_{N,\reg}(H)$. 

\medskip 

The next result proves part~\ref{item:INTROb} of Theorem~A. 

\begin{theorem} \label{thm:transvS1}
There exists an open subset $\cH\cJ'_{\mathrm{reg}}\subset \cH\cJ'$
which is dense in a neighbourhood of $\cH_*\cJ'\subset \cH\cJ'$ and
consisting of triples $(H,J,g)$ such that  
$$
H\in\cH^{S^1}_{N,\mathrm{reg}},\qquad (J,g)\in\cJ^{S^1}_{N,\mathrm{reg}}(H).
$$
\end{theorem}

\begin{remark}[on symplectic asphericity] \label{rmk:sympasph}
 The previous theorem can be rephrased by
 saying that we can achieve transversality within the special class of adapted almost
 complex structures, after possibly perturbing a Hamiltonian which is
 either generic (in the sense that it belongs to
 $\cH_{\mathrm{gen}}$), or split (in the sense that it belongs to
 $\cH_{\mathrm{split}}$). We would like to draw the reader's attention
 to the fact that, in the case of split Hamiltonians, our proof
 uses the assumption that $W$ is symplectically aspherical and the
 Hamiltonian is $C^2$-small on $W$. For
 generic Hamiltonians, these assumptions are not used.  
\end{remark}

We denote in this section by $X$ the infinitesimal generator of the
$S^1$-action on $S^{2N+1}$. Also, we make extensive use of the
notation $\tH$, $\tJ$, $\tU$ introduced in
Definition~\ref{defi:tilde}, and of the fact that $(\tH,\tJ,\tU=(u,\tlambda))$
solve~(\ref{eq:tFloer1}--\ref{eq:tasymptotic}) if and only if
$(H,J,U=(u,\lambda))$
solve~(\ref{eq:Floer1par}--\ref{eq:asymptoticpar}).  

\medskip  

Our first result is an analogue of Lemma~\ref{lem:unique}. We recall
the notation $V_h(s_0,\theta_0):=]s_0-h,s_0+h[ \, \times \,
]\theta_0-h,\theta_0+h[$. 

\begin{lemma} \label{lem:tunique} 
Let $H\in\cH^{S^1}_{N,\mathrm{reg}}\cap \cH'$ and $(J,g)\in\cJ'(H)$. 
Let $\op=(\og,\olambda),\up=(\ug,\ulambda)\in\cP(H)$ and 
$\tU_i=(u_i,\tlambda_i):\R\times S^1\to \widehat W\times S^{2N+1}$,
$i=0,1$ be solutions of
equations~(\ref{eq:tFloer1}--\ref{eq:tasymptotic}).  

We assume that, for some $(s_0,\theta_0)\in \R\times S^1$,  
$$
\tU_0(s_0,\theta_0)=\tU_1(s_0,\theta_0) \ \mbox{ and } \  du_0(s_0,\theta_0), \
d\tU_1(s_0,\theta_0)  \mbox{ are injective.} 
$$
We also assume there exists $h_0>0$ such that, for all $0<h'\le h_0$,
there exists $h>0$ with the following property: for any 
$(s,\theta)\in V_h(s_0,\theta_0)$, there exists
$(s',\theta')\in V_{h'}(s_0,\theta_0)$ such that 
$$
\tU_0(s,\theta)=\tU_1(s',\theta').
$$

There exists a neighbourhood $\cU\subset \widehat W\times S^{2N+1}$ of
$\ug(S^1)\times (S^1\cdot\ulambda)$, independent of $\tU_0$ and
$\tU_1$, such that,  
if $\tU_0(s_0,\theta_0)\in\cU$, the above assumptions imply
$\tU_0=\tU_1$ (the same holds for the asymptote at $-\infty$). 
\end{lemma} 

\begin{proof} By Proposition~\ref{prop:tunique}, it is enough to prove
  that $\tU_0$ and $\tU_1$ coincide on some open
  neighbourhood of $(s_0,\theta_0)$. Let us choose $0<h'\le h_0$ small
  enough so that $\tU_1:V_{h'}(s_0,\theta_0)\to \widehat W\times
  S^{2N+1}$ is an embedding. Upon further diminishing the
  corresponding $h>0$, we can assume that $u_0:V_h(s_0,\theta_0)\to
  \widehat W$ is also an embedding. 

  By assumption we have $\tU_0(V_h(s_0,\theta_0))\subset
  \tU_1(V_{h'}(s_0,\theta_0))$. We can therefore define a smooth
  embedding $G:=(\tU_1)^{-1}\circ \tU_0:V_h(s_0,\theta_0)\to
  V_{h'}(s_0,\theta_0)$. Moreover, we have by assumption that
  $(s_0,\theta_0)=G(s_0,\theta_0)\in\mathrm{im}(G)$. There exists
  therefore $0<h''<h'$ such that $V_{h''}(s_0,\theta_0)\subset
  \mathrm{im}(G)$. By the implicit function theorem, we obtain a
  smooth embedding $F:=G^{-1}:=(\phi,\psi):V_{h''}(s_0,\theta_0)\to
  V_h(s_0,\theta_0)$. It follows from the definition that 
$$
\tU_1(s,\theta)=\tU_0(\phi(s,\theta),\psi(s,\theta))
$$
for all $(s,\theta)\in V_{h''}(s_0,\theta_0)$. Substituting this relation
into~\eqref{eq:tFloer1} for $u_1$ we obtain 
\begin{eqnarray}
 0 & = & \p_s u_1 (s,\theta) + \tJ_{\tlambda_1}(u_1)\big(\p_\theta u_1
 (s,\theta) - X_{\tH_{\tlambda_1}}(u_1) \big) \nonumber \\ 
 & = & \p_su_0(F) \p_s \phi + \p_\theta u_0(F) \p_s \psi \nonumber \\
 &  & + \ \tJ_{\tlambda_0(F)}(u_0(F))\big[\p_s
 u_0(F)  \p_\theta \phi + \p_\theta u_0(F) \p_\theta\psi -
 X_{\tH_{\tlambda_0(F)}}(u_0(F))\big]
 \nonumber \\ 
 & = & \p_su_0(F) \p_s \phi + \p_\theta u_0(F) \p_s \psi + 
\tJ_{\tlambda_0(F)}\Big[\tJ_{\tlambda_0(F)}(-\p_\theta u_0 +
 X_{\tH_{\tlambda_0(F)}})  \p_\theta \phi \nonumber \\
 &  & + \ 
 (\tJ_{\tlambda_0(F)}\p_s u_0 + X_{\tH_{\tlambda_0(F)}})
 \p_\theta\psi - X_{\tH_{\tlambda_0(F)}}\Big] \nonumber \\ 
 & = & (\p_s\phi-\p_\theta\psi)\p_su_0(F) + (\p_s\psi +
 \p_\theta\phi)\p_\theta u_0(F) - \p_\theta\phi \,
 X_{\tH_{\tlambda_0(F)}}(u_0(F)) \nonumber \\ 
 & &  -
 (1-\p_\theta\psi)\,  \tJ_{\tlambda_0(F)}(u_0(F))
 X_{\tH_{\tlambda_0(F)}}(u_0(F)). \label{eq:tdu} 
\end{eqnarray}
The third equality uses the Floer equation~\eqref{eq:tFloer1} for
$(u_0, \tlambda_0)$. 

By Definition~\ref{defi:adaptedJ}, we can choose a neighbourhood
$\cU\subset \widehat W\times S^{2N+1}$ of $\ug(S^1)\times(S^1\cdot
\ulambda)$ such that  
$$
\big[\tJ_\lambda
X_{\tH_\lambda},X_{\tH_\lambda}\big](x)
\notin \mathrm{Span}\big(\tJ_\lambda(x)
X_{\tH_\lambda}(x),X_{\tH_\lambda}(x)\big)
$$
for all $(x,\lambda)\in \cU$. 

Up to further diminishing $h>0$, we can assume that
$\tU_0(V_h(s_0,\theta_0))\subset \cU$.  We now claim that the four
vectors 
$\p_s u_0$, $\p_\theta u_0$, $X_{\tH_{\tlambda_0}}(u_0)$,
$\tJ_{\tlambda_0}(u_0) X_{\tH_{\tlambda_0}}(u_0)$ are linearly
independent on an open dense subset of $V_h(s_0,\theta_0)$. This
follows from the argument in~\cite[Lemma~7.7]{FHS}. More precisely,
assume by contradiction the existence of a 
nonempty open subset $\Omega\subset
V_h(s_0,\theta_0)$, such that these four 
vectors are linearly dependent on $\Omega$.

Let us first use the assumption of strong admissibility on $H$. 
Since $u_0$ is an embedding on 
$V_h(s_0,\theta_0)$, we can further assume, after slightly
moving the base point $(s_0,\theta_0)$, that
$u_0(V_h(s_0,\theta_0))$ 
does not intersect the geometric image of $\ug$. Also, by assumption,
$\tlambda_0$ is close to $S^1\cdot \ulambda$, and therefore 
$X_{\tH_{\tlambda_0}}(u_0)\neq 0$ on $V_h(s_0,\theta_0)$. 

On the other hand, by assumption the vectors $\p_su_0$ and $\p_\theta
u_0$ are linearly independent on $V_h(s_0,\theta_0)$. Let us use the
shorthand notation $\tJ=\tJ_{\tlambda_0}$ and
$X_\tH=X_{\tH_{\tlambda_0}}$. Since $\p_\theta u_0 = \tJ\p_s u_0 +
X_\tH(u_0)$, the linear dependence of the above four vectors
on $\Omega$ is equivalent to the linear dependence of $\p_s u_0,
\tJ\p_s u_0, X_\tH, \tJ X_\tH$. This in turn implies that 
$\p_s u_0\in \mathrm{Span}(\tJ X_\tH,X_\tH)$, 
i.e. there exist smooth functions $a,b:\Omega\to \R$ such that 
$$
\p_s u_0 = a \tJ X_\tH(u_0) + bX_\tH(u_0).  
$$
From $\p_\theta u_0=\tJ\p_s u_0 + X_\tH$ we also obtain 
$$
\p_\theta u_0 = b \tJ X_\tH(u_0) + (1-a)
X_\tH(u_0). 
$$
We now use the fact that $[\p_su_0,\p_\theta u_0]=0$ on $\Omega$ to obtain 
\begin{equation} \label{eq:tbracket}
(a^2+b^2-a)[\tJ X_\tH,X_\tH]=(\p_\theta a - \p_s b) \tJ X_\tH + 
(\p_s a + \p_\theta b) X_\tH.
\end{equation}
Note that the linear independence of $\p_s u_0$ and $\p_\theta u_0$ is
equivalent to the condition $a^2+b^2-a\neq 0$. We infer that, for all
$(s,\theta)\in \Omega$, we have 
$$
\big[\tJ_{\tlambda_0(s,\theta)}
X_{\tH_{\tlambda_0(s,\theta)}},X_{\tH_{\tlambda_0(s,\theta)}}\big]
(u_0(s,\theta)) \in \mathrm{Span}(\tJ_{\tlambda_0(s,\theta)}
X_{\tH_{\tlambda_0(s,\theta)}},X_{\tH_{\tlambda_0(s,\theta)}}).
$$
This contradicts our choice of $\cU$. We thus proved that the four vectors 
$\p_s u_0$, $\p_\theta u_0$, $X_{\tH_{\tlambda_0}}(u_0)$,
$\tJ_{\tlambda_0}(u_0) X_{\tH_{\tlambda_0}}(u_0)$ are linearly
independent on an open dense subset  $\cV\subset V_h(s_0,\theta_0)$.

Since $F:V_{h''}(s_0,\theta_0)\to V_h(s_0,\theta_0)$ is an embedding,
we infer that $F^{-1}(\cV)$ is open and
nonempty. Equation~\eqref{eq:tdu} now implies that $\p_s \phi - \p_\theta \psi=0$,
$\p_s \psi + \p_\theta \phi=0$, $\p_\theta\phi=0$, and
$1-\p_\theta\psi=0$, so that $F(s,\theta)=(s+\os,\theta+\otheta)$ on
$F^{-1}(\cV)$ for suitable constants $\os$ and $\otheta$. Since
$F(s_0,\theta_0)=(s_0,\theta_0)$, we must have 
$\os=0$ and $\otheta=0$, and therefore $\tU_1=\tU_0$ on
$F^{-1}(\cV)$. This concludes the proof.  
\end{proof}

\bigskip 

The next Lemma is the analogue of Lemma~7.6 in~\cite{FHS}. 
 
\begin{lemma} \label{lem:indep} 
Let $\tU=(u,\tlambda)$ be a solution of
(\ref{eq:tFloer1}-\ref{eq:tasymptotic}), and assume $\p_s\tU\not\equiv
(0,0)$.   

(i) If $du\not \equiv 0$, the set of points $(s,\theta)\in
\R\times S^1$ such that  
\begin{equation} \label{eq:vectors}
\p_s\tU(s,\theta) \quad \mbox{and} \quad 
\p_\theta\tU(s,\theta) 
\end{equation}
are linearly independent is open and dense in $\{(s,\theta)\, : \,
du(s,\theta)\neq 0\}$.  

(ii) If $du \equiv 0$, the vectors are linearly independent on
$\R\times S^1$.   
\end{lemma} 
 
 \begin{remark}
The key point in the statement is that $\p_s \tU$ and $\p_\theta \tU$
lie in the kernel of the operator which linearizes equations (\ref{eq:tFloer1}--\ref{eq:tFloer3}),
since the latter are independent of $s$ and $\theta$. Equivalently, $(\p_s u, \dot\lambda)$
 and $(\p_\theta u, -X)$ lie in the kernel of the linearized operator $D_{(u,\lambda)}$.
 \end{remark}

 \begin{proof}
 Openness is clear by continuity of $\p_s \tU$ and $\p_\theta \tU$. To prove density,
 we argue by contradiction and assume that $\p_s\tU$ and
 $\p_\theta\tU$ 
 are linearly dependent on some open set $\Omega\subset \R\times
 S^1$. By Proposition~\ref{prop:regularpts}(i), the set of points
 where $\p_s \tU\neq (0,0)$ is open and dense in $\R\times S^1$, so that
 we can assume without loss of generality that $\p_s\tU\neq (0,0)$ on
 $\Omega$. We thus find a smooth function $\mu:\Omega\to \R$ such that 
 $\p_\theta \tU(s,\theta) = \mu(s,\theta) \p_s \tU(s,\theta)$
 for all $(s,\theta)\in\Omega$. More explicitly,
 $$ 
 (\p_\theta u(s,\theta), (-\theta)_* X_{\lambda(s)}) = \mu(s,\theta) 
 (\p_s u(s,\theta), (-\theta)_* \p_s \lambda(s)) .
 $$
 Since $X\neq 0$, we infer $\mu\neq
 0$. Since $X_{\lambda(s)}$ and $\p_s\lambda(s)$ do not depend on
 $\theta$, we infer that the same holds for $\mu$,  so that we have $$ 
 \mu(s,\theta)=\mu(s).
$$

In the computations that follow we denote total derivatives by $d$,
and partial derivatives by $\p$. 

We compute
\begin{eqnarray*}
d_s(\tH\circ \tU) & = & \p_x \tH \cdot \p_s u + \p_\lambda \tH \cdot \p_s\tlambda \\
& = & \omega (X_\tH, \p_s u ) + \p_\lambda \tH \cdot \p_s \tlambda \\
& = & \omega (-\tJ\p_s u +\p_\theta u,\p_s u) + \p_\lambda \tH \cdot
\p_s \tlambda \\
& = & |\p_s u |^2 + \p_\lambda \tH \cdot \p_s \tlambda.
\end{eqnarray*}
For the third equality we used the Floer equation \eqref{eq:tFloer1}
for $u$. We also have 
 \begin{eqnarray*}
d_\theta(\tH\circ \tU) & = & d\tH \cdot \p_\theta \tU \\
& = & \mu \, d\tH \cdot \p_s \tU \\
& = & \frac 1 \mu |\p_\theta u |^2 +\p_\lambda \tH \cdot \p_\theta \tlambda.
\end{eqnarray*}
We now compute crossed second derivatives. 
\begin{eqnarray*}
d_\theta d_s(\tH\circ \tU) & = & 2\langle \nabla_\theta \p_s u,\p_s u\rangle + 
\omega(\p_s u,(\p_\lambda \tJ\cdot \p_\theta \tlambda) \cdot \p_s u) \\
& & + \ \p_\theta ( \p_\lambda \tH \cdot \p_s \tlambda) \\ 
d_s d_\theta (\tH\circ \tU) & = & -\frac {\mu'} {\mu^2} |\p_\theta u|^2 + 
\frac 2 \mu \langle \nabla_s\p_\theta u,\p_\theta u\rangle + \frac 1
\mu \omega (\p_\theta u,(\p_\lambda \tJ\cdot \p_s \tlambda)\cdot
\p_\theta u) \\ 
& & + \ \p_s(\p_\lambda \tH \cdot \p_\theta \tlambda). 
\end{eqnarray*}
The equality $d_\theta d_s(\tH\circ \tU)=d_s d_\theta (\tH\circ \tU)$ implies
$(\mu'/\mu^2)|\p_\theta u|^2=0$. Indeed, we have  
$$
2\langle \nabla_\theta \p_s u,\p_s u\rangle = 2\langle \nabla_s
\p_\theta u,\p_s u\rangle =  
\frac 2 \mu \langle \nabla_s\p_\theta u,\p_\theta u\rangle
$$
because $\nabla_\theta \p_s u=\nabla_s \p_\theta u$, we have 
\begin{eqnarray*}
 \omega(\p_s u,(\p_\lambda \tJ\cdot\p_\theta\tlambda)\cdot \p_s u)   
 = \frac 1 \mu \omega (\p_\theta u,(\p_\lambda \tJ\cdot \p_s
 \tlambda)\cdot \p_\theta u)
\end{eqnarray*}
because $\p_\theta u = \mu \p_s u$, $\p_\theta \tlambda = \mu \p_s \tlambda$, and we similarly have
\begin{eqnarray*}
\p_\theta ( \p_\lambda \tH \cdot \p_s \tlambda) 
& = &
\nabla_{\p_\theta \tlambda} (\p_\lambda \tH) \cdot \p_s \tlambda
+ (\p_\lambda \tH) \cdot \nabla_\theta \p_s \tlambda
+ \p_x \p_\lambda \tH \cdot (\p_\theta u, \p_s \tlambda) \\
& = &
\nabla_{\p_s \tlambda} (\p_\lambda \tH) \cdot \p_\theta \tlambda
+ (\p_\lambda \tH) \cdot \nabla_s \p_\theta \tlambda
+ \p_x \p_\lambda \tH \cdot (\p_s u, \p_\theta \tlambda) \\
& = & \p_s(\p_\lambda \tH \cdot \p_\theta \tlambda).
\end{eqnarray*}
Thus 
\begin{equation} \label{eq:muprime}
({\mu'} /{\mu^2}) |\p_\theta u|^2 = \mu' |\p_s u|^2=0.
\end{equation}

We now prove (i). In this case we have
$\p_s u\neq 0$ or $\p_\theta u\neq 0$ on
$\Omega$. Then~\eqref{eq:muprime} implies $\mu'=0$, so that 
$\mu$ is constant on $\Omega$.  We now claim that  
\begin{equation} \label{eq:intermediate}
\tU(s-\mu\tau,\theta+\tau)=\tU(s,\theta) 
\end{equation}
for $\tau$ sufficiently close to $0$ and $(s,\theta)$ in some nonempty
open subset of $\Omega$. Indeed, this clearly holds for $\tau=0$ and
the derivative of the left hand side with respect to $\tau$ is given
by  
$$
-\mu\p_s \tU + \p_\theta \tU = 0. 
$$
Both sides in~\eqref{eq:intermediate} define solutions
of~(\ref{eq:tFloer1}--\ref{eq:tasymptotic}), 
and they must coincide by
the unique continuation property (Proposition~\ref{prop:tunique}). In
particular, their asymptotes must also coincide. This leads to a
contradiction, since the asymptote (say at $-\infty$) of the left term
in~\eqref{eq:intermediate} is $(-\tau)\cdot\op$, which is
different from the asymptote $\op$ for small $\tau\neq 0$,
due to the fact that $S^1$ acts freely on $S^{2N+1}$. 

We now prove~(ii). In this case we have $\p_s u
=\p_\theta u = 0$ on $\R\times S^1$, so that $u(s,\theta)\equiv
x$. We have 
\begin{eqnarray*}
0 \ = \ \frac {d} {d\theta} \int _{S^1} \tH(x,\tlambda(s,\theta+\tau))d\tau
& = & \int _{S^1}
\vec\nabla_\lambda\tH(x, \tlambda(s,\theta+\tau))\cdot
X_{\tlambda(s,\theta+\tau)}d\tau \\
& = & \int _{S^1}
\tau_*\vec\nabla_\lambda\tH(x, \tlambda(s,\theta+\tau))d\tau  \cdot
X_{\tlambda(s,\theta)}.
\end{eqnarray*}
For the last equality, we used that
$X_{\tlambda(s,\theta)}=\tau_*X_{\tlambda(s,\theta+\tau)}$. 
Assuming by contradiction that $-X_{\tlambda(s,\theta)}=\mu(s) \p_s \tlambda
(s,\theta)$ at some point $(s,\theta)\in \R\times S^1$, we obtain
$0=\p_s\tlambda(s,\theta) \cdot 
X_{\tlambda(s,\theta)}=-(1/\mu(s))\|X_{\tlambda(s,\theta)}\|^2$, which
is impossible.   
 \end{proof}

\begin{definition}Let $H\in \cH^{S^1}_N$. 
 Given maps $u:\R\times S^1\to \widehat W$ and $\lambda:\R\to
 S^{2N+1}$, denote
 $\tU:=(u,\tlambda)$ as in Definition~\ref{defi:tilde}, and assume
 that $\tU$ satisfies the asymptotic conditions~\eqref{eq:tasymptotic}
 for $(\og,\olambda),(\ug,\ulambda)\in\cP(H)$. A point
 $(s_0,\theta_0)\in \R\times S^1$ is called {\bf injective} if  
$$
\tU^{-1}(\tU(s_0,\theta_0))=\{(s_0,\theta_0)\}, \qquad 
d\tU(s_0,\theta_0) \ \mbox{is injective}
$$
and 
\begin{equation} \label{eq:tinjective3}
\tU(s_0,\theta_0) \neq (\og(\theta),(-\theta)\cdot\olambda), \quad
\tU(s_0,\theta_0) \neq (\ug(\theta),(-\theta)\cdot\ulambda), \quad
\forall \, \theta\in S^1.
\end{equation} 
 We denote the set of injective points by $R(\tU)$. 
\end{definition}

\begin{proposition} \label{prop:tinjective} 
Let $H\in\cH^{S^1}_{N,\mathrm{reg}}\cap\cH'$ and $(J,g)\in\cJ'(H)$. 
Let $\op,\up\in\cP(H)$ and 
$\tU=(u,\tlambda):\R\times S^1\to \widehat W\times S^{2N+1}$ be a
solution of~(\ref{eq:tFloer1}--\ref{eq:tasymptotic}) satisfying  
$\p_s\tU\not\equiv(0,0)$. 
For every $R>0$, there exists a nonempty open set $\Omega  \subset \,
[R,\infty[\, \times \, S^1$ consisting of injective points.
\end{proposition}

\begin{proof} 
\medskip
\noindent {\it Step~1. We prove that $R(\tU)$ is open.} 
\medskip 

The second and
  third conditions in the definition of an injective point are clearly
  open, and we must prove that the first one is open as well. Arguing by
  contradiction, there exists $(s_0,\theta_0)\in R(\tU)$, a sequence
  $(s^\nu,\theta^\nu)\to (s_0,\theta_0)$, and a sequence
  $(s^{\prime\nu},\theta^{\prime\nu}) \neq (s^\nu,\theta^\nu)$ such
  that
  $\tU(s^{\prime\nu},\theta^{\prime\nu})=\tU(s^\nu,\theta^\nu)$. Since
  $d\tU(s_0,\theta_0)$ is injective, the sequence
  $(s^{\prime\nu},\theta^{\prime\nu})$ is bounded away from
  $(s_0,\theta_0)$. On the other hand, since $\tU(s_0,\theta_0)$ does
  not belong to any of the asymptotes, it follows that the sequence
  $s^{\prime\nu}$ is bounded. We can therefore extract from
  $(s^{\prime\nu},\theta^{\prime\nu})$ a subsequence converging to
  $(s'_0,\theta'_0)\neq (s_0,\theta_0)$. On the other hand, we must
  have $\tU(s'_0,\theta'_0)=\tU(s_0,\theta_0)$, which contradicts the
  assumption that $(s_0,\theta_0)\in R(\tU)$. 

\medskip
\noindent {\it Step~2. The set $R(\tU)$ is
  dense in 
$$
\tU^{-1}(\cU)\, \cap \, \{(s,\theta)\in \R \times \, S^1 \, : \, du(s,\theta) \
\mbox{injective}\},
$$
where $\cU$ is chosen as in Lemma~\ref{lem:tunique}.
}
\medskip  

  Arguing by contradiction, we find a nonempty open set $\Omega\subset\R \times
  \, S^1$ consisting of non-injective points, and such that
  $du(s,\theta)$ is injective for all $(s,\theta)\in\Omega$. By
  Lemma~\ref{lem:indep}(i) we can assume without 
  loss of generality that $d\tU$ is injective on $\Omega$. This
  implies that the set of points $(s,\theta)$ such that
  condition~\eqref{eq:tinjective3} is satisfied is open and dense in
  $\Omega$, so that we can assume without loss of generality that it
  is satisfied on $\overline \Omega$. Thus, the fact that points in
  $\Omega$ are non-injective is equivalent to 
$$
\forall \, (s,\theta)\in \Omega, \qquad \exists\,
(s',\theta')\neq(s,\theta), \qquad \tU(s',\theta')=\tU(s,\theta).
$$
By further shrinking $\Omega$, we can assume that $\tU|_{\Omega}$ is an
embedding. Following~\cite[Proof of Lemma~7.8]{FHS} we denote 
$$
\Omega':=\{  (s',\theta')\in \, \R\times S^1 \, \setminus \, \Omega\, : \,
\tU(s',\theta')\in \tU(\Omega) \}. 
$$
Since condition~\eqref{eq:tinjective3} is satisfied on $\Omega$, we
infer the existence of some $T>0$ such that $\Omega'\subset [-T,T]\times S^1$.  
We claim now that $\Omega'$ must contain a nonempty open set. To prove
this, consider the map $\Phi:\Omega'\to \Omega$ defined by the
commutative diagram  
\begin{equation} \label{diag:Phi}
\xymatrix{\Omega' \ar[rr]^-\Phi \ar[dr]_-{\tU} & & \Omega \\
& \tU(\Omega) \ar[ur]_-{\tU^{-1}} & 
}
\end{equation} 
This extends to a smooth map on an open neighbourhood of $\Omega'$ (compose
$\tU$ in the target with a projection onto the submanifold
$\tU(\Omega)$, then apply $\tU^{-1}$). If a point
$(s,\theta)\in \Omega$ is a regular value of $\Phi$, then
$d\tU(s',\theta')$ is injective for all $(s',\theta')\in \Omega'$ such that
$\tU(s',\theta')=\tU(s,\theta)$. This implies that a
regular value of $\Phi$ has only a finite number of preimages in
$\Omega'$ (otherwise we could find an accumulation point of
preimages, which would be a preimage at which the condition of
injectivity of $d\tU$ would be violated). By Sard's theorem, we can
choose such a regular value $(s_0,\theta_0)$.
Let $(s_1,\theta_1),\dots,(s_N,\theta_N)$ be the other preimages of
$\tU(s_0,\theta_0)$. As in the proof of
Proposition~\ref{prop:regularpts}, one sees that for any $r>0$ there
exists $\delta>0$ such that  
$$
\forall \, (s,\theta)\in V_{2\delta}(s_0,\theta_0), \quad \exists\,
(s',\theta')\in \bigcup_{j=1}^N V_r(s_j,\theta_j), \quad  
\tU(s,\theta)=\tU(s',\theta').
$$
(if this was not true, one would produce by a compactness argument in
$[-T,T] \times S^1$ a preimage of
$\tU(s_0,\theta_0)$ distinct from $(s_j,\theta_j)$,
$j=0,\dots,N$). Let us define  
$$
\Sigma_j:=\{(s,\theta)\in \overline V_\delta (s_0,\theta_0) \, : \,
\exists \, (s',\theta')\in \overline V_r(s_j,\theta_j),\
\tU(s',\theta')=\tU(s,\theta) \}. 
$$
Then $\Sigma_j$ is closed and $\overline V_\delta
(s_0,\theta_0)=\Sigma_1\cup \dots \cup \Sigma_N$. It follows from
Baire's theorem that some $\Sigma_j$ -- say $\Sigma_1$ -- has nonempty
interior. Then $\tU(\mathrm{int}(\Sigma_1))$ is nonempty 
and open in $\tU(\Omega)$, so that
$\Sigma'_1:=
(\tU|_{V_r(s_1,\theta_1)})^{-1}(\tU(\mathrm{int}(\Sigma_1)))$ 
is nonempty and open in $\Omega'$, which proves our claim. 

Let $(\os,\otheta)\in \mathrm{int}(\Sigma_1)$ and denote
$(\os_1,\otheta_1)\in\Sigma'_1$ the unique preimage of
$\tU(\os,\otheta)$. Let $0<r_1<r$ be such that
$V_{r_1}(\os_1,\otheta_1)\subset \Sigma'_1$, and $0<\delta_1<\delta$
be such that $V_{\delta_1}(\os,\otheta)\subset \Sigma_1$ and
$\tU(V_{\delta_1}(\os,\otheta)) \subset
\tU(V_{r_1}(\os_1,\otheta_1))$. It follows from our construction that,
for all $0<h'\le r_1$, there exists $0<h\le \delta_1$, such that 
$\tU(V_h(\os,\otheta)) \subset \tU(V_{h'}(\os_1,\otheta_1))$. We can
therefore apply Lemma~\ref{lem:tunique} with
$(s_0,\theta_0):=(\os,\otheta)$, $\tU_0:=\tU$, $\tU_1:=\tU(\cdot
+\os_1-\os,\cdot+\otheta_1-\otheta)$, and $h_0=r_1$. Since $\tU(s_0,\theta_0)\in\cU$, we obtain $\tU_0=\tU_1$. 

We can now get the desired contradiction as follows. We first note
that, by construction, we have $(\os_1,\otheta_1)\neq
(\os,\otheta)$. Assume first that $\otheta_1\neq \otheta$. Since
$\lim_{s\to-\infty}\tU(s,\theta)=(\og(\theta),(-\theta)\cdot
\olambda)$, we obtain from $\tU_0=\tU_1$ that
$\olambda=(\otheta-\otheta_1)\cdot\olambda$, a contradiction. 
Thus $\otheta_1=\otheta$, so that $\os_1\neq \os$. Then 
$$
\tU(s,\theta)=\lim_{k\to\pm\infty}\tU(s+k(\os_1-\os),\theta)=
(\og(\theta),(-\theta)\cdot\olambda),   
$$
so that $\p_s\tU\equiv (0,0)$, a contradiction again. The proof of
Step~2 is complete. 

\medskip
\noindent {\it Step~3. Assume there exists $R_0>0$ such that
  $du(s,\theta)$ is non-injective for all $s\ge R_0$ and $\theta\in
  S^1$. Assume that $\lim_{s\to\infty}
  (u(s,\theta),\lambda(s))=(\ug(\theta),\ulambda)$ and $\ug$ is
  non-constant. Then $\p_su\equiv 0$ on $[R_0,\infty[\,\times \, S^1$ and 
  $R(\tU)$ is dense in $[R_0,\infty[\,
  \times \, S^1$. }
\medskip 

Let us choose $R\ge R_0$ large enough so that
$X^\theta_{H_{\lambda(s)}}(u(s,\theta))\neq 0$ for all $s\ge R$ and
$\theta\in S^1$. This is possible since, by assumption, the
Hamiltonian $H$ is strongly admissible (see Definition~\ref{defi:strongH}).
As a consequence of the Floer equation for $u$, the vectors $\p_s u$
and $\p_\theta u$ cannot vanish simultaneously for $s\ge R$. 

We first show that there exist $\alpha,\beta\in\R$ such that 
\begin{equation} \label{eq:relation} 
\alpha \p_s u(s,\theta) + \beta \p_\theta u(s,\theta) =0
\end{equation} 
for all $(s,\theta)\in \, ]R,\infty[\, \times \, S^1$. By our
assumption on $du$, this relation holds for some choice of smooth
functions $\alpha,\beta:]R,\infty[\, \times \, S^1\to \R$ such that
$\alpha^2+\beta^2=1$.  Let us use
the shorthand notation $\tJ:=\tJ_\tlambda$, $X_\tH:=X_{\tH_\tlambda}$,
so that the Floer equation for $u$ writes $\p_su +\tJ\p_\theta u= \tJ
X_\tH$. We obtain 
$$
\p_s u = \beta^2 \tJ X_\tH -\alpha \beta X_\tH
$$
and
$$
\p_\theta u = -\alpha\beta \tJ X_\tH +\alpha^2 X_\tH. 
$$
Let us denote $a:=\beta^2$, $b:=-\alpha\beta$, so that $\p_s u=a\tJ
X_\tH+ bX_\tH$, $\p_\theta u = b\tJ X_\tH + (1-a) X_\tH$, and
$a^2+b^2-a=0$. From $[\p_s u, \p_\theta u]=0$ we obtain (see
also~\eqref{eq:tbracket}) 
$$
0=(\p_\theta a - \p_s b) \tJ X_\tH + (\p_s a + \p_\theta b) X_\tH.
$$
By our choice of $R>0$ we have $X_\tH\neq 0$, hence the linear
combination above must be trivial. The map $(b,a):]R,\infty[\,\times\,
S^1 \to \C$ is therefore holomorphic. On the other hand, its image
lies on the circle $a^2+b^2-a=0$, and this map must be constant. It
then follows that $\alpha$ and $\beta$ are constant as well. 

By assumption, the asymptote $\ug$ is nonconstant. This implies that
$\beta=0$, as seen by passing to the limit $s\to\infty$
in~\eqref{eq:relation}. Thus $\p_s u\equiv 0$ on $[R,\infty[\,\times\,
S^1$. 

We now prove that $\p_su\equiv 0$ on $[R_0,\infty[\,\times\, S^1$. Let
$\cI:=\{R\in [R_0,\infty[ \, : \, \p_su\equiv 0 \mbox{ on }
[R,\infty[\,\times \, S^1\}$. We just showed $\cI\neq \emptyset$, and
clearly $\cI$ is closed. On the other hand, the previous proof also
shows that $\cI$ is open (if $R\in \cI$, then $u([R,\infty[\,\times
S^1)=\ug(S^1)$, hence $u([R-\eps,\infty[\,\times S^1)$ is close to
$\gamma(S^1)$ for $\eps>0$ small enough, so that $X_H$ is nonzero
along the image of $u$ and one can apply the previous argument). 
Thus $\cI=[R_0,\infty[$.

We claim that $d\tU(s,\theta)$ is injective for all
$(s,\theta)\in[R_0,\infty[\,\times \, S^1$. Indeed, the component
$\lambda$ of $(u,\lambda)$ solves a time-independent ODE on
$[R_0,\infty[$ and, since $\p_s\tU\not\equiv 0$, we infer that
$\dot\lambda(s)\neq 0$ for all $s\ge R_0$. Since $\p_\theta u\neq 0$ on
$[R_0,\infty[\,\times\, S^1$, the claim follows. 

We finally claim that the set of injective points $R(\tU)$ is open and dense
in $[R_0,\infty[\,\times\, S^1$. Arguing by contradiction, we find a
non-empty open set $\Omega\subset [R_0,\infty]\,\times\, S^1$
consisting of non-injective points. Since $\p_\theta u\neq 0$ on
$\Omega$ and $d\tU$ is injective, it follows that the second and
third conditions in the definition of an injective point are
satisfied. Thus, the fact that points in $\Omega$ are non-injective
is equivalent to  
$$
\forall \, (s,\theta)\in \Omega, \qquad \exists\,
(s',\theta')\neq(s,\theta), \qquad \tU(s',\theta')=\tU(s,\theta).
$$
Arguing verbatim as in Step~2, we find (after possibly shrinking
$\Omega$) an open set $\Omega'\subset ]-\infty,R_0[\,\times \, S^1$, 
disjoint from $\Omega$, and a diffeomorphism
$\Phi:=(\phi,\psi):\Omega'\to \Omega$ such that $\tU|_{\Omega'}
= \tU|_{\Omega}\circ \Phi$ (see
diagram~\eqref{diag:Phi}). Substituting the relation
$u(s,\theta)=u(\phi(s,\theta),\psi(s,\theta))$ for all
$(s,\theta)\in\Omega'$ in the Floer equation~\eqref{eq:tFloer1} for
$u$, we obtain as in~\eqref{eq:tdu} 
\begin{eqnarray*}
0 & = & (\p_s\phi-\p_\theta\psi)\p_su(\Phi) + (\p_s\psi +
 \p_\theta\phi)\p_\theta u(\Phi) - \p_\theta\phi \,
 X_{\tH_{\tlambda(\Phi)}}(u(\Phi))  \\ 
 & &  -
 (1-\p_\theta\psi)\,  \tJ_{\tlambda(\Phi)}(u(\Phi))
 X_{\tH_{\tlambda(\Phi)}}(u(\Phi)). 
\end{eqnarray*}
Using that $\p_su=0$ and $\p_\theta u=X_{\tH_\tlambda}\neq 0$ on $\Omega$, we obtain 
$$
\p_s\psi=0,\qquad \p_\theta\psi=1. 
$$
 The same substitution in~\eqref{eq:tFloer3} for $\tlambda$ yields 
 \begin{eqnarray} 
 0 & = & \p_\theta\tlambda + X_{\tlambda} \nonumber \\
 & = & \p_s\tlambda(\Phi)\p_\theta\phi +
 \p_\theta\tlambda(\Phi)\p_\theta\psi + X_{\tlambda(\Phi)} \nonumber \\
 & = & \p_s\tlambda(\Phi)\p_\theta\phi + (\p_\theta\psi-1)
 (-X_{\tlambda(\Phi)}). \label{eq:tltheta}
 \end{eqnarray}  
 The third equality uses equation~\eqref{eq:tFloer3} for
 $\tlambda$. Since $\p_\theta \psi=1$ and $\p_s\tlambda\neq 0$ on
 $\Omega$, we obtain 
$$
\p_\theta\phi=0.
$$

Thus $\phi(s,\theta)=\phi(s)$ and $\psi(s,\theta)=\theta + \theta_0$,
$\theta_0\in S^1$ are actually defined on some open strip $I'\times S^1$
which intersects $\Omega'$. Let us denote $I:=\phi(I')$, so that we
have a diffeomorphism
$$
\overline \Phi =(\phi,\psi):I'\times S^1\to I\times S^1.
$$
We first observe that 
$$
\tlambda(\overline \Phi(s,\theta))=\tlambda(s,\theta),\qquad \forall
\, (s,\theta)\in I'\times S^1. 
$$
This follows from the fact that both $\tlambda\circ\overline \Phi$ and
$\tlambda$ solve the same ODE~\eqref{eq:tFloer3}, due to the special
form of $\overline \Phi$. We now claim that $u\circ\overline \Phi$ and
$u$ coincide on $I'\times S^1$. This follows from the unique
continuation property for the standard Floer equation, since
$u\circ\overline \Phi(s,\theta)=\ug(\theta+\theta_0)$ and 
therefore  
$$
\p_s(u\circ \Phi) + \tJ_\tlambda(u\circ\Phi)(\p_\theta
(u\circ\overline \Phi) - X_{\tH_\tlambda})=0
$$
on $I'\times S^1$. We thus obtained 
$$
\tU\circ \overline \Phi=\tU
$$
on $I'\times S^1$. Let now $s'_0\in I'$ and denote
$s_0:=\phi(s'_0)$. The maps $\tU$ and
$\tU(\cdot+s_0-s'_0,\cdot+\theta_0)$ coincide along $\{s'_0\}\times
S^1$ and solve~(\ref{eq:tFloer1}--\ref{eq:tFloer3}), hence by unique
continuation (Proposition~\ref{prop:tunique}) they coincide on
$\R\times S^1$. Arguing as in the last paragraph of Step~2, we obtain
a contradiction with our standing assumption $\p_s\tU\not \equiv
(0,0)$. This proves Step~3. 

\medskip 
\noindent {\it Step~4. Assume there exists $R_0>0$ such that
  $du(s,\theta)$ is non-injective for all $s\ge R_0$ and $\theta\in
  S^1$. Assume that $\lim_{s\to\infty}
  (u(s,\theta),\lambda(s))=(x,\ulambda)$ for some $x\in\widehat
  W$ (recall that, in this case, we have $h=K+f$ near $(x,\ulambda)$). Then: 
 \begin{itemize}
\item either $du\equiv 0$ on $[R_0,\infty[\,\times S^1$ and $R(\tU)$
  is dense in $[R_0,\infty[\, \times \, S^1$,  
\item or there exists $R\ge R_0$ such that $\p_\theta u\equiv 0$ and
  $u$ is a nonconstant gradient trajectory of $K$ on
  $[R,\infty[\,\times \, S^1$. In this case, $R(\tU)$ is dense in
  $[R,\infty[\, \times \, S^1$.   
 \end{itemize}
}
\medskip

By condition~\eqref{item:strongH2} in Definition~\ref{defi:strongH},
there exists $R\ge R_0$ such that, for $s\ge R$, the components $u$
and $\lambda$ solve the decoupled equations
\begin{eqnarray*}
\p_s u + J(u)(\p_\theta u - X_K(u)) & = & 0, \\
\dot\lambda - \nabla f(\lambda) & = & 0.
\end{eqnarray*}
The first equation implies that $C_R(u):=\{(s,\theta)\in ]R,\infty[\,\times\,
S^1 \, : \, \p_su(s,\theta)=0\}$ either coincides with
$]R,\infty[\,\times\, S^1$ or is discrete~\cite[Lemma~4.1]{FHS}. 
In the first case, we obtain $du\equiv 0$ on $]R,\infty[\,\times\,
S^1$. In the second case, the complement of $C_R(u)$ is connected. As in Step~3,
one then shows that there exist $\alpha,\beta\in\R$ such that
$\alpha^2+\beta^2=1$ and   
\begin{equation*} 
\alpha \p_s u(s,\theta) + \beta \p_\theta u(s,\theta) =0
\end{equation*} 
for all $(s,\theta)\in \, ]R,\infty[\, \times \, S^1$. 

If $\alpha\neq 0$, let us assume without loss of generality that
$\alpha>0$. Then $u(s,\theta)=u(s+\alpha t,\theta +\beta t)$
for all $t\ge 0$ and $(s,\theta)\in]R,\infty[\,\times\, S^1$. Letting
$t\to \infty$ we see that $u(s,\theta)=x$ and we again obtain
$du\equiv 0$.  
If $\alpha=0$, then $\p_\theta u\equiv 0$ and $u$ is a gradient
trajectory of $K$. 

We now prove the following: if $du\equiv 0$ on $[R,\infty[\,\times \,
S^1$ for some $R\ge R_0$, then the same holds on
$[R_0,\infty[\,\times\, S^1$. Arguing as in Step~3, we consider the
set $\cI:=\{R\in[R_0,\infty[\, : \, du\equiv 0 \mbox{ on }
[R,\infty[\, \times \, S^1\}$. Then $\cI\neq\emptyset$ and $\cI$ is
closed. On the other hand, $\cI$ is open: if $R\in \cI$, then
$u([R,\infty[\,\times\, S^1)=x$ and therefore
$u([R-\eps,\infty[\,\times\, S^1)$ belongs to a neighbourhood of $x$
where $H$ has the form $K+f$, provided $\eps>0$ is small enough. The
above argument shows that either $du\equiv 0$ on
$[R-\eps,\infty[\,\times\, S^1$, or $u$ is a nonconstant gradient
trajectory of $K$ on the same domain. The latter is impossible since
$du\equiv 0$ on $[R,\infty[\,\times \, S^1$. This shows
$\cI=[R_0,\infty[$. 

Let us refer to the case $du\equiv 0$ as Case~1, and to the case when $u$
is a non-constant gradient trajectory of $K$ as Case~2. We denote
$\tR:=R_0$ in Case~1, respectively $\tR:=R$ in Case~2. We now prove
that  
injective points are dense in $[\tR,\infty[\,\times\,S^1$. 

Let us first assume that $\dot\lambda\neq 0$ on $[\tR,\infty[$. Since
$X$ and $\dot\lambda$ are orthogonal (by $S^1$-invariance of $f$), we
infer that $d\tU$ is injective on $[\tR,\infty[\,\times\,
S^1$. Moreover, since $\lambda(s)\notin (S^1\cdot\ulambda)$ for $s\ge
\tR$, the third condition in the definition of an injective point is
satisfied. Arguing by contradiction as in Step~2 (using Baire's
theorem), we find open sets $\Omega'\subset ]-\infty,\tR[\,\times\,
S^1$ and $\Omega\subset [\tR,\infty[\,\times\, S^1$ and a
diffeomorphism $\Phi=(\phi,\psi):\Omega'\to \Omega$ such that
$\tU(s,\theta)=\tU(\Phi(s,\theta))$ for all
$(s,\theta)\in\Omega'$. Using~\eqref{eq:tltheta} in Step~3 we obtain
$\p_\theta\phi=0$ and $\p_\theta\psi=1$. Thus $\phi=\phi(s)$ and
$\psi=\theta+\overline \psi(s)$, and $\Phi$ admits an extension
$\overline \Phi=(\phi,\psi):I'\times S^1\to I\times S^1$ as in
Step~3. The same arguments as in Step~3 show that $\tU\circ\overline
\Phi=\tU$ on $I'\times S^1$. We then fix $s'_0\in I'$ and denote
$s_0:=\phi(s'_0)$. The maps $\tU$ and $\tU(\cdot+s_0-s'_0,\cdot +
\overline\psi(s'_0))$ coincide along $\{s'_0\}\times S^1$ and
solve~(\ref{eq:tFloer1}--\ref{eq:tFloer3}), hence by unique
continuation (Proposition~\ref{prop:tunique}) they coincide on
$\R\times S^1$. As in the final paragraph of Step~2, this contradicts
$\p_s\tU\not\equiv(0,0)$. 

We now assume that $\dot\lambda\equiv 0$ on $[R,\infty[$. Then $u$ is a
non-constant gradient trajectory of $K$ (Case~2). Under our
assumptions $d\tU$ is injective on $[R,\infty[\,\times\,
S^1$. Moreover, $u$ is not equal to $x$ on this domain and hence the
third condition in the definition of an injective point is
satisfied. Arguing as above and using the same notation, we find that 
$\tU(s,\theta)=\tU(\Phi(s,\theta))$ for all $(s,\theta)\in\Omega'$,
with $\Phi=(\phi,\psi)$. Using~\eqref{eq:tdu} we obtain $\p_s\phi=1$
and $\p_\theta\phi=0$, so that $\phi(s,\theta)=s+\os$ for some
$\os\in\R$. Using~\eqref{eq:tltheta} we obtain $\p_\theta\psi=1$, so
that $\psi(s,\theta)=\theta+\overline \psi(s)$. We conclude exactly as
above. This proves Step~4.
\end{proof}

\begin{proposition} \label{prop:duinjective}
Let $H\in\cH^{S^1}_{N,\mathrm{reg}}\cap\cH'$ and $(J,g)\in\cJ'(H)$. 
Let $(u,\tlambda):\R\times S^1\to \widehat W\times S^{2N+1}$ be a
solution of~(\ref{eq:tFloer1}--\ref{eq:tasymptotic}) satisfying
$\p_su\not\equiv(0,0)$. Assume one of the following holds:  
\begin{itemize}
\item one of the asymptotes of $u$ has a non-constant first component, 
\item both asymptotes have a constant first component and $u$ differs
  from a non-constant gradient trajectory in the neighbourhood of
  $-\infty$ or $+\infty$.  
\end{itemize}
Then there exists a nonempty open set $\Omega\subset \R\times S^1$
consisting of injective points and such that $du$ is injective on
$\Omega$.  
\end{proposition}

\begin{proof}
We distinguish several cases.

Assume there exists a sequence $(s^\nu,\theta^\nu)$ such that
$s^\nu\to\infty$, $\nu\to\infty$ and $du(s^\nu,\theta^\nu)$ is
injective. In this case, the claim follows from Step~2 in the proof of
Proposition~\ref{prop:tinjective}.  

Now assume there exists $R_0$ such that $du$ is non-injective on
$[R_0,\infty[\,\times S^1$, and the asymptote
$\ug:=\lim_{s\to\infty}u(s,\cdot)$ is non-constant. Let
$R_-:=\inf\{R\,:\, \p_su\equiv 0 \mbox{ on } [R,\infty[\,\times \,
S^1\}$. By Step~3 in Proposition~\ref{prop:tinjective}, we have
$R_-\le R_0$. Moreover, the assumption $\p_s u \not\equiv 0$ ensures
that $R_->-\infty$. Applying Step~3 again, we find a sequence
$(s^\nu,\theta^\nu)$ such that $s^\nu\to R_-$, $\nu\to\infty$ (with
$s^\nu<R_-$) and $du(s^\nu,\theta^\nu)$ is injective. Then the claim
follows from Step~2.  

Finally, assume there exists $R_0$ such that $du$ is non-injective on
the domain $[R_0,\infty[\,\times S^1$, and the asymptote
$\ug:=\lim_{s\to\infty}u(s,\cdot)$ is constant. By Step~4 in
Proposition~\ref{prop:tinjective} and our above assumption, we must
have $du\equiv 0$ on $[R_0,\infty[\,\times \, S^1$. Let
$R_-:=\inf\{R\,:\, du\equiv 0 \mbox{ on } [R,\infty[\,\times \,
S^1\}$. Then $R_-\le R_0$ and, because $\p_su\not\equiv 0$, we have
$R_->-\infty$. Applying Step~4 again, we find a sequence
$(s^\nu,\theta^\nu)$ such that $s^\nu\to R_-$, $\nu\to\infty$ (with
$s^\nu<R_-$), and  $du(s^\nu,\theta^\nu)$ is injective (if such a
sequence did not exist, we could find $\eps>0$ such that $du$ is
non-injective on $[R_- -\eps,\infty[\,\times \, S^1$, so that, by
Step~4, we either have $du\equiv 0$ on this domain and get a
contradiction with the definition of $R_-$, or $u$ is a non-constant
gradient trajectory and get a contradiction with our assumption). The
claim then follows from Step~2.  
\end{proof}

\begin{proof}[Proof of Theorem~\ref{thm:transvS1}.] 
We start by defining the neighbourhood of $\cH_*\cJ'\subset \cH\cJ'$
for which we will prove the theorem. Let us fix $(H_0,J_0,g_0)\in
\cH_*\cJ'$ and define which perturbations $(H,J,g)$ of $(H_0,J_0,g_0)$
are allowed.  If $H_0\in\cH_{\mathrm{gen}}$, then we allow any
$H\in\cH_{\mathrm{gen}}$ and any $(J,g)\in\cJ'(H)$. If
$H_0\in\cH_{\mathrm{split}}$, then the $S^1$-invariant metric $g$ on
$S^{2N+1}$ is allowed to be arbitrary. The pair $(H,J)$ is required to
be a perturbation of  
$(H_0,J_0)$ supported away from the constant orbits of $H$, and close
enough to $(H_0,J_0)$ such that the following two conditions hold:   
\begin{itemize}
\item for all $(\gamma_1,\lambda_1),(\gamma_2,\lambda_2)\in\cP(H)$
  such that $\gamma_1=\gamma_2$ and $\lambda_1\neq \lambda_2$, and for
  every solution $\lambda:\R\to S^{2N+1}$ of the equation 
\begin{equation} \label{eq:final}
 \dot \lambda = \int_{S^1}\vec\nabla_\lambda H(\theta,\gamma_1(\theta),\lambda)\, d\theta
\end{equation}
with $\lim_{s\to -\infty} \lambda(s)=\lambda_1$ and $\lim_{s\to
  \infty} \lambda(s)=\lambda_2$, there exists a nonempty open interval
$\cI\subset\R$ such that, for any $s\in \cI$ and $s'\in
\R\setminus\{s\}$, we have $\lambda(s')\notin S^1\cdot \lambda(s)$. 
\item for any $\op=(\og,\olambda),\up=(\ug,\ulambda)\in\cP(H)$ such
  that $\og\equiv\ox$, $\ug\equiv \ux$ are constant, and any
  $(u,\lambda)\in\widehat \cM(\op,\up;H,J,g)$ such that, near $\op$
  and $\up$, the components $u$, $\lambda$ are nonconstant gradient
  trajectories of $\oK$, $\of$, respectively $\uK$, $\uf$,   
we have 
\begin{equation} \label{eq:final2}
\lambda(\R)\cap (S^1\cdot \olambda) =\emptyset \quad \mbox{or} \quad
\lambda(\R)\cap (S^1\cdot \ulambda) =\emptyset. 
\end{equation}
\end{itemize}

The condition involving~\eqref{eq:final} is clearly satisfied for
$(H_0,g)$ with $g$ arbitrary. Hence it will still be satisfied for
small enough perturbations of $H_0$.   

The condition involving~\eqref{eq:final2} is also satisfied for the
pair $(H_0,J_0)$. Let us write $H_0=K_0+f_0$. By the maximum principle
and taking into account that constant orbits of $K_0$ are situated in
$W$, the trajectories involved in condition~\eqref{eq:final2} are
contained in $W$. Since $K_0$ is $C^2$-small on $W$, and $W$ is
symplectically aspherical, these must be gradient trajectories of
$K_0$~\cite{SZ}. Similarly, the $\lambda$-components are gradient
trajectories of $f_0$. Hence~\eqref{eq:final2} is satisfied due to
$S^1$-invariance of $f_0$.  As a consequence, it will still be
satisfied after a small perturbation of the pair $(H_0,J_0)$. 

\medskip 

Once the above neighbourhood of $\cH_*\cJ'$ has been defined, the proof 
is set up as for Theorem~\ref{thm:Jreg},
with obvious modifications dictated by $S^1$-invariance and the fact
that, in the split case, we only allow perturbations supported away
from the constant orbits.  
The main equation is \eqref{eq:ann1}, namely
\begin{eqnarray}
\int_{\R\times S^1} \Big\langle \eta,D_u\zeta + (D_\lambda J\cdot
\ell)J \partial _s u - J(D_\lambda X_{H_\lambda}\cdot \ell) -
JX_{h_\lambda} +
Y^\theta_\lambda J\partial _s u \Big\rangle \, dsd\theta 
\nonumber
\\
+\int_\R \Big\langle k,\nabla_s\ell -\nabla_\ell \int_{S^1} \vec
\nabla_\lambda H -  \int_{S^1} \nabla_\zeta \vec \nabla _\lambda H 
-\int_{S^1} \vec \nabla_\lambda h +
A\cdot \dot \lambda \Big\rangle \, ds =  0. \nonumber \\
\label{eq:tann1}
\end{eqnarray}
We must show that, if \eqref{eq:tann1} is satisfied for all $(\zeta, \ell,
h, Y, A) \in T_{(u,\lambda)} \cB \oplus T_h\cH^{r,S^1}_{N,\mathrm{reg}}
\oplus T_{(J,g)} \cJ^{r,S^1}_N$, then $(\eta, k) \equiv 0$.
Taking $h=0$, $Y=0$, $A=0$ we obtain that $(\eta,k)$ lies in the kernel of
the formal adjoint $D_{(u,\lambda)}^*$. The latter has the same form
as $D_{(u,\lambda)}$ and is therefore elliptic with smooth
coefficients. By elliptic regularity, it follows that $\eta$ and $k$
are smooth and the pair $(\eta, k)$ satisfies the unique continuation property. 
It is therefore enough to show that $(\eta, k)$ vanishes on a nonempty open set.
We distinguish now three cases. 

\medskip 

\noindent {\it Case 1.} $\p_s u \equiv 0$ and $\p_s \lambda \equiv
0$. In this case $(u,\lambda)\equiv (\gamma,\lambda_0)\in \cP(H)$. The
operator $D_{(u,\lambda)}$ is Fredholm of index $1$
(using the notation in the proof of Proposition~\ref{prop:indexMB},
the index is easily seen to differ by~$1$ from the index of the
operator $\widetilde D_{(u,\lambda)}$, which is equal to $0$
since $\widetilde D_{(u,\lambda)}$ is bijective). We must therefore
show that $D_{(u,\lambda)}$ has a $1$-dimensional kernel. Let $V\in
\ker\, D_{(u,\lambda)}$, and denote $V(s):=V(s,\cdot)\in
H^1(S^1,\gamma^* T\widehat W)\oplus T_{\lambda_0} S^{2N+1}$. Let
$V(s)^\perp$ be the $L^2$-orthogonal of $(\dot \gamma,-X)$ and
consider the asymptotic operator $D_{(\gamma,\lambda)}:
H^1(S^1,\gamma^* T\widehat W)\oplus T_{\lambda_0} S^{2N+1}\to
L^2(S^1,\gamma^* T\widehat W)\oplus T_{\lambda_0} S^{2N+1}$. In
suitable coordinates, we can write $D_{(u,\lambda)}=\p_s +
D_{(\gamma,\lambda_0)}$. Since $V\in \ker \, D_{(u,\lambda)}$, we have  
$$
(\p_s-D_{(\gamma,\lambda_0)})(\p_s+D_{(\gamma,\lambda_0)})V  =  \p_s^2 V -D_{(\gamma,\lambda_0)}^2V 
= 0.
$$
Taking $L^2$-scalar product with $V(s)^\perp$, using that
$D_{(\gamma,\lambda_0)}$ is self-adjoint, and using that $(\p_s
V)^\perp=\p_s(V^\perp)$, we obtain $\langle \p_s^2
V^\perp,V^\perp\rangle - \|D_{(\gamma,\lambda_0)}V^\perp\|^2=0$. By
assumption, the kernel of $D_{(\gamma,\lambda_0)}$ has dimension $1$
and is generated by $(\dot\gamma,-X)$. Hence there exists a constant
$c>0$ such that  
$$
\|D_{(\gamma,\lambda_0)} V(s)^\perp\|^2_{L^2}\ge c\|V(s)^\perp\|^2_{L^2},\qquad \forall \, s\in \R.
$$
As a consequence $\p_s^2\|V^\perp\|^2\ge 2\langle \p_s^2
V^\perp,V^\perp\rangle\ge 2c\|V^\perp\|^2$. Since $\|V^\perp\|\to 0$
as $s\to\pm\infty$, we infer by the maximum principle that
$V^\perp\equiv 0$. Thus $V(s)=a(s)(\dot \gamma,-X)$ and we obtain  
$$
0=D_{(u,\lambda)}V= a'(s)(\dot\gamma,-X) + a(s)D_{(\gamma,\lambda_0)}(\dot\gamma,-X)=a'(s)(\dot\gamma,-X),
$$
so that $a$ is constant. This proves that $\ker D_{(u,\lambda)}$ is generated by $(\dot\gamma,-X)$, as desired.

\medskip 

\noindent {\it Case~2.} $\p_s u\equiv 0$ and $\p_s\lambda\not\equiv 0$. 
By Steps~3 and~4 in Proposition~\ref{prop:tinjective}, the set of
injective points is open and dense in $\R\times S^1$. By
condition~\eqref{eq:final}, there exists a nonempty open set
$\Omega\subset \R\times S^1$ consisting of injective points such that
$\lambda(s')\notin S^1\cdot \lambda(s)$ for all $s'\neq s$ and all
$(s,\theta)\in\Omega$.  
Note that $\dot\lambda\neq 0$ and, up to further shrinking $\Omega$,
we can assume without loss of generality that $\tlambda$ is an
embedding on $\Omega$. 

We claim that $k(s)=0$ for all $(s,\theta)\in\Omega$. Arguing by
contradiction, we find $(s_0,\theta_0)\in\Omega$ such that $k(s_0)\neq
0$. We take $\zeta=0$, $\ell=0$, $Y=0$, $h=0$, and $A$ supported near
$S^1\cdot\lambda(s_0)$ and satisfying
$A(\lambda(s_0))\cdot\dot\lambda(s_0)=k(s_0)$. The first integral
in~\eqref{eq:tann1} vanishes, and the second integral is localized
near $s_0$ and is positive. This contradicts~\eqref{eq:tann1}.  

We now claim that $\eta\equiv 0$ on $\Omega$. If not, let
$(s_0,\theta_0)\in\Omega$ be such that $\eta(s_0,\theta_0)\neq 0$. Let
us consider a function $\th$ of the form
$\th(x,\lambda)=\phi(x)\psi(\lambda)$ such that $\psi$ is a cutoff
function supported near
$\tlambda(s_0,\theta_0)=(-\theta_0)\cdot\lambda(s_0)$, $\phi$ is
supported near $u(s_0,\theta_0)$ and satisfies  
$-\tJ_{\tlambda(s_0,\theta_0)}X_{\th_{\tlambda(s_0,\theta_0)}}(u(s_0,\theta_0))=\eta(s_0,\theta_0)$.
This determines uniquely an $S^1$-invariant function $h$ via
$h(\theta,x,\lambda)=\th(x,(-\theta)\cdot \lambda)$. We now remark
that, if the support of $\psi$ is small enough (depending on the
choice of $\phi$), we have  
$$
\langle \eta(s,\theta),-\tJ_{\tlambda(s,\theta)}X_{\th_{\tlambda(s,\theta)}}(u(s,\theta))\rangle \ge 0
$$
on $\R\times S^1$, and vanishes outside a small neighbourhood of
$(s_0,\theta_0)$. To see this, one uses that $(s_0,\theta_0)$ is an
injective point and that $\tlambda$ is an embedding on $\Omega$. We
now take $\zeta,\ell,Y,A$ to be zero, and $h$ as above. Then both
integrals in~\eqref{eq:tann1} are localized near
$(s_0,\theta_0)$. Since $k$ vanishes on $\Omega$, the second integral
vanishes, whereas the first one is positive. This
contradicts~\eqref{eq:tann1}.  

\medskip 

\noindent {\bf Remark.} The perturbation $h$ is admissible even if
$u\equiv x$ is a constant orbit. Indeed, in this case $\lambda$ is a
gradient trajectory of an $S^1$-invariant function on $S^{2N+1}$, so
that $\lambda(s_0)\notin S^1\cdot \ulambda$, with
$\ulambda:=\lim_{s\to\infty}\lambda(s)$. Thus, the Hamiltonian $H$
remains "split" in a neighbourhood of $\{x\}\times (S^1\cdot
\ulambda)$ under perturbations that are supported away from this set.  

\medskip 

\noindent {\it Case~3.} $\p_s u\not\equiv 0$. Let us first assume that
$u$ satisfies the assumptions of Proposition~\ref{prop:duinjective},
and let $\Omega\subset \R\times S^1$ be a nonempty open set consisting
of injective points and such that $du$ is injective on $\Omega$.  

We claim that $\eta\equiv 0$ on $\Omega$. If not, we can find
$(s_0,\theta_0)\in \Omega$ such that $\eta(s_0,\theta_0)\neq 0$.  
Moreover, we have $\p_s u(s_0,\theta_0)\neq 0$ by the definition of $\Omega$. 
Let $\tY:\widehat W\times S^{2N+1}\to \mathrm{End}(T\widehat W)$ be a
function supported near
$p_0:=(u(s_0,\theta_0),\tlambda(s_0,\theta_0))=(u(s_0,\theta_0),(-\theta_0)\cdot\lambda_0)$,
and which satisfies the relation
$\tY(p_0)J^{\theta_0}_\lambda(u(s_0,\theta_0))\p_s
u(s_0,\theta_0)=\eta(s_0,\theta_0)$. 
This uniquely determines an $S^1$-invariant function $Y$ via
$Y^\theta_\lambda(x):=\tY(x,(-\theta)\cdot \lambda)$. Taking
$\zeta=0$, $\ell=0$, $h=0$, $A=0$, and $Y$ as above,  
the first integral in~\eqref{eq:tann1} is localized near the injective
point $(s_0,\theta_0)$ and hence is positive, whereas the second
integral in~\eqref{eq:tann1} is zero. This
contradicts~\eqref{eq:tann1} and proves that $\eta\equiv 0$ on
$\Omega$.  

We now claim that $k(s)= 0$ for all $(s,\theta)\in \Omega$. Arguing by
contradiction, we find $(s_0,\theta_0)\in\Omega$ such that $k(s_0)\neq
0$. Let $\th:\widehat W\times S^{2N+1}\to \R$ be a function of the
form $\th(x,\lambda):=\phi(x)\psi(\lambda)$ such that $\phi$ is a
cutoff function near $u(s_0,\theta_0)$, and $\psi$ is supported near
$\tlambda(s_0,\theta_0)=(-\theta_0)\cdot\lambda(s_0)$ and satisfies
$\vec\nabla_\lambda\psi(\tlambda(s_0,\theta_0))=-k(s_0)$. This
uniquely determines an $S^1$-invariant function $h$ via
$h(\theta,x,\lambda):=\th(x,(-\theta)\cdot\lambda)$.  
The main observation is that, if the support of $\phi$ is small enough, then 
$$
\langle k(s),\int_{S^1}\vec\nabla_\lambda h(\theta,u(s,\theta),\lambda(s))\,d\theta\rangle\ge 0
$$
and vanishes outside a small neighbourhood of $(s_0,\theta_0)$. This
follows from the fact that $(s_0,\theta_0)$ is injective and the
assumption that $du(s_0,\theta_0)$ is injective, so that $u$ is an
embedding near $(s_0,\theta_0)$.  
Taking $\zeta=0$, $\ell=0$, $Y=0$, $A=0$, and $h$ as above, we see
that both integrals in~\eqref{eq:tann1} are localized near
$(s_0,\theta_0)$. Since $\eta$ was shown to vanish on $\Omega$, the
first integral vanishes, whereas the second integral is positive. This
contradicts~\eqref{eq:tann1} and proves the claim. 

\medskip 

We are now left to deal with the case when both asymptotes of $u$ are
constant and $u$ is a nonconstant gradient trajectory near
$\pm\infty$. It follows from Step~4 in
Proposition~\ref{prop:tinjective} that there exists $R>0$ large enough
such that $\p_su\neq 0$ on $\Omega:=(]-\infty,-R]\cup[R,\infty[)\times
S^1$ and the set of injective points is open and dense in this
domain. The same argument as above shows that $\eta\equiv 0$ on
$\Omega$.  

We claim that $k(s)=0$ for all $(s,\theta)\in\Omega$. If $\lambda$ is
constant near $-\infty$ or $+\infty$, the same construction as above
proves the claim. Let us therefore assume that $\lambda$ is a
nonconstant gradient trajectory near both $\pm\infty$.  

We now use that $(u,\lambda)$ satisfies~\eqref{eq:final2}, say at
$+\infty$. This implies that, for $s>0$ large enough, we have
$\lambda(\R\setminus\{s\})\cap S^1\cdot \lambda(s)=\emptyset$. Let us
choose an injective point $(s_0,\theta_0)$ with $s_0$ large enough,
such that $k(s_0)\neq 0$. Since $\dot\lambda(s_0)\neq 0$, we can
choose an $S^1$-invariant function $A$ supported in a neighbourhood of
$S^1\cdot\lambda(s_0)$ and satisfying $A(\lambda(s_0))\cdot
\dot\lambda(s_0)=k(s_0)$. The first integral in~\eqref{eq:tann1}
vanishes since $\eta$ was shown to be zero near $+\infty$, and the
second integral is localized near $s_0$ and is positive. This
contradicts~\eqref{eq:tann1} and finishes the proof. 
\end{proof}



\begin{thebibliography}{99} 
 
{\small 
 
 

\bibitem{Ar} N.~Aronszajn, A unique continuation theorem for solutions
of elliptic partial differential equations or inequalities of second
order, {\it J. Math. Pures Appl. (9)} {\bf 36} (1957), 235--249. 
 
 
 
\bibitem{BOauto} F.~Bourgeois, A.~Oancea, Symplectic homology, autonomous 
Hamiltonians, and Morse-Bott moduli spaces. {\it Duke Math. J.} {\bf
146}~:~1 (2009), 71--174.  
 
\bibitem{BOcont} ------, An exact sequence for contact- and symplectic 
homology. {\it Invent. Math.} {\bf 175}~:~3 (2009), 611--680.
 
\bibitem{BO3} ------, The Gysin exact sequence for $S^1$-equivariant
  symplectic homology. Preprint. 

\bibitem{BO4} ------, Linearized contact homology and
  $S^1$-equivariant symplectic homology. In preparation. 
 
 
 
\bibitem{B} H.~Br\'ezis, {\it Analyse fonctionnelle:~th\'eorie et 
    applications}, Masson, Paris, 1983. 
 
 
\bibitem{Ca} T.~Carleman, Sur les syst\`emes lin\'eaires aux
d\'eriv\'ees partielles de premier ordre \`a deux variables, {\it
C.~R.~Acad.~Sci.} {\bf 197} (1933), 471--474. 

 
 

\bibitem{CF} K.~Cieliebak, U.~Frauenfelder, A Floer homology for exact
  contact embeddings. {\it Pacific J. Math.} {\bf 239}~:~2 (2009), 251--316.
 
 
\bibitem{CO} K.~Cieliebak, A.~Oancea, Non-equivariant contact
  homology. In preparation. 

 
 
 
 
 
 
\bibitem{FHS} A.~Floer, H.~Hofer, and D.~Salamon, Transversality in 
   elliptic Morse theory for the symplectic action, {\it Duke Math. J.} 
   {\bf 80} (1995), 251--292. 
   
\bibitem{Gi} A.~Givental, Homological Geometry. I. Projective Hypersurfaces, {\it Selecta Math. (N.S.)} {\bf 1} (1995), 325--345.
 
 
 

\bibitem{Hu} M.~Hutchings, Floer homology for families. I. {\it
    Algebraic and Geometric Topology} {\bf 8} (2008), 435--492. 
 
 
 
 
 
\bibitem{McDS} D.~McDuff, D.~Salamon, {\it $J$-holomorphic curves and 
symplectic topology}, AMS Colloquium Publications, vol. 52, AMS, 
2004. 
 
 
 
 
\bibitem{OW} A.~Oancea, J.~Wehrheim, $G$-equivariant Floer
  homology. In preparation. 

 
\bibitem{RSspec} J.~Robbin, D.~Salamon,  
The spectral flow and the Maslov index, 
{\it Bull. London Math. Soc.} {\bf 27} (1995), 1--33.  
 
\bibitem{S} D.~Salamon, Lectures on Floer Homology, in {\it 
Symplectic Geometry and Topology}, Eds. Y.~Eliashberg and 
L.~Traynor. IAS/Park City Math. Series, vol. 7, AMS, 1999, 
pp. 143--229. 
 
\bibitem{SZ} D.~Salamon, E.~Zehnder, Morse theory for periodic 
   solutions of Hamiltonian systems and the Maslov index, {\it Comm. 
     Pure Appl. Math.} {\bf 45} (1992), 1303--1360. 
 
 
 
\bibitem{Sik} J.-C.~Sikorav, 
Some properties of holomorphic curves in almost complex manifolds, in {\it Holomorphic curves in symplectic geometry}, Eds. M.~Audin and J.~Lafontaine.  
Progr. Math., vol. 117, Birkh\"auser, Basel, 1994, pp. 165--189.  
 
\bibitem{V} C.~Viterbo, Functors and computations in Floer homology 
   with applications. I. 
{\it Geom. Funct. Anal.} {\bf 9} (1999), 985--1033. 
 
 
 
} 
 
 
\end{thebibliography}
\end{document}